\documentclass[]{amsart}
\usepackage{amsmath}
\usepackage{amssymb,amsfonts}
\usepackage{xypic}
\xyoption{all}

\newtheorem{theorem}{Theorem}[subsection]
\newtheorem{definition}[theorem]{Definition}

\newtheorem{corollary}[theorem]{Corollary}
\newtheorem{lemma}[theorem]{Lemma}
\newtheorem{remark}[theorem]{Remark}

\newtheorem{fact}[theorem]{Fact}

\newcommand{\T}{\ensuremath{\mathcal{T}}}
\def\B{\ensuremath{\mathcal{B}}}

\def\b{\ensuremath{\mathcal{B}}}
\def\c{\ensuremath{\mathcal{C}}}
\def\d{\ensuremath{\mathcal{D}}}

\def\v{\ensuremath{\mathrm{v}}}

\newcommand{\rep}{\ensuremath{\mathrm{Rep}}}
\newcommand{\urep}{\ensuremath{\mathrm{{ Rep}}_{\mathrm u}}}
\newcommand{\hrep}{\ensuremath{\mathrm{{ Rep}}_{\mathrm h}}}
\newcommand{\uhrep}{\ensuremath{\mathrm{{ Rep}}_{\mathrm{uh}}}}
\newcommand{\bhrep}{\ensuremath{\mathbf{{Rep}}_{\mathrm h}}}
\newcommand{\brep}{\ensuremath{\mathbf{Rep}}}
\newcommand{\burep}{\ensuremath{\mathbf{{Rep}}_{\mathrm u}}}
\newcommand{\buhrep}{\ensuremath{\mathbf{{Rep}}_{{\mathrm{uh}}}}}
\def\gner{\ensuremath{\Delta}}
\def\ner{\ensuremath{\mathrm{N}}}
\def\bgner{\ensuremath{\underline{\Delta}}}

\def\bsner{\ensuremath{{\mathrm{S}}}}
\def\set{\ensuremath{\mathbf{Set}}}

\newcommand{\bicat}{\ensuremath{\mathbf{Bicat}}}
\newcommand{\Hom}{\ensuremath{\mathbf{Hom}}}
\def\cat{\ensuremath{\mathbf{Cat}}}
\newcommand{\Top}{\ensuremath{\mathbf{Top}}}
\newcommand{\class}{\ensuremath{\mathrm{B}}}
\def\diag{\ensuremath{\mathrm{diag}}}
\newcommand{\f}{\ensuremath{\mathcal{F}}}
\newcommand{\func}{\ensuremath{\mathrm{Func}}}
\newcommand{\sset}{\ensuremath{\mathbf{Simpl.Set}}}
\newcommand{\g}{\ensuremath{\mathcal{G}}}
\newcommand{\h}{\ensuremath{\mathcal{H}}}

\begin{document}
\title{\em Geometric Realizations of Tricategories}
\author{Antonio M. Cegarra}\author{Benjam\'in A. Heredia}
\thanks{The authors would like to express their gratitude to Eduardo Pareja Tobes for the useful discussions we maintained during the realization of this paper.\newline
This work has been supported by DGI
of Spain, Projects: MTM200765431 and MTM2011-22554 and in the case of the second author, partly by Ministerio de Ciencia, Cultura
y Deporte of Spain, FPU grant AP2010-3521.}
\address{\newline
Departamento de \'Algebra, Facultad de Ciencias, Universidad de Granada.
\newline 18071 Granada, Spain \newline
acegarra@ugr.es\ baheredia@ugr.es}
\begin{abstract} The homotopy theory of higher categorical structures has become a
relevant  part of the machinery of algebraic topology and algebraic
K-theory. This paper contains some contributions to the study of
classifying spaces for tricategories, with applications to the
homotopy theory of categories, monoidal categories, bicategories,
braided monoidal categories, and monoidal bicategories. Any
tricategory characteristically has  associated various simplicial or
pseudo-simplicial objects. This paper explores the
 relationship amongst three of them: the
 pseudo-simplicial bicategory so-called {\em Grothendieck
 nerve} of the tricategory,
  the simplicial bicategory termed its {\em Segal nerve}, and
 the simplicial set called its {\em Street geometric nerve}, and it proves the fact that
the geometric realizations of all of these possible candidate
`nerves of the tricategory' are homotopy equivalent. Any one of
these realizations could therefore be taken as the classifying space
of the tricategory. Our results provide coherence for all reasonable
extensions to tricategories of Quillen's definition of the
'classifying space' of a category as the geometric realization of
the category's Grothendieck nerve. Many properties of the
classifying space construction for tricategories may be easier to
establish depending on the nerve used for realizations. For
instance, by using Grothendieck nerves we state and prove the
precise form in which the process of taking classifying spaces
transports tricategorical coherence to homotopy coherence. Segal
nerves allow us to obtain an extension to bicategories of the
results by Mac Lane, Stasheff, and Fiedorowicz about the relation
between loop spaces and monoidal or braided monoidal categories by
showing that the group completion of the classifying space of a
bicategory enriched with a monoidal structure is, in a precise way,
a loop space. With the use of geometric nerves, we obtain genuine
simplicial sets whose simplices have a pleasing geometrical
description in terms of the cells of the tricategory and,
furthermore,  we obtain an extension  of the results by Joyal,
Street, and Tierney  about the classification of braided categorical
groups and their relationship with connected, simply connected
homotopy 3-types, by showing that, via the classifying space
construction, bicategorical groups are a convenient algebraic model
for connected homotopy 3-types.
\end{abstract}
\maketitle
\section{Introduction and summary}
The process of taking classifying spaces of categorical structures
has shown its relevance as a tool in algebraic topology and
algebraic K-theory, and one of the main reasons is that the
classifying space constructions transport categorical coherence to
homotopic coherence. We can easily stress the historical relevance
of the construction of classifying spaces by recalling that Quillen
\cite{quillen} defines a higher algebraic $K$-theory by taking
homotopy groups of the classifying spaces of certain categories.
Joyal and Tierney \cite{j-t} have shown that Gray-groupoids are a
suitable framework for studying homotopy 3-types. Monoidal
categories were shown by Stasheff \cite{sta} to be algebraic models
for loop spaces, and work by May \cite{may79} and Segal
\cite{segal74} showed that classifying spaces of symmetric monoidal
categories provide the most noteworthy examples of spaces with the
extra structure required to define an $\Omega$-spectrum, a fact
exploited with great success in algebraic K-theory.

This paper contains some contributions to the study of classifying
spaces
 for tricategories  $\T=(\T,
\boldsymbol{a},\boldsymbol{l},\boldsymbol{r},
\pi,\mu,\lambda,\rho)$, introduced  by Gordon-Power-Street in
\cite{g-p-s}. Since our results find here direct applications to
monoidal categories, bicategories, braided monoidal categories, or
monoidal bicategories, the paper will quite possibly be of special
interest to $K$-theorists as well as to researchers interested in
 homotopy theory of higher
categorical structures, a theory with demonstrated relevance as a
tool for the treatment of an extensive list of subjects of
recognized mathematical interest in several mathematical contexts
beyond
 homotopy theory, such as algebraic geometry,  geometric structures on low-dimensional manifolds, string
field theory, quantum algebra, or topological quantum theory and
conformal field theory.

As for bicategories \cite{ccg},  there is a miscellaneous collection
of different `nerves' that have been (or might reasonably be)
characteristically associated to any tricategory $\T$. This paper
explores the
 relationship amongst three of these nerves: the
 pseudo-simplicial bicategory called the {\em Grothendieck nerve} $\ner\T=(\ner\T,\chi,\omega):\Delta^{\!\mathrm{op}}\
  \to \bicat$,
 the simplicial bicategory termed the {\em Segal nerve} $\bsner\T:\Delta^{\!\mathrm{op}}\ \to \bicat$, and
 the simplicial set called the {\em Street geometric nerve} $\gner\T:\Delta^{\!\mathrm{op}}\ \to \set$.
 Since, as we prove, these three nerve constructions
  lead to homotopy equivalent spaces, any one of these spaces could
  therefore be taken as {\em the}
classifying space $\class\T$ of the bicategory. Many properties of
the classifying space construction for tricategories,
 $\T\mapsto \class\T$, may
 be easier to establish depending on the nerve used for
 realizations. Here, both for historical reasons and for theoretical
 interest, it is appropriate to start with the Grothendieck nerve
 construction to introduce $\class\T$. Let us briefly recall that it was
Grothendieck \cite{groth} who first associated a simplicial set
$\ner C$ to a small category $C$, calling it its nerve, whose
$p$-simplices are composable $p$-tuples $x_p\to\cdots\to x_0$ of
morphisms in $ C$.
 Geometric realization of its nerve
is the classifying space of the category, $\class C=|\ner C|$.
Hence, our first relevant result in the paper shows how
 Grothendieck nerve construction for
 categories rises to tricategories, that is, is to prove
 that
\begin{quote}{\em ``Any tricategory  $\T$ defines a normal pseudo-simplicial
bicategory
$$\ner\T=(\ner\T,\chi,\omega):\Delta^{\!\mathrm{op}}\ \to \bicat,$$
whose bicategory of $p$-simplices consists of $p$-tuples of
composable cells,}$$\ner_{\!p}\!\mathcal{T} =
\bigsqcup_{(x_0,\ldots,x_p)\in \mbox{\scriptsize
Ob}\!\mathcal{T}^{p+1}}\hspace{-0.3cm}
\mathcal{T}(x_1,x_0)\times\mathcal{T}(x_2,x_1)\times\cdots\times\mathcal{T}(x_p,x_{p-1}).$$
{\em  If $[q]\overset{a}\to [p]$ is any map in the simplicial
category $\Delta$, then the associated homomorphism
$a^*:\ner_{\!p}\!\T \to\ner_{\!q}\!\T$ is induced by the unit $1\to
\mathcal{T}(x,x)$ and composition ${\mathcal{T}(y,z)\times
\mathcal{T}(x,y)\to \mathcal{T}(x,z)}$ homomorphisms. The
structure pseudo-equivalences $\chi_{\!a,b}:b^*
a^*\Rightarrow(ab)^*$, for
 maps $[n]\overset{b}\to [q]\overset{a}\to
[p]$,  canonically arise  from the structure pseudo-equivalences
$\boldsymbol{a}, \boldsymbol{l}$, and $\boldsymbol{r}$, while the
structure invertible modifications
$\chi_{\!a,bc}\circ
\chi_{b,c}\,a^*\Rrightarrow\chi_{\!ab,c}\circ c^*\chi_{\!a,b}$, for
 maps $[m]\overset{c}\to [n]\overset{b}\to
[q] \overset{a}\to [p]$, are canonically provided by the
modifications data $\pi,\mu, \lambda,\rho$ of the tricategory."}
\end{quote}

Then, heavily dependent on the results by Carrasco-Cegarra-Garz\'on
\cite{c-c-g}, where an analysis of classifying spaces is performed
for lax diagrams of bicategories following the way Segal
\cite{segal74} and Thomason \cite{thomason} analyzed lax diagrams of
categories, we introduce
 {\em the classifying space $\class\T$},
of a tricategory $\T$, to be the classifying space of its
Grothendieck nerve $\ner\T$. Briefly, say that the so-called
Grothendieck construction \cite[\S 3.1]{c-c-g} on the
pseudo-simplicial bicategory $\ner\T$ produces a bicategory
$\int_{\!\Delta}\!\!\ner\T$, whose objects are the pairs $(x,p)$,
where $p\geq 0$ is an integer and $x=(x_p\to\cdots\to x_0)$ is an
object of the bicategory $\ner_{\!p}\!\T$,  whose hom-categories are
$$\xymatrix{\int_{\!\Delta}\!\!\ner\T\big((y,q),(x,p)\big)=
\bigsqcup\limits_{[q]\overset{a}\to [p]}\ner_{\!q}\!\T(y,a^*x),}$$
 and whose compositions, identities,  and
structure constraints are defined naturally. Again, Grothendieck
nerve construction on this bicategory, $\int_{\!\Delta}\!\!\ner\T$,
now gives rise to a normal pseudo-simplicial category
$\ner(\int_{\!\Delta}\!\!\ner\T):\Delta^{\!\mathrm{op}}\ \to \cat$,
on which once more Grothendieck construction leads to a category,
$\int_{\!\Delta}\!\!\ner(\int_{\!\Delta}\!\!\ner\T)$, whose
classifying space is, by definition, the classifying space of the
tricategory, that is,
$$\xymatrix{\class\T=
|\ner(\int_{\!\Delta}\!\!\ner(\int_{\!\Delta}\!\!\ner\T))|.}$$

The precise behavior of this classifying space construction,
$\T\mapsto \class\T$,  can be summarized as follows:
\begin{quote}
{\em  ``- Any trihomomorphism between tricategories $F:\T\to\T'$
induces a continuous map $\class F:\class\T\to\class\T'$.

\noindent - If $F,G:\T\to\T'$ are two trihomomorphisms between
tricategories, then any tritransformation, $F\Rightarrow G$,
canonically defines a homotopy between the induced maps on
classifying spaces, $ \class F\simeq \class G:\class\T\to\class\T'$.

\noindent - For any tricategory $\T$, there is a homotopy $\class
1_\T\simeq 1_{\class \T}:\class\T\to\class\T$ and, for any
composable trihomomorphisms $H:\T\to \T'$ and $H':\T'\to \T''$,
there is a homotopy $\class H'\ \class H\simeq \class
(H'H):\class\T\to\class T''$.

\noindent - Any triequivalence of tricategories $\T\to \T'$ induces
a homotopy equivalence on classifying spaces $\class\T\simeq
\class\T'$."}
\end{quote}

For instance, we provide a positive answer (long assumed) to the
question whether, as a consequence of the coherence theorem for
tricategories by Gordon-Power-Street \cite{g-p-s}, every tricategory
is `homotopy equivalent' to a Gray-category. More precisely, the
coherence theorem  states that, for any tricategory $\T$, there is a
Gray-category $\mathrm{G}(\T)$ with a triequivalence $\T\to
\mathrm{G}(\T)$. Then, it is a result that  \begin{quote}{\em
``There is an induced homotopy equivalence $\class\T\simeq \class
\mathrm{G}(\T).$"}
\end{quote}

To deal with the delooping properties of certain classifying spaces,
for any tricategory $\T$,  we introduce  its Segal nerve
 $\bsner\T$. This is a simplicial bicategory whose bicategory of $p$-simplices is the bicategory
 of {\em unitary homomorphic representations} of the ordinal category $[p]$ in $\T$ (roughly speaking,
 trihomomorphisms $[p] \to \T$ satisfying various requirements of normality).
 Each $\bsner\T$ is a `special' simplicial
bicategory, in the sense that the Segal projection homomorphisms on
it are biequivalences of bicategories, and thus it is a weak
3-category from the standpoint of Tamsamani \cite{tam} and Simpson
\cite{sim}. When $\T$ is a reduced tricategory (i.e., with only one
object), then the simplicial space
$\class\bsner\T:\Delta^{\!{\mathrm{op}}}\to\Top$, obtained by
replacing the bicategories $\bsner_p\T$ by their classifying spaces
$\class\bsner_p\T$, is a special simplicial space. Therefore,
according to Segal \cite{segal74}, $\Omega|\class\bsner\T|$ is a
group completion of $ \class\bsner_1\T$. In our development here, a
relevant result is the following:
\begin{quote} {\em ``For any tricategory $\T$, there is a
 homotopy equivalence $\class\T\simeq |\class\bsner\T|$,"
} \end{quote} which we apply  to the study of classifying spaces of
monoidal bicategories. Recall the aforementioned fact by Stasheff
and Mac Lane, that the group completion of the classifying space of
a category enriched with a monoidal structure is, in a precise way,
a loop space. Any monoidal bicategory $(\b,\otimes)$ gives rise to a
one-object tricategory $\Sigma(\b,\otimes)$, its `suspension'
tricategory following Street's terminology (or `delooping' in the
terminology of Kapranov-Voevodsky \cite{KV94b} or Berger
\cite{berger}). Defining the classifying space of a monoidal
bicategory $(\b,\otimes)$ to be the classifying space of its
suspension tricategory, that is,
$\class(\b,\otimes)=\class\Sigma(\b,\otimes)$, we mainly prove the
following extension to bicategories of Stasheff's result on monoidal
categories:
\begin{quote}{\em ``For any monoidal bicategory $(\b,\otimes)$, the loop space of its classifying
space, $\Omega\class(\b,\otimes)$, is a group completion of the
classifiying space of the underlying bicategory, $\class\b$. In
particular, if the monoid  of connected components $\pi_0\b$ is a
group, then there is a homotopy equivalence $\class \b\simeq
 \Omega\class(\b,\otimes)$."}
 \end{quote}

 If $(C,\otimes,\boldsymbol{c})$ any braided monoidal
category, then, thanks to the braiding, the suspension of the
underlying monoidal category $\Sigma(C,\otimes)$, which is actually
a bicategory,  has a structure  of monoidal bicategory. Hence, the
double suspension tricategory
$\Sigma^2(C,\otimes,\boldsymbol{c})=\Sigma(\Sigma(C,\otimes),\otimes)$
is defined. Since the classifying space of the braided monoidal
category is just the classifying space of its double suspension
tricategory, that is,
$\class(C,\otimes,\boldsymbol{c})=\class\Sigma^2(C,\otimes,\boldsymbol{c})$,
the above result implies the existence of a homotopy equivalence
$\Omega\class(C,\otimes,\boldsymbol{c})\simeq\class(C,\otimes)$,
between the loop space of classifying space of the braided monoidal
category and the classifying space of the underlying monoidal
category, a fact recently proved by Carrasco-Cegarra-Garz\'on
\cite{c-c-g}. Further, we conclude that
\begin{quote}{\em ``For any braided monoidal category $(C,\otimes,\boldsymbol{c})$, the double
loop space $\Omega^2\class(C,\otimes,\boldsymbol{c})$ is a group
completion of $\class C$,"} \end{quote} and thus a new proof that
the group completion of the underlying category of a braided
monoidal category is a double loop space, as was noted by Stasheff
\cite{sta} but originally proven by Fiedorowicz \cite{fie} (see also
Berger \cite{berger} and Baltenau-Fiedorowicz-Schw\"{a}nzl-Vogt
 \cite{b-f-s-v}). Let us
stress that, just because of this double-loop property, braided
monoidal categories have been playing a key role in recent
developments in quantum theory and related topics.

The process followed for defining the classifying space of a
tricategory $\T$, by means of its Grothendieck nerve $\ner\T$, is
quite indirect and  the CW-complex $\class\T$ thus obtained has
little apparent intuitive connection with the cells of the original
tricategory. However, when $\T$ is a (strict) $3$-category, then the
space $|\diag\ner\ner\ner\T|$, the geometric realization of the
simplicial set diagonal of the $3$-simplicial set $3$-fold nerve of
$\T$, has usually been taken in the literature  as the `correct'
classifying space of the $3$-category. A result in the paper states
that
\begin{quote}
{\em `` For any $3$-category $\T$, there is homotopy equivalence
$\class\T\simeq |\diag\ner\ner\ner\T|$."}
\end{quote}

The construction of the simplicial set $\diag\ner\ner\ner\T$ for
3-categories does not work in the non-strict case since the
compositions in arbitrary tricategories are not associative and not
unitary, which is crucial for the 3-simplicial structure of the
triple nerve  $\ner\ner\ner\T$, but only up to coherent equivalences
or isomorphisms. There is, however, another convincing way of
associating a simplicial set to a $3$-category $\T$ through  its
{\em geometric nerve} $\gner\T$, thanks to Street \cite{street2}. He
extends each ordinal $[p]=\{0<1<\cdots<p\}$ to a $p$-category
$\mathcal{O}_p$, the {\em $p^{\mathrm{th}}$-oriental}, such that the
$p$-simplices of $\gner\T$ are just the $3$-functors
$\mathcal{O}_p\to \T$.  Thus, $\gner\T$ is a simplicial set whose
0-simplices are the objects (0-cells) $ F_0$ of $\T$, whose
1-simplices are the 1-cells $F_{01}: F_1\to F_0$, whose 2-simplices
$$
\xymatrix@C=10pt@R=10pt{& F_ 2\ar[ld]_{ F_{12}}\ar[rd]^{ F_{02}}
\ar@{}[d]|(0.58){ \overset{F_{012_{}}}\Rightarrow}& \\  F_ 1\ar[rr]_{ F_{01}}&& F_ 0,}
$$
consist of two composable 1-cells and a 2-cell  $ F_{012}:
F_{01}\otimes  F_{12}\Rightarrow  F_{02}$, and so on. For instance,
if $\T=\Sigma^2A$ is the 3-groupoid having only one $i$-cell for
$0\leq  i \leq 2$ and whose $3$-cells are the elements of an abelian
group $A$, with all the compositions given by addition in $A$, then
$\gner\Sigma^2A=K(A,3)$, the Eilenberg-Mac Lane minimal complex. In
fact, as for the strict case, the geometric nerve construction
$\gner\T$ even works for arbitrary tricategories $\T$, as Duskin
\cite{duskin} and Street \cite{street3} pointed out, and we discuss
here in detail.  The geometric nerve $\gner\T$  is defined to be the
simplicial set  whose $p$-simplices are {\em unitary
representations} of the ordinal category $[p]$ in the tricategory
$\T$ (roughly speaking,
 lax functors $[p] \to \T$ satisfying various requirements of normality).
This is a simplicial set which completely encodes all the structure
of the tricategory and, furthermore, the cells of its geometric
realization $|\gner\T|$ have a pleasing geometrical description in
terms of the cells of $\T$. As a main result in the paper, we state
and prove that
\begin{quote}{\em ``For any tricategory $\T$, there is a
 homotopy equivalence $\class\T\simeq  |\Delta\T|$."}
\end{quote}

If $(\b,\otimes)$ is any monoidal bicategory, then its geometric
nerve, $\gner(\b,\otimes)$, is defined to be the geometric nerve of
its suspension tricategory $\Sigma(\b,\otimes)$.  Then, we obtain
the following:
\begin{quote} {\em ``For any monoidal bicategory $(\b,\otimes)$, there is a
homotopy equivalence $$\class(\b,\otimes)\simeq
|\gner(\b,\otimes)|."$$ }
\end{quote}
For instance,  since the geometric nerve of a braided monoidal
category $(C,\otimes,\boldsymbol{c})$ is the geometric nerve of its
double suspension tricategory, that is,
$\gner(C,\otimes,\boldsymbol{c})=\gner\Sigma^2(C,\otimes,\boldsymbol{c})$,
 the existence of a  homotopy equivalence
 $\class(C,\otimes,\boldsymbol{c})\simeq |\gner(C,\otimes,\boldsymbol{c})|$ follows, a fact  proved
 in
\cite{c-c-g}.

The geometric nerve $\gner(\b,\otimes)$, of any given monoidal
bicategory $(\b,\otimes)$,  is a Kan complex if and only if
$(\b,\otimes)$ is a {\em bicategorical group} (or weak  $3$-group,
or Gr-bicategory),  that is,  a monoidal bicategory whose 2-cells
are isomorphisms,
 whose 1-cells are equivalences, and each object $x$ has a quasi-inverse with respect to the tensor product.
In other words, a bicategorical group is a monoidal bicategory whose
suspension tricategory $\Sigma(\b,\otimes)$ is a {\em trigroupoid}
(or {\em Azumaya tricategory} in the terminology of
Gordon-Power-Street \cite{g-p-s}). The geometric nerve of any
bicategorical group $(\b,\otimes)$ is then a one vertex Kan complex, whose homotopy groups can be described
using only the  algebraic structure of $(\b,\otimes)$ by
\begin{itemize}
\item[-] $\pi_i\gner(\b,\otimes)= 0$,\ \, if $i \neq 1,2,3$.
\vspace{0.1cm}
\item[-] $\pi_1\gner(\b,\otimes)=\mathrm{Ob}\b/\!\thicksim$, \ the group of equivalence classes of objects in $\b$ where multiplication is induced by the tensor product.
    \vspace{0.1cm}
\item[-] $\pi_2\gner(\b,\otimes)=\mathrm{Aut}_\b(1)/\!\cong$, \ the group of isomorphism classes of autoequivalences of the unit object where the operation is induced by the horizontal composition in $\b$.
    \vspace{0.1cm}
\item[-]  $\pi_3\gner(\b,\otimes)=\mathrm{Aut}_\b(1_1)$,\  the group of automorphisms
of the identity 1-cell of the unit object where the operation is
vertical composition in $\b$.
\end{itemize}
Hence, its classifying space $\class(\b,\otimes)$ is a
(path) connected homotopy 3-type. In fact, every connected
homotopy 3-type can be realized in this way from a bicategorical
group, as suggested by the unpublished but widely publicized  result
of Joyal and Tierney \cite{j-t} that Gray-groups (called semistrict
3-groups by Baez-Neuchl \cite{b-m}) model connected homotopy
3-types (see also Berger \cite{berger}, Lack \cite{lack2}, or Leroy
\cite{leroy}). Recall that, by the coherence
theorem for tricategories, every bicategorical group is monoidal
biequivalent to a Gray-group. In the last example in
the paper, we outline in some detail the proof of the following
statement:
\begin{quote}
{\em ``For any connected CW-complex $X$ with $\pi_iX=0$ for
$i\geq 4$, there is a bicategorical group $(\b,\otimes)$ with a
homotopy equivalence $\class(\b,\otimes)\simeq X.$"}
\end{quote}
In the proof of this result, we explicitly show how to construct
from any connected minimal complex $M$ with $\pi_iM=0$ for
$i\neq 1,2,3$, a  bicategorical group $(\b,\otimes)$ with a
simplicial isomorphism $\gner(\b,\otimes)\cong M$. In the particular
case when, in addition, $\pi_3M=0$, then the resulting bicategorical
group has all its 2-cells identities, that is, it is actually a {\em
categorical group} $(C,\otimes)$ \cite{joyal}. While in the
particular case where $\pi_1M=0$, the bicategorical group has only
one object, so that it is actually the suspension of a {\em braided
categorical group} $(C,\otimes,\boldsymbol{c})$ \cite{joyal}. Hence,
our proof implicitly covers two relevant particular cases, already
well-known from Joyal-Tierney \cite{j-t} (see also \cite{c-c}), one stating that categorical
groups are a convenient algebraic model for connected homotopy
2-types, and the other that braided categorical groups are algebraic
models for connected, simply-connected homotopy 3-types, namely:
\begin{quote}
{\em  ``For any path-connected CW-complex $X$, there is categorical
group $(C,\otimes)$ and a homotopy equivalence
$\class(C,\otimes)\simeq X$ if and only if $\pi_iX=0$ for $i\geq
3$."}
\end{quote}
\begin{quote}
{\em  ``For any path-connected CW-complex $X$, there is a braided
categorical group $(C,\otimes,\boldsymbol{c})$ and a homotopy
equivalence $\class(C,\otimes,\boldsymbol{c})\simeq X$ if and only
if $\pi_iX=0$ for $i=1$ and $i\geq 3$."}
\end{quote}

\subsection{The organization of the paper} The plan of this paper is, briefly, as follows. After this
introductory Section 1, the paper is organized in four sections.
Section 2 is long and very technical, but crucial to our
discussions. It is dedicated to establishing and proving some needed
results concerning the notion of  representation of a category in a
tricategory, which is at the heart of the several constructions of
nerves for tricategories used in the paper. In Section 3, we mainly
include the construction of the Grothendieck nerve
$\ner\T:\Delta^{\!{\mathrm{op}}}\ \to \bicat$, for any tricategory
$\T$, and the  study of the basic properties concerning the behavior
of the Grothendieck nerve construction, $\T\mapsto \ner\T$, with
respect to trihomomorphisms of tricategories and tritransformations
between them. Section 4 contains the definition of classifying space
$\class\T$, for any tricategory $\T$. The main facts concerning the
classifying space construction $\T\mapsto \class\T$ are established
here. In this section we also study the relationship between
$\class\T$ and the space realization of the Segal nerve of a
tricategory, $\bsner\T:\Delta^{\!\mathrm{op}}\to\bicat$, which,  for
instance, we apply to  prove that the classifying space of any
monoidal bicategory is up to group completion a loop space. Finally,
Section 5 is mainly dedicated to describing  the geometric nerve
$\gner\T:\Delta^{\!{\mathrm{op}}}\ \to \set$, of any tricategory
$\T$, and to proving the existence of homotopy equivalences
$\class\T\simeq |\gner\T|$. Also, by means of the geometric nerve
construction for monoidal bicategories, $(\b,\otimes)\mapsto
\gner(\b,\otimes)$, we show here that bicategorical groups are a
convenient algebraic model for connected homotopy 3-types.

\subsection{Notations} We refer to Benabou \cite{benabou} and
Street \cite{street} for the background on bicategories. For any
bicategory $\b$, the composition
  in each hom-category $\b(x,y)$, that is, the vertical composition of 2-cells,
 is denoted by $v\cdot u$, while the symbol $\circ$ is used to denote the horizontal composition
 functors $\circ:\b(y,z)\times\b(x,y)\to\b(x,z)$.
The identity of an object is written as $1_x:x\to x$, and we shall
use the letters $\boldsymbol{a}$, $\boldsymbol{r}$, and
$\boldsymbol{l}$ to denote the associativity, right unit, and left
unit constraints of the bicategory, respectively. Hereinafter
\cite[Notation 4.9]{g-p-s}, the category of bicategories with
homomorphisms (or pseudo-functors) between them will be denoted by
$\mathbf{ Hom}$.

      In this paper we use the notion of  tricategory $\T=(\T,
\boldsymbol{a},\boldsymbol{l},\boldsymbol{r}, \pi,\mu,\lambda,\rho)$
as it was introduced  by Gordon, Power, and Street in \cite{g-p-s},
but with a minor alteration:  we require that the homomorphisms of
bicategories picking out units are normalized, and then written
simply as $1_t\in\T(t,t)$. This restriction is not substantive (see
Gruski \cite[Theorem 4.3.3]{gurski}), but it does slightly reduce
the amount of coherence data we have to deal with. For any object
$t$ of the tricategory $\T$, the arrow $\boldsymbol{r}1_t:1_t\to
1_t\otimes 1_t$ is an equivalence in the hom-bicategory $\T(t,t)$,
with the arrow $\boldsymbol{l}1_t:1_t\otimes 1_t\to 1_t$ an adjoint
quasi-inverse (see \cite[Lemma A.3.1]{gurski}). Hereafter, we
suppose the adjoint quasi-inverse of $\boldsymbol{r}$,
$\boldsymbol{r}^\bullet\dashv \boldsymbol{r}$, has been chosen such
that $\boldsymbol{r}^\bullet 1_t =\boldsymbol{l}1_t$.

As in \cite[\S 5]{g-p-s} and \cite[\S 6.3]{gurski}, $\mathbf{Bicat}$
denotes the tricategory of small bicategories, homomorphisms,
pseudo-transformations and modifications.  If $F,F':\b \to \c$ are
lax functors between bicategories, then we follow the convention of
\cite{g-p-s} in what is meant by a {\em lax transformation}
${\alpha:F\Rightarrow F'}$. Thus, $\alpha$ consists of morphisms
${\alpha x\!:Fx\to F'x}$, and of 2-cells $\alpha_{a}:\alpha y\circ
Fa\Rightarrow F'a\circ \alpha x$, subject to the usual axioms. In
the structure of $\bicat$ we use, the composition of
pseudo-transformations is taken to be
$$\big(\xymatrix@C=8pt{\c \ar@/^0.6pc/[rr]^{G}\ar@{}[rr]|{\Downarrow \beta}
\ar@/_0.6pc/[rr]_{ G'} &  &\d}\big)\big(\xymatrix@C=8pt{\b
\ar@/^0.6pc/[rr]^{F}\ar@{}[rr]|{\Downarrow \alpha}
\ar@/_0.6pc/[rr]_{ F'} &  &\c}\big)=\big(\xymatrix@C=9pt{\b
\ar@/^0.7pc/[rr]^{GF}\ar@{}[rr]|{\Downarrow \beta\alpha}
\ar@/_0.7pc/[rr]_{ G'\!F'} &  &\d}\big), $$ where  $\beta
\alpha=\beta F'\circ G\alpha:GF\Rightarrow G'F\Rightarrow G'\!F'$,
but note the existence of the useful invertible modification
\begin{equation}\label{4}\begin{array}{l}\xymatrix@R=15pt@C=20pt{GF\ar@{}@<-4pt>[rd]^(.5){\Rrightarrow}
\ar@{=>}[r]^{\beta F}\ar@{=>}[d]_{G\alpha}&G'F
\ar@{=>}[d]^{G'\beta}\\ GF'\ar@{=>}[r]^{\beta
F'}&G'F'}\end{array}\end{equation} whose component  at an object $x$
of $\b$, is $\beta_{\alpha x}$, the component of $\beta$ at the
morphism $\alpha x$.

For the general background on simplicial sets,  we mainly refer  to
\cite{g-j}. The {\em simplicial category} is denoted by $\Delta$,
and its objects, that is, the ordered sets ${[n]=\{0,1,\dots,n\}}$,
are usually considered as categories with only one morphism
$(i,j):j\rightarrow i$ when $0\leq i\leq j\leq n$. Then, a
non-decreasing map $[n]\rightarrow [m]$ is the same as a functor, so
that we see $\Delta $, the simplicial category of finite ordinal
numbers, as a full subcategory of $\cat$, the category (actually the
2-category) of small categories.

Throughout the paper, Segal's {\em geometric realization}
\cite{segal68} of a simplicial (compactly generated topological)
space $X:\Delta^{\!{^\mathrm{op}}}\to  \mathbf{Top}$ is denoted by
$|X|$. By regarding a set as a discrete space, the (Milnor's)
geometric realization of a simplicial set
$X:\Delta^{\!{^\mathrm{op}}}\to \set$ is $|X|$.

\section{Representations of categories in
tricategories}\label{secrep}

As we will show in the paper, the classifying space of any
tricategory
 can be realized up to homotopy by a simplicial set
$\Delta\T$, whose $p$-simplices $\Delta[p]\to \Delta\T$ are the
different {\em unitary representations} $[p]\to \T$, of the category
$[p]$ in the tricategory $\T$. Hence, we present below a  study
of these representations of categories in tricategories, and the
bicategories they form.
\subsection{(Unitary, homomorphic) Representations}
Roughly speaking, a representation of a category
$I$ in a tricategory $\T$ is a lax functor $I \to \T$, where $I$ is
regarded as tricategory in which the 2-cells and 3-cells are all
identities, satisfying various requirements of normality. But
noting that the definition given in \cite[Definition 3.1]{g-p-s}
for trihomomorphisms does not work for lax functors between
tricategories (the 3-cell $\bar{\pi}$ in the equation expressing
axiom (HTA1) is not well defined), we establish the following
explicit definition:

\begin{definition}\label{repr} A {\em representation} $F:I\to \T$, of a category
$I$ in a tricategory $\T$, is a system of data consisting of: an
object {\em $Fi\in \mbox{Ob}\T$}, for each object $i$ in $I$; a
$1$-cell $Fa: Fj\to Fi$, for each arrow $a:j\to i$ in $I$; $2$-cells
$F_{a,b}:Fa\otimes Fb\Rightarrow F(ab)$ and $F_i:1_{Fi}\Rightarrow
F1_i$, for each pair of composable arrows $k\overset{ b}\to j
\overset{a}\rightarrow i$ and each object {\em $i\in \mbox{Ob}I$},
respectively; and $3$-cells
$$\xymatrix@R=20pt@C=8pt{(Fa\!\otimes\!
Fb)\!\otimes\! Fc\ar@{=>}[d]_{ F_{\hspace{-1pt}a,b}\otimes
1}\ar@{}[rrd]|{ {
 F_{\hspace{-1pt}a,b,c}}\,\Rrightarrow}\ar@{=>}[rr]^{
{\boldsymbol{a}}}&&Fa\!\otimes\!(Fb\!\otimes\!Fc)\ar@{=>}[d]^{ 1\otimes
F_{b,c}}\\ F(ab)\!\otimes\! Fc\ar@{=>}[r]_(0.53){
F_{\hspace{-1pt}ab,c}}&F(abc)& Fa\!\otimes\! F(bc),\ar@{=>}[l]^(0.53){
F_{\hspace{-1pt}a,bc}} }
\xymatrix@C=-4pt@R=20pt{&1_{Fi}\!\otimes\!
Fa\ar@{=>}[rd]^{\boldsymbol{l}}\ar@{=>}[ld]_{F_i\otimes 1}
\ar@{}[d]|(.55){
\widehat{F}_a\,\Rrightarrow}&\\F1_i\!\otimes\! Fa\ar@{=>}[rr]_{F_{1, a}}&&Fa,}
\xymatrix@C=-4pt@R=20pt{&Fa\!\otimes\!1_{F\!j}\ar@{=>}[rd]^{\boldsymbol{r}^\bullet}\ar@{=>}[ld]_{1\otimes
F_{\!j}}\ar@{}[d]|(.55){ \widetilde{F}_a\,\Rrightarrow}&
\\Fa\!\otimes\! F1_{\!j}\ar@{=>}[rr]_{F_{a,1}}&&Fa,}
 $$
respectively associated to any three composable arrows $l\overset{
c}\to k \overset{b}\to j\overset{ a}\to i$ and any arrow $j\overset{ a}\to
i$ in the category $I$. These data are required to satisfy the coherence
conditions {\bf (CR1)} and {\bf (CR2)} as stated below.

The set of representations of a small category $I$ in a small
tricategory $\T$ is denoted by
\begin{equation}\label{rep} \rep(I,\T).\end{equation}

A representation $F:I\to \T$  is termed {\em unitary} or {\em
normal} whenever the following conditions hold:  for each object $i$
of $I$, $F1_i=1_{Fi}$ and $F_i=1_{1_{Fi}}$;  for each arrow $a:j\to
i$ of $I$, $F_{a,1_j}=\boldsymbol{r}^\bullet:Fa\otimes 1\Rightarrow
Fa$, $F_{1_i,a}=\boldsymbol{l}:1\otimes Fa\Rightarrow Fa$,  and the
$3$-cells $F_{a,1,c}$, $\hat{F}_a$, and $\tilde{F}_a$ are the unique
coherence isomorphisms. Furthermore, a representation $F:I\to \T$
whose structure $2$-cells $F_{\!a,b}$ are all equivalences (in the
corresponding hom-bicategories of $\T$ where they lie) and whose
structure $3$-cells $F_{\!a,b,c}$, $\hat{F}_a$, and $\tilde{F}_a$,
are all invertible is called a {\em homomorphic representation}. The
subsets of $\rep(I,\T)$ whose elements are the
unitary, homomorphic, and unitary homomorphic representations, are
denoted respectively by
\begin{equation}\label{urep} \urep(I,\T),\  \hrep(I,\T), \  \uhrep(I,\T).\end{equation}
\end{definition}

\noindent {\bf (CR1):} for any four composable arrows in $I$,
$m\overset{ d}\to l\overset{ c}\to k \overset{ b}\to j\overset{a}\to
i$, the equation $A=A'$ on $3$-cells in $\T$ holds, where:
$$
\xymatrix@C=-35pt@R=3pt{((Fa\!\otimes\! Fb)\!\otimes\!
Fc)\!\otimes\! Fd\ar@{}@<-66pt>[dddddddddd]|{\textstyle A\!=}\ar@{=>}[rr]^-{ \boldsymbol{a}}
\ar@{=>}[rddd]^-{ \boldsymbol{a}\otimes 1}\ar@{=>}[ddddd]_{ (F_{a,b}\otimes
1)\!\otimes 1}&&(Fa\!\otimes\!Fb)\!\otimes\!(Fc\!\otimes\!
Fd)\ar@{=>}[rr]^-{ \boldsymbol{a}}  & &
Fa\!\otimes\!(Fb\!\otimes\!(Fc\!\otimes\! Fd))\ar@{=>}[ddddd]^{ 1\otimes(1\otimes F_{c,d})}\\
&&&&\\
&&^{\pi\cong}&&\\
 &(Fa\!\otimes\!(Fb\!\otimes\! Fc))\!\otimes\! Fd\ar@{=>}[rr]^-{ \boldsymbol{a}}
 \ar@{=>}@<35pt>[dddd]|(0.3){\ \ (1\otimes F_{b,c})\otimes 1}&
 &Fa\!\otimes\!((Fb\!\otimes\! Fc)\!\otimes\! Fd)\ar@{=>}[ruuu]^-{ 1\otimes \boldsymbol{a}}
 \ar@{=>}@<-35pt>[dddd]|(0.7){1\otimes (F_{b,c}\otimes 1)\ \ }&\\
 &&&&\\
\hspace{-15pt}(F(ab)\!\otimes\! Fc)\!\otimes\! Fd\ar@{=>}[ddddd]_{
F_{ab,c}\otimes 1} \ar@{}[r]|(0.9){ {
 F_{a,b,c}\otimes 1 }\,\Rrightarrow}&&^{\cong}&
\ar@{}[r]|(0.1){ { 1\otimes
 F_{b,c,d}}\,\Rrightarrow}&
\hspace{15pt}Fa\!\otimes\!(Fb\!\otimes\!F(cd))\ar@{=>}[ddddd]^{ 1\otimes F_{b,cd}}\\
&&&&\\
\ar@{}[rrrrddd]|(0.5){{
 F_{a,bc,d}}\, \Rrightarrow}&(Fa\!\otimes\! F(bc)) \!\otimes\!
Fd\ar@{=>}[rr]^-{
\boldsymbol{a}}\ar@{=>}[lddd]|(0.4){F_{a,bc}\otimes 1}&&
Fa\!\otimes\!(F(bc)\!\otimes\! Fd)\ar@{=>}[rddd]|(0.4){1\otimes F_{bc,d}}&\\
&&&&\\
&&&&\\
F(abc)\!\otimes\! Fd\ar@{=>}[rr]^{
F_{abc,d}}&&F(abcd)&&Fa\!\otimes\! F(bcd)\ar@{=>}[ll]_{ F_{a,bcd} }}
$$
$$\xymatrix@C=-32pt@R=3pt{((Fa\!\otimes\! Fb)\!\otimes\!
Fc)\!\otimes\! Fd
\ar@{}@<-66pt>[dddddddddd]|{\textstyle A'\!=}
 \ar@{=>}[rr]^-{ \boldsymbol{a}}\ar@{=>}[ddddd]_{
(F_{a,b}\otimes 1)\otimes 1}\ar@{}[rddd]^(0.6){ \cong}&&
 (Fa\!\otimes\!Fb)\!\otimes\!(Fc\!\otimes\! Fd)
 \ar@{=>}[rr]^-{ \boldsymbol{a}}
 \ar@{=>}[dddl]|{F_{a,b}\otimes (1\otimes 1)}\ar@{=>}[dddr]|{(1\otimes 1) \otimes F_{c,d}}& &
Fa\!\otimes\!(Fb\!\otimes\!(Fc\!\otimes\! Fd))
\ar@{=>}[ddddd]^{ 1\otimes(1\otimes F_{c,d})}  \ar@{}[lddd]_(0.6){ \cong}\\
&&&&\\
&&&&\\
&F(ab)\!\otimes\!(Fc\!\otimes\! Fd) \ar@{=>}[ddr]_(.4){1 \otimes
F_{c,d}} \ar@{}[rr]^{ \cong}&&(Fa\!\otimes\!Fb)\!\otimes\!
F(cd) \ar@{=>}[ddl]^(.4){\ F_{a,b}\otimes 1}
\ar@{=>}[rdd]^-{\boldsymbol{a}}&\\
&&&&\\
\hspace{-15pt}(F(ab)\!\otimes\! Fc)\!\otimes\! Fd \ar@{=>}[ddddd]_{
F_{ab,c}\otimes 1}\ar@{=>}[ruu]^-{\boldsymbol{a}} &&F(ab)\!\otimes\!
F(cd)\ar@{=>}[ddddd]|{F_{ab,cd}}&&\hspace{15pt}Fa\!\otimes\!(Fb\!\otimes\!F(cd))
\ar@{=>}[ddddd]^{ 1\otimes F_{b,cd}}\\
&&&&\\
&&&&\\
&&&&\\
&&&&\\
F(abc)\!\otimes\! Fd\ar@{=>}[rr]^{
F_{abc,d}}\ar@{}[rruuuuu]^(0.55){{
 F_{ab,c,d}}\,\Rrightarrow}&&F(abcd)&&Fa\!\otimes\!
F(bcd)\ar@{=>}[ll]_{
F_{a,bcd}}\ar@{}[lluuuuu]_(0.55){{
 F_{a,b,cd}}\,\Rrightarrow} }
$$
\noindent {\bf (CR2):} for any two composable arrows  $k\overset{
b}\to j \overset{ a}\rightarrow i$ in $I$, the equations $B=B'$,
$C=C'$, and $D=D'$, on $3$-cells in $\T$ hold, where:
$$
\xymatrix@C=12pt@R=10pt{
\ar@{}@<-56pt>[dd]|{\textstyle B\!=}
&(Fa\!\otimes\! 1_{F\!j})\!\otimes\! Fb\ar@{=>}[rr]^{(1\otimes F_j)\otimes
1}
\ar@{=>}[ld]_{\boldsymbol{a}}\ar@{=>}[dd]|(.6){\boldsymbol{r}^\bullet\otimes
1}\ar@{}[dr]|(.6){ \Lleftarrow\widetilde{F}_a\otimes 1}&&
(Fa\!\otimes\! F1_j)\!\otimes\! Fb\ar@{=>}[rd]^(.6){F_{a,1}\otimes 1}\ar@{=>}[lldd]^{\ F_{a,1}\otimes 1}& \\
Fa\!\otimes\!(1_{F\!j}\!\otimes\! Fb)\ar@{=>}[rd]_{1\otimes \boldsymbol{l}}\ar@{}[r]|(.6){\cong { \mu}}&&&&
Fa\!\otimes\! Fb \ar@{}[ll]|(.6){ =}\ar@{=>}[ld]^{F_{a,b}}\\
&Fa\!\otimes\! Fb\ar@{=>}[rr]^{F_{a,b}}&&F(ab)&
}
$$
$$
\xymatrix@C=-10pt@R=13pt{
\ar@{}@<-64pt>[dd]|{\textstyle B'\!=}
&(Fa\!\otimes\!1_{F\!j})\!\otimes\! Fb\ar@{=>}[rr]^{(1\otimes F_j)\otimes
1} \ar@{=>}[ld]_{\boldsymbol{a}}&&(Fa\!\otimes\! F1_j)\!\otimes\!
Fb\ar@{=>}[rd]^(.6){F_{a,1}\otimes 1}
\ar@{=>}[ld]^{\boldsymbol{a}}\ar@{}[dd]|{\Lleftarrow
 F_{a,1,b}}& \\
Fa\!\otimes\!(1_{F\!j}\!\otimes\! Fb)\ar@{=>}[rr]^{1\otimes(F_j\otimes 1)}
\ar@{}[rrru]|{\cong}\ar@{=>}[rd]_{1\otimes \boldsymbol{l}}&&
Fa\!\otimes\!(F1_j\!\otimes\! Fb)\ar@{=>}[ld]^(.3){1\otimes F_{1,b}}&&Fa\!\otimes\! Fb \ar@{=>}[ld]^{F_{a,b}}\\
&Fa\!\otimes\! Fb\ar@{=>}[rr]^(.6){F_{a,b}}\ar@{}[u]|(.55){
\Lleftarrow 1\otimes \widehat{F}_{a}}&&F(ab)&
} $$
$$
\xymatrix@C=10pt@R=10pt{
\ar@{}@<-66pt>[ddddd]|{\textstyle C\!=}
  &(1_{F\!i}\!\otimes\! Fa)\otimes Fb\ar@{=>}[rd]\ar@{}@<2pt>[rd]^(.6){(F_i\otimes 1)\otimes 1}
                                 \ar@{=>}[ld]_{\boldsymbol{a}}
                                 \ar@{=>}[rdddd]_{\boldsymbol{l}\otimes 1}
                                 \ar@{}@<-5pt>[rdd]|(0.8){\Lleftarrow
                                   \widehat{F}_{a}\otimes 1}
                                 \ar@{}[dd]_(0.65){\cong
                                  \lambda}
                                 \ar@{}@<-10pt>[ddddd]_(0.7){ \cong}
                                 & \\
  1_{F\!i}\!\otimes\!(Fa\!\otimes\!Fb)\ar@{=>}[rrddd]^{\boldsymbol{l}}
                               \ar@{=>}[ddd]_{1\otimes F_{a,b}}
   & & (F1_i\!\otimes\! Fa)\!\otimes\! Fb\ar@{=>}[ddd]^{F_{1,a}\otimes 1} \\ &&
   \\&&
   \\
  1_{F\!i}\!\otimes\! F(ab) \ar@{=>}[rd]_{\boldsymbol{l}}
   && Fa\!\otimes\!Fb\ar@{=>}[ld]^{F_{a,b}} \\
   & F(ab)&
}
$$
$$
\xymatrix@C=6pt@R=10pt{
\ar@{}@<-66pt>[ddddd]|{\textstyle C'\!=}
  &(1_{F\!i}\!\otimes\! Fa)\!\otimes\! Fb\ar@{=>}[rd]\ar@{}@<2pt>[rd]^(.6){(F_i\otimes 1)\otimes 1}
                                 \ar@{=>}[ld]_{\boldsymbol{a}}
                                 & \\
  1_{F\!i}\!\otimes\!(Fa\!\otimes\! Fb)\ar@{=>}[rd]^{F_i\otimes 1}
                               \ar@{=>}[ddd]_{1\otimes F_{a,b}}
                               \ar@{}@<-5pt>[rddd]|(.4){ \cong}
   & ^{\cong}& (F1_i\!\otimes\!Fa)\!\otimes\!Fb\ar@{=>}[ddd]^{F_{1,a}\otimes 1}
                                   \ar@{=>}[ld]_{\boldsymbol{a}}\\
  & F1_i\!\otimes\!(Fa\!\otimes\!Fb)\ar@{=>}[d]^{1\otimes F_{a,b}}
                               \ar@{}[rdd]^(.6){\Lleftarrow F_{1,a,b}}
                                & \\
  & F1_i\!\otimes\!F(ab) \ar@{=>}[dd]^{F_{1,ab}} & \\
  1_{F\!i}\!\otimes\!F(ab) \ar@{=>}[rd]_{\boldsymbol{l}}
                        \ar@{=>}[ru]^{F_i\otimes 1}
                        \ar@{}[r]|(0.7){\Lleftarrow\widehat{F}_{ab}}
   && Fa\!\otimes\! Fb\ar@{=>}[ld]^{F_{a,b}} \\
   & F(ab)&
}
$$
$$
\xymatrix@C=40pt@R=6pt{
\ar@{}@<-66pt>[ddddd]|{\textstyle D\!=}
  &(Fa\!\otimes\! Fb)\!\otimes\! 1_{F\!k}\ar@{=>}[rd]^(.6){F_{a,b}\otimes 1}
                                 \ar@{=>}[ld]_{\boldsymbol{a}}
                                 \ar@{=>}[ldddd]^{\boldsymbol{r}^\bullet}
                                 & \\
  Fa\!\otimes\!(Fb\!\otimes\! 1_{F\!k})\ar@{=>}[ddd]_{1\otimes \boldsymbol{r}^\bullet}
                               \ar@{}[rdd]|(0.35){\rho\, \cong}
   & & F(ab)\!\otimes\! 1_{F\!k}\ar@{=>}[ddd]^{1\otimes F_k}
                            \ar@{=>}[ldddd]_{\boldsymbol{r}^\bullet} \\
   & \ar@{}[d]|{ \cong}&
   \\&&
   \\
  Fa\!\otimes\! Fb \ar@{=>}[rd]_{F_{a,b}}
   && F(ab)\!\otimes\! F1_k\ar@{=>}[ld]^{F_{ab,1}}
                         \ar@{}@<16pt>[uu]|(0.6){\Lleftarrow{\widetilde{F}_{ab}}} \\
   & F(ab)&
}
$$
$$
\xymatrix@C=-15pt@R=5pt{
\ar@{}@<-56pt>[ddddd]|{\textstyle D'\!=}
  &&(Fa\!\otimes\! Fb)\!\otimes\!1_{F\!k}\ar@{=>}[rrd]^{F_{a,b}\otimes 1}
                                 \ar@{=>}[lld]_{\boldsymbol{a}}
                                 \ar@{=>}[rdd]^{1\otimes F_k}
                                 \ar@{}[dddl]|{ \cong}
                                 && \\
  Fa\!\otimes\!(Fb\!\otimes\!1_{F\!k})\ar@{=>}[ddd]_{1\otimes \boldsymbol{r}^\bullet}
                               \ar@{=>}[rdd]^{1\otimes(1\otimes F_k)}
   & &&& F(ab)\!\otimes\!1_{F\!k}\ar@{=>}[ddd]^{1\otimes F_k}
                              \ar@{}[ld]|{ \cong}
                            \\
  \ar@{}@<3pt>[d]^(0.5){\Lleftarrow{ 1\otimes \widetilde{F}_b}}
  & &&(Fa\!\otimes\! Fb)\!\otimes\! F1_k\ar@{=>}[rdd]_{F_{a,b}\otimes 1}
                               \ar@{=>}[dll]^{\boldsymbol{a}}& \\
  & Fa\!\otimes\!(Fb\!\otimes\!F1_k)\ar@{=>}[dl]^{1\otimes F_{b,1}}&
       \ar@{}[d]|{\Lleftarrow{ F_{a,b,1}}}&&
   \\
  Fa\!\otimes\! Fb \ar@{=>}[rrd]_{F_{a,b}}
   &&&& F(ab)\!\otimes\! F1_k\ar@{=>}[lld]^{F_{ab,1}}
                          \\
   && F(ab)&&
}
$$

\subsection{Example} \label{eta}
 Let $A$ be an abelian group, and let $\Sigma^2A$  denote the tricategory
(actually a 3-groupoid) having only one $i$-cell for $0\leq i\leq 2$
and whose $3$-cells are the elements of $A$, with all the
compositions given by addition in $A$. Then, for any small category
$I$ (e.g., a group $G$ or a monoid $M$), a unitary representation
$F:I\to \Sigma^2A$ is the same as a function $F:\ner_3I\to A$
satisfying the equations
$$F(b,c,d)+F(a,bc,d)+F(a,b,c)=F(ab,c,d)+F(a,b,cd),$$ and such that
$F(a,b,c)=0$ whenever any of the arrows $a$, $b$, or $c$ is an
identity. Thus $\urep(I,\Sigma^2A)=Z^3(I,A)$, the set of normalized
$3$-cocycles of (the nerve $\ner I$ of) the category $I$ with
coefficients in the abelian group $A$.

\subsection{The bicategory of
representations}
 For any category $I$ and any tricategory $\T$, the set $\rep(I,\T)$ of representations of $I$ in $\T$
 is the set of objects of a bicategory, {\em the bicategory of
representations of $I$ in $\T$}, denoted by
\begin{equation}\label{brep} \brep(I,\T), \end{equation}
whose 1-cells are a kind of degenerated lax transformations between
representations that agree on objects. When $\T=\b$ is a bicategory, that is, when
the 3-cells are all identities, these degenerated lax
transformations has been considered in \cite{b-c,ccg,c-c-g} under
the name of {\em relative to objects lax transformations}, whereas
in \cite{garner-gurski, lack, lack-paoli} they are termed {\em
icons}, short for `identity component oplax natural
transformations'. The description of the bicategory $\brep(I,\T)$ is
as follows:

$\bullet$ {\em The cells of  $\brep(I,\T)$}. As we said above, representations $F:I\to \T$ are the 0-cells of this bicategory. For any two representations $F,G:I\to \T$, a
1-cell $\Phi:F\Rightarrow G$ may exists only if $F$ and $G$ agree on
objects, that is, $Fi=Gi$ for all $i\in\mbox{Ob}I$; and is then
given by specifying, for every arrow $a:j\to i$  in $I$, a 2-cell
$\Phi a:Fa\Rightarrow Ga$ of $\T$,  and $3$-cells
\begin{equation}\label{3phi}\begin{array}{ll}
\xymatrix@R=20pt@C=15pt{Fa\otimes Fb\ar@{=>}[r]^{F_{a,b}}\ar@{=>}[d]_{\Phi
a\otimes \Phi b}&F(ab) \ar@{=>}[d]^{\
\Phi(ab)}\ar@{}[dl]|{ {
\Phi_{a,b}}\Rrightarrow}\\
Ga\otimes Gb\ar@{=>}[r]_{G_{a,b}}&G(ab),}&\xymatrix@R=20pt@C=5pt{&1_{Fi=Gi}\ar@{=>}[ld]_{F_i}
\ar@{=>}[rd]^{G_i}\ar@{}[d]|(.55){ { \Phi_i } \cong}& \\
F1_i\ar@{=>}[rr]_{\Phi{1_i}}&&G1_i,
}
\end{array}
\end{equation}
respectively  associated to each pair of composable arrows
$k\overset{ b}\to j\overset{ a}\rightarrow i$ and each object $i$ of
the category $I$,  subject to the two coherence axioms {\bf (CR3)} and {\bf (CR4)} below.

\noindent {\bf (CR3):} for any two composable arrows  triplet of composable morphisms of
$I$, $l\overset{ c}\to k \overset{ b}\to j\overset{ a}\to i$,  the equation $E=E'$
 on $3$-cells in $\T$ holds, where:
$$
\xymatrix@R=20pt@C=20pt{(Fa\!\otimes\! Fb)\!\otimes\!Fc
\ar@{}@<-86pt>[dd]|{\textstyle E=}
\ar@{=>}[rr]^{ \boldsymbol{a}}
\ar@{=>}[rrd]^{F_{a,b}\otimes 1}
\ar@{=>}[d]_{(\Phi a\otimes \Phi b)\otimes \Phi c}&&Fa\!\otimes\!(Fb\!\otimes\! Fc)
\ar@{=>}[rr]^{1\otimes F_{b,c}}&
\ar@{}@<-25pt>[d]|{{ F_{a,b,c}}\Rrightarrow}&Fa\!\otimes\! F(bc)
\ar@{=>}[d]^{F_{a,bc}}\\
(Ga\!\otimes\! Gb)\!\otimes\! Gc \ar@{=>}[d]_{G_{a,b}\otimes 1}
\ar@{}[rr]|(.5){ { \Phi_{a,b}\otimes 1}\Rrightarrow}&&
F(ab)\!\otimes\! Fc\ar@{=>}[lld]^{\ \Phi(ab)\otimes \Phi c}\ar@{=>}[rr]^(.7){F_{ab,c}}
\ar@{}@<55pt>[d]|{ { \Phi_{ab,c}}\Rrightarrow}&& F(abc)\ar@{=>}[d]^{\Phi(abc)}\\
G(ab)\!\otimes\! Gc\ar@{=>}[rrrr]^{G_{ab,c}}&&&&G(abc),}
$$
$$
\xymatrix@R=20pt@C=20pt{
(Fa\!\otimes\! Fb)\!\otimes\!Fc
\ar@{}@<-70pt>[dd]|{\textstyle E'=}
\ar@{}[rrd]|{
\cong}\ar@{=>}[rr]^{ \boldsymbol{a}}\ar@{=>}[d]_{(\Phi
a\otimes \Phi b)\otimes \Phi c}&&Fa\!\otimes\!(Fb\!\otimes\!
Fc)\ar@{}[rd]^(.65){ {
1\otimes\Phi_{b,c}}\Rrightarrow}\ar@{=>}[d]|{\Phi a\otimes(\Phi
b\otimes\Phi c)}
\ar@{=>}[rr]^{1\otimes F_{b,c}}&&Fa\!\otimes\! F(bc)
\ar@{=>}[ld]_{\Phi a\otimes\Phi(bc)}\ar@{=>}[d]^{F_{a,bc}}\\
(Ga\!\otimes\! Gb)\!\otimes\! Gc \ar@{=>}[d]_{G_{a,b}\otimes 1}\ar@{=>}[rr]^{\boldsymbol{a}}&&
Ga\!\otimes\!(Gb\!\otimes\! Gc)\ar@{}@<-35pt>[d]|{{ G_{a,b,c}}\Rrightarrow}
\ar@{=>}[r]_-{1\otimes G_{b,c}}&Ga\!\otimes\! G(bc)\ar@{=>}[rd]_{G_{a,bc}}
\ar@{}[r]|(.6){ { \Phi_{a,bc}}\Rrightarrow}& F(abc)\ar@{=>}[d]^{\Phi(abc)}\\
G(ab)\!\otimes\! Gc\ar@{=>}[rrrr]^{G_{ab,c}}&&&&G(abc).}
$$

\noindent {\bf (CR4):} for any morphism of $I$, $j\overset{ a}\to
i$, the following two pasting equalities hold:

$$\xymatrix@R=10pt@C=10pt{ &1\otimes
Fa\ar@{}[d]|(.55){{
\widehat{F}_a}\Rrightarrow}\ar@{=>}[rd]^{\boldsymbol{l}}\ar@{=>}[ld]_{F_i\otimes
1}&&&& 1\otimes Fa\ar@{}[dddr]|(.5){\cong}
\ar@{=>}[dd]|{1\otimes \Phi a}\ar@{=>}[rd]^{\boldsymbol{l}}\ar@{=>}[ld]_{F_i\otimes 1}&\\
F1\otimes Fa
 \ar@{}[rrdd]|(.5){{
\Phi_{1,a}}\Rrightarrow}
\ar@{=>}[rr]_{F_{1,a}}\ar@{=>}[dd]_{\Phi 1\otimes \Phi a}&&Fa\ar@{=>}[dd]^{\Phi a}&&
F1\otimes Fa \ar@{}[rd]|(.55){{
\Phi_i\otimes 1}\cong}
\ar@{=>}[dd]_{\Phi 1\otimes \Phi a}&&Fa\ar@{=>}[dd]^{\Phi a}\\
&&&=&&1\otimes Ga^{}\ar@{}[d]|(.55){{
\widehat{G}_a}\Rrightarrow}\ar@{=>}[rd]^{\boldsymbol{l}}\ar@{=>}[ld]_{G_i\otimes 1}&\\
G1\otimes Ga\ar@{=>}[rr]_{G_{1,a}}&&Ga&&G1\otimes Ga\ar@{=>}[rr]_{G_{1,a}}& &Ga,}
$$
$$
\xymatrix@C=10pt@R=10pt{ &Fa\otimes 1
\ar@{}[d]|(.55){{
\widetilde{F}_a}\Rrightarrow}
\ar@{=>}[rd]^{\boldsymbol{r}^\bullet}\ar@{=>}[ld]_{1\otimes F_j}&&&&
Fa\otimes 1\ar@{=>}[dd]|{\Phi a\otimes
1}\ar@{=>}[rd]^{\boldsymbol{r}^\bullet}
\ar@{=>}[ld]_{1\otimes F_j}\ar@{}[rddd]|(.5){\cong}&\\
Fa\otimes F1 \ar@{}[rrdd]|(.5){{
\Phi_{1,a}}\Rrightarrow}
\ar@{=>}[rr]_{F_{a,1}}\ar@{=>}[dd]_{\Phi a\otimes \Phi 1}&&
Fa\ar@{=>}[dd]^{\Phi a}&&Fa\otimes F1\ar@{}[rd]|(.5){{
1\otimes \Phi_j}\cong}\ar@{=>}[dd]_{\Phi a\otimes \Phi 1}&&Fa\ar@{=>}[dd]^{\Phi a}\\
&&&=&&Ga\otimes 1
\ar@{}[d]|(.55){{
\widetilde{G}_a}\Rrightarrow}\ar@{=>}[rd]^{\boldsymbol{r}^\bullet}\ar@{=>}[ld]_{1\otimes G_j}&\\
Ga\otimes G1\ar@{=>}[rr]_{G_{a,1}}&&Ga&&Ga\otimes G1\ar@{=>}[rr]_{G_{a,1}}& &Ga.}
$$

A 2-cell $M:\Phi\Rrightarrow\Psi$, for $\Phi,\Psi:F\Rightarrow G$
two 1-cells in $\brep(I,\T)$, consists of a family of 3-cells in
$\T$, $ Ma:\Phi a\Rrightarrow \Psi a$, one for each arrow $a:j\to i$
in $I$, subject to the coherence condition {\bf (CR5)} below.

\noindent {\bf (CR5):} for any object $i$ and each two composable
arrows $k \overset{ b}\to j\overset{ a}\to i$ of $I$, the diagrams
of 3-cells below commute. $$\xymatrix@C=8pt@R=16pt{\Phi 1_i\circ
F_i\ar@3{->}[rr]^{M1_i\circ 1}\ar@3{->}[rd]_{\Phi_i}&& \Psi 1_i\circ
F_i\ar@3{->}[ld]^{\Psi_i}\\&G_i& }\hspace{0.6cm}
 \xymatrix@C=12pt@R=16pt{{G}_{\hspace{-1pt}a,b}\!\circ\!(\Phi a\!\otimes\!\Phi b)
\ar@3{->}[d]_{ 1\circ(Ma\otimes M
b)}\ar@3{->}[r]^-{\Phi_{a,b}}&\ar@3{->}[d]^{ M(ab)\otimes
1}\Phi(ab)\!\circ\! {F}_{\hspace{-1pt} a,b}\\\ar@3{->}[r]^-{
\Psi_{a,b}}{G}_{\hspace{-1pt}a,b}\!\circ\! (\Psi a\!\otimes\!\Psi
b)&\Psi(ab)\!\circ\! {F}_{\hspace{-1pt}a,b} }
$$

$\bullet$ {\em{Compositions in  $\brep(I,\T)$}}. The vertical
composition of a 2-cell $M:\Phi\Rrightarrow\Psi$ with a 2-cell
$N:\Psi\Rrightarrow\Gamma$, for $\Phi,\Psi,\Gamma:F\Rightarrow G$,
yields the 2-cell $N\cdot M:\Phi\Rrightarrow\Gamma$ which is defined
using pointwise vertical composition in the hom-bicategories of
$\T$; that is, for each $a:j\to i$ in $I$, $(N\cdot
M)a=(Na)\cdot(Ma):\Phi a\Rrightarrow\Gamma a:Fa\Rightarrow Ga$. The
horizontal composition of 1-cells $\Phi:F\Rightarrow G$ and
$\Psi:G\Rightarrow H$, for $F,G,H:I\to \T$ representations, is
$\Psi\circ\Phi:F\Rightarrow H$, where $(\Psi\circ\Phi) a=\Psi
a\circ\Phi a:Fa\Rightarrow Ha$, for each arrow $a:j\to i$ in $I$.
Its component $(\Psi\circ\Phi)_{a,b}:(\Psi\circ\Phi)a\otimes
(\Psi\circ\Phi)b\Rrightarrow \Psi(ab)\circ \Phi(ab)$, attached at a
pair of composable arrows $k \overset{ b}\to j\overset{ a}\to i$ of
the category $I$, is given by pasting in the bicategory $\T(Fk,Fi)$
the diagram
$$
\xymatrix@C=20pt@R=12pt{
Fa\otimes Fb\ar@{=>}[dd]_{(\Psi a\circ\Phi a)\otimes (\Psi b\circ\Phi b)}
\ar@{=>}[rr]^{F_{a,b}}\ar@{}[rrrd]|{{ \Phi_{a,b}}\Rrightarrow}
\ar@{=>}[dr]^{\Phi a\otimes\Phi b}&& F(ab)\ar@{=>}[rd]^{\Phi(ab)}&\\
&Ga\otimes Gb\ar@{}[l]|(.65){ \cong}\ar@{=>}[rr]^{G_{a,b}}\ar@{}[rd]|{{ \Psi_{a,b}}\Rrightarrow}
\ar@{=>}[ld]^{\Psi  a\otimes\Psi b}
&&G(ab)\ar@{=>}[ld]^{\Psi(ab)}\\
Ha\otimes Fb\ar@{=>}[rr]_{H_{a,b}}&&H(ab),&}
$$
and its component $(\Psi\circ\Phi)_i:(\Psi\circ\Phi)1_i\circ
F_i\Rrightarrow H_i$, attached at any object $i$ of $I$, is given by
pasting in $\T(Fi,Fi)$ the diagram
$$
\xymatrix@R=10pt@C=10pt{&1_{Fi=Gi=Hi}\ar@{=>}[rdd]^{H_i}\ar@{=>}[ldd]_{F_i}\ar@{=>}[dd]|{G_i}&\\
&&\\ F1_i\ar@{=>}[r]_{\Phi 1_i}&G 1_i
\ar@{}@<6pt>[ru]|(.37){{ \Psi_i}\cong}
\ar@{}@<-6pt>[lu]|(.37){{ \Phi_i}\cong}\ar@{=>}[r]_{\Psi 1_i}&H 1_i.
}
$$

The horizontal composition  of 2-cells $M:\Phi\Rrightarrow
\Psi:F\Rightarrow G$ and $N:\Gamma\Rrightarrow \Theta:G\Rightarrow
H$ in $\brep(I,\T)$
is $N\circ M:\Gamma\circ\Phi\Rrightarrow \Theta\circ \Psi$, which,
at each $a:j\to i$ in $I$, is given by the formula $(N\circ
M)a=Na\circ Ma.$

$\bullet$ {\em{Identities in  $\brep(I,\T)$}}. The identity 1-cell
of a representation $F:I\to \T$ is $1_F:F\Rightarrow F$, where
$(1_F)a=1_{Fa}$, the identity of $Fa$ in the bicategory $\T(Fi,Fi)$,
for each $a:j\to i$ in $I$. Its structure 3-cell
$(1_F)_{a,b}:F_{a,b}\circ (1_{Fa}\otimes 1_{Fa})\Rrightarrow
1_{F(ab)}\otimes F_{a,b}$,
 attached at each
pair of composable arrows $k \overset{ b}\to j\overset{ a}\to i$, is
the canonical one obtained from the identity constraints of the
bicategory $\T(Fk,Fi)$ by pasting  the diagram
$$
\xymatrix@R=25pt{Fa\otimes
Fb\ar@{=>}[r]^{F_{a,b}}\ar@{=>}[d]_{1\otimes 1}\ar@{=>}@/^2pc/[d]^1\ar@{}@<12pt>[d]|{
\cong}&F(ab)
\ar@{=>}[d]^{1}\ar@{}@<-20pt>[d]|{
\cong}\\
Fa\otimes Fb\ar@{=>}[r]_{F_{a,b}}&F(ab),}
$$
and its component attached at any object $i$ of $I$ is the left unit
constraints of the bicategory $\T(Fi,Fi)$ at $F_i:1_{F_i}\Rightarrow
F1_i$, that is, $(1_F)_i=\boldsymbol{l}:1_{F1_i}\circ F_i\cong F_i$.
The identity 2-cell $1_{\Phi}$, of a 1-cell $\Phi:F\Rightarrow G$,
is defined at any arrow $a:j\to i$ of $I$ by the simple formula
$(1_\Phi) a=1_{\Phi a}:\Phi a\Rrightarrow\Phi a $.

$\bullet$ {\em{The structure constraints in $\brep(I,\T)$}}. For any
three composable 1-cells $F\overset{\Phi}\Rightarrow
G\overset{\Psi}\Rightarrow H \overset{\Theta}\Rightarrow K$ in
$\brep(I,T)$, the component of the structure associativity
isomorphism $(\Theta\circ \Psi)\circ \Phi \cong \Theta\circ
(\Psi\circ\Phi)$, at any arrow $j\overset{a}\to i$ of the category
$I$, is provided by the associativity constraint $(\Theta a\circ
\Psi a)\circ \Phi a \cong \Theta a\circ (\Psi a\circ\Phi a)$
 of the hom-bicategory $\T(Fj,Fi)$. And similarly, the components of structure left and right
 identity isomorphisms $1_G\circ \Phi\cong \Phi$ and $\Phi\circ 1_F\cong \Phi$, at any arrow $a:j\to i$
 as above, are provided by the identity constraints $1_{Ga}\circ \Phi a\cong \Phi a$,
 and $\Phi a\circ 1_{Fa}\cong \Phi a$, of the bicategory $\T(Fj,Fi)$, respectively.

\subsection{The bicategories of unitary and homomorphic
representations} The bicategory of representations of a category $I$
in a tricategory $\T$, $\brep(I,\T)$, contains three
sub-bicategories that are of interest in our development:

The {\em bicategory of unitary representations},  denoted by
\begin{equation} \label{burep}
 \burep(I,\T),
\end{equation}
whose 0-cells are the unitary representations of $I$ in $\T$; its
1-cells are those $\Phi:F\Rightarrow G$ in $\brep(I,\T)$ that are
unitary, in the sense that $\Phi{1_i}=1_{1_{\hspace{-1.5pt}Fi}}$ and
the 3-cell $\Phi_i$ in $(\ref{3phi})$ is the canonical isomorphism
$1\circ 1\cong 1$, for all objects $i$ of $I$, and it is full on
2-cells $M:\Phi\Rrightarrow\Psi$ between such normalized 1-cells.

The {\em bicategory of homomorphic representations}, denoted by
\begin{equation} \label{bhrep}
 \bhrep(I,\T),
\end{equation}
whose 0-cells are  the homomorphic representations, its 1-cells are
those $\Phi:F\Rightarrow G$ in $\burep(I,\T)$ such that the
structure 3-cells $\Phi_{a,b}$ are all invertible, and it is full on
2-cells $M:\Phi\Rrightarrow\Psi$ between such 1-cells.

The {\em bicategory of unitary homomorphic representations}, denoted by $\buhrep(I,\T)$,  which
is defined to be the intersection of the above two, that is,
\begin{equation}\label{buhrep}\buhrep(I,\T)= \burep(I,\T)\cap
\bhrep(I,\T).
\end{equation}

\subsection{Example}\label{eta2}Let $\Sigma^2A$ be
the strict tricategory
 defined by an abelian group $A$ as in Example \ref{eta} and let $I$ be any category. Then,  the bicategory
   $\burep(I,\Sigma^2A)$ is actually a 2-groupoid whose objects  are normalized 3-cocycles  of $I$ with
 coefficients in $A$. If $F,G:\ner_3I\to A$ are two such 3-cocycles, then  a 1-cell $\Phi:F\Rightarrow G$
 is a normalized 2-cochain $\Phi:\ner_2 I\to A$ satisfying
 $$ G(a,b,c)+\Phi(b,c)+\Phi(a,bc)= F(a,b,c)+\Phi(ab,c)+\Phi(a,b),$$
 that is, $G=F+\partial \Phi$. And for any two 1-cells $\Phi,\Psi:F\Rightarrow G$ as above, a 2-cell
 $M:\Phi\Rrightarrow \Psi$ consists of a normalized 1-cochain $M:\ner_1I\to A$ such
 that $\Psi=\Phi+\partial
 M$, that is, such that
 $$\Psi(a,b)+M(a)+M(b)=M(ab)+\Phi(a,b).$$

 \subsection{Functorial proprieties of $\bhrep(I,-)$}  For a tricategory $\T$, any functor $a:J\to I$ induces a homomorphism (actually a 2-functor)
  $a^*:\bhrep(I,\T)\to \bhrep(J,\T)$, and the construction (\ref{brep}), $I\mapsto \bhrep(I,\T)$,  is
functorial on the  category $I$. For a trihomomorphism of
tricategories $H=(H,\chi,\iota,\omega,\gamma,\delta):\T\to\T'$, as
defined in \cite[Definition 3.1]{g-p-s}, we have the following
result:
\begin{lemma}\label{transtri} Let $I$ be any given small category.

(i) Every trihomomorphism $H:\T\to\T'$ gives rise to a homomorphism
$$H_*:\bhrep(I,\T)\to\bhrep(I,\T'),$$
which is natural on $I$, that is, for any functor $a:J\to I$,
 $$H_*a^*=a^*H_*:\bhrep(I,\T)\to\bhrep(J,\T').$$

 (ii) If $H:\T\to\T'$ and $H':\T'\to\T''$ are any two
composable trihomomorphisms, then there is a  pseudo-equivalence
$m:H'_*H_*\Rightarrow (H'H)_*$, such that,  for any functor $a:J\to
I$,
 the equality $ma^*=a^*m$ holds.

(iii) For any tricategory $\T$, there is a pseudo-equivalence
$m:(1_\T)_*\Rightarrow 1$, such that,  for any functor $a:J\to I$,
the equality $ma^*=a^*m$ holds.
\end{lemma}
\begin{proof} $(i)$ The homomorphism
$H_*$ is defined as follows: It carries a homomorphic representation
$F:I\to\T$ to the homomorphic representation $H_*F:I\to\T'$, which
is defined on objects $i$ of $I$ by $(H_*F)i=HFi$, and on arrows
$a:j\to i$ by $(H_*F)a=HFa:HFj\to HFi$. The 2-cell
$(H_*F)_{a,b}:(H_*F)a\otimes (H_*F)b\Rightarrow (H_*F)(ab)$, for
each pair of composable arrows $k\overset{ b}\to j
\overset{a}\rightarrow i$,  is the composition $ HFa\otimes
HFb\overset{\chi}\Longrightarrow H(Fa\otimes
Fb)\overset{HF_{a,b}}\Longrightarrow HF(ab)$. For each object $i$,
the 2-cell $(H_*F)_i:1_{(H_*F)i}\Rightarrow (H_*F)1_i$ is the
composite of $ 1_{HFi}\overset{\iota}\Longrightarrow
H1_{Fi}\overset{HF_i}\Longrightarrow HF1_i$.  The structure 3-cell
of $H_*F:I\to \T'$ associated to any three composable arrows
$l\overset{ c}\to k \overset{b}\to j\overset{ a}\to i$, is that
obtained by pasting the diagram
$$
\xymatrix@C=-7pt@R=6pt{&(H\!Fa\!\otimes\! H\!Fb)\!\otimes\! H\!Fc \ar@{}[rrrrrrdd]|{ {
  \omega}\Rrightarrow}
\ar@{=>}[dl]_(.6){\chi\otimes 1}\ar@{=>}^{\boldsymbol{a}}[rrrrrr]&&&&&&
H\!Fa \!\otimes\! (H\!Fb\!\otimes\! H\!Fc)\ar@{=>}[dr]^(.6){1\otimes \chi}&\\
H\!(Fa\!\otimes\! Fb)\!\otimes\!H\!Fc\ar@{=>}[rd]^(.55){\chi}
\ar@{=>}[dd]_{H\!F_{a,b}\otimes 1}
&&&&&&&&
H\!Fa\!\otimes\! H\!(Fb\!\otimes\! Fc)
\ar@{=>}[ld]_(.55){\chi}
\ar@{=>}[dd]^{1\otimes H\!F_{b,c}}
\\
&H\!((Fa\!\otimes\!Fb)\!\otimes\! Fc)\ar@{=>}[dd]|{H(F_{a,b}\otimes 1)}\ar@{=>}[rrrrrr]^{H\boldsymbol{a}}
\ar@{}[rrrrrrddd]|{  F_{a,b,c}\Rrightarrow}&&&&&&
H\!(Fa\!\otimes\!(Fb\!\otimes\! Fc))\ar@{}[rd]|{ \cong}
\ar@{=>}[dd]|{H(1\otimes F_{b,c})}&\\H\!F(ab)\!\otimes\!H\!Fc\ar@{=>}[rd]_{\chi}
\ar@{}[ru]|{ \cong}&&&&&&&&H\!Fa\otimes H\!F(bc)\ar@{=>}[ld]^{\chi}\\
&H\!(F(ab)\!\otimes\! Fc)\ar@{=>}[rrrd]_-{H\!F_{ab,c}}&&&
&&&H\!(Fa\!\otimes\!F(bc))\ar@{=>}[llld]\ar@{}@<3pt>[llld]^(.45){\ H\!F_{a,bc}}&
\\
&&&&H\!F(abc)&&&&}
$$
whereas the structure 3-cells of the representation $H_*F$ attached to
an arrow $a:j\to i$ of the category $I$, are respectively those
obtained by pasting the diagrams below.
$$
\xymatrix@C=12pt@R=5pt{
&H1_{Fi}\otimes HFa\ar@{}[ddd]|{ \cong}
\ar@{=>}[rd]_{\chi}\ar@{=>}[lddd]_{HF_i\otimes 1}
&&&1_{HFi}\otimes HFa\ar@{=>}[ddd]^{\boldsymbol{l}}\ar@{=>}[lll]_{\iota\otimes 1}\\
&&H(1_{Fi}\otimes Fa)\ar@{}[dd]|{ { H\widehat{F}_a}\Rrightarrow}
\ar@{=>}[rrdd]^{H\boldsymbol{l}}\ar@{=>}[ddl]|{H(F_i\otimes 1)}\ar@{}[rru]_(.6){ { \gamma} \Rrightarrow}&&\\
&&&&\\
HF1_i\otimes HFa\ar@{=>}[r]^{\chi}&H(F1_i\otimes Fa)\ar@{=>}[rrr]^(.65){HF_{1,a}}&&&HFa,
}
$$
$$
\xymatrix@C=12pt@R=5pt{
&HFa\otimes H1_{Fj}\ar@{}[ddd]|{ \cong}
\ar@{=>}[rd]_{\chi}\ar@{=>}[lddd]_{1\otimes HF_j}
&&&HFa\otimes 1_{HFj}\ar@{=>}[ddd]^{\boldsymbol{r}^\bullet}\ar@{=>}[lll]_{1\otimes \iota}\\
&&H(Fa\otimes 1_{Fj})\ar@{}[dd]|{ { H\widetilde{F}_a}\Rrightarrow}
\ar@{=>}[rrdd]^{H\boldsymbol{\boldsymbol{r}^\bullet}}\ar@{=>}[ddl]|{H(1\otimes F_j)}
\ar@{}[rru]_(.6){ { \delta^{-1}} \Rrightarrow}&&\\
&&&&\\
HFa\otimes HF1_j\ar@{=>}[r]^{\chi}&H(Fa\otimes F1_j)\ar@{=>}[rrr]^(.65){HF_{a,1}}&&&HFa.
}
$$

If $\Phi:F\Rightarrow G$ is any 1-cell in the bicategory
$\bhrep(I,\T)$, then $H_*\Phi:H_*F\Rightarrow H_*G$ is the 1-cell in
$\brep(I,\T')$ whose component at an arrow $a:j\to i$ of $I$ is the
2-cell of $\T'$ defined by $(H_*\Phi)a=H\Phi a:HFa\Rightarrow HGa$.
For any pair of composable arrows $k\overset{ b}\to j
\overset{a}\rightarrow i$ and any object $i$ of $I$, the
corresponding structure 3-cells (\ref{3phi}), $(H_*\Phi)_{a,b}$ and
$(H_*\Phi)_i$,  are respectively  given by pasting in
$$
\xymatrix@R=35pt@C=20pt{HFa\otimes HFb\ar@{}[rd]|{ \cong}
\ar@{=>}[r]^{\chi}\ar@{=>}[d]_{H\Phi a\otimes H\Phi b} &
H(Fa\otimes Fb)\ar@{=>}[r]^{HF_{a,b}}\ar@{=>}[d]|(.3){H(\Phi a\otimes \Phi b)}
\ar@{}[rd]|{{ H\Phi_{a,b}}\Rrightarrow}&HF(ab)\ar@{=>}[d]^{H\Phi(ab)}\\
HGa\otimes HGb\ar@{=>}[r]_{\chi}&H(Ga\otimes Gb)\ar@{=>}[r]_-{HG_{a,b}}&HG(ab),}
\xymatrix@R=12pt@C=8pt{&1_{HFi=HGi}\ar@{=>}[d]^\iota& \\ &H1_{Fi}\ar@{=>}[ld]_{HF_i}\ar@{=>}[rd]^{HG_i}
\ar@{}[d]|(.55){ { H\Phi_{i}} \cong}& \\HF1_i\ar@{=>}[rr]_{H\Phi_i}&&HG1_i.}
$$
And a 2-cell $M:\Phi\Rrightarrow \Psi$ of $\brep(I,\T)$ is applied
by the homomorphism $H_*$ to the 2-cell $H_*M:H_*\Phi\Rrightarrow
H_*\Psi$ of $\brep(I,\T')$, such that $(H_*M)a=HMa:H\Phi
a\Rrightarrow H\Psi a$ for any arrow $a:j\to i$ of the category $I$.

 Finally, if $\Phi:F\Rightarrow G$ and $\Psi:G\Rightarrow H$ are
any two composable 1-cells in $\brep(I,\T)$,
 and $F:I\to \T$ is any representation,
 then the constraints
 $(H_*\Psi)\circ (H_*\Phi)\cong H_*(\Psi\circ \Phi)$ and $1_{H_*F}\cong H_*1_F$ are,
 at any arrow $a:j\to i$ of $I$,
 the structure isomorphisms $H\Psi a\circ H\Phi a\cong H(\Psi a\circ \Phi a)$ and $1_{HFa}\cong H1_{Fa}$
 of the homomorphism $H:\T(Fj,Fi)\to \T'(HFj,HFi)$, respectively.

$(ii)$ For any homomorphic representation $F:I\to \T$, the 2-cell
attached by
$$m=m_F:H'_*(H_*F)\Rightarrow (H'H)_*F$$ at any arrow $a:j\to i$ of
$I$ is the identity, that is, $ma =1_{H'HFa}$. For any pair of
composable arrows $k\overset{ b}\to j \overset{a}\rightarrow i$ and
any object $i$ of $I$, the corresponding invertible structure
3-cells (\ref{3phi}),
$$\begin{array}{l}m_{a,b}:(H'_*(H_*F))_{a,b}\circ (ma\otimes
mb)\Rrightarrow m(ab)\circ ((H'H)_*F)_{a,b},\\ m_i:m1_i\circ
(H'_*(H_*F))_i\Rrightarrow ((H'H)_*F)_i,\end{array}$$ are,
respectively, given by pasting in the diagrams below.
$$
\xymatrix@C=35pt{
H'HFa\otimes H'HFb\ar@{}[rd]|{ \cong}\ar@{=>}[r]^\chi\ar@{=>}[d]_{1\otimes 1}&H'(HFa\otimes HFb)\ar@{=>}[rd]^{H'\chi}\ar@{=>}[d]_{1}\ar@{=>}[rr]^{H'(HF_{a,b}\circ \chi)}&&H'HF(ab)\ar@{=>}[d]^{1} \\
H'HFa\otimes H'HFb\ar@{=>}[r]_\chi&H'(HFa\otimes HFb)\ar@{=>}[r]_{H'\chi}\ar@{}@<4pt>[ru]|(.3){ \cong}&H'H(Fa\otimes Fb)\ar@{=>}[r]_-{H'HF_{a,b}}\ar@{=>}[ru]^{H'HF_{a,b}}\ar@{}[u]|{ \cong}&H'HF(ab),\ar@{}@<-4pt>[lu]|(.3){ \cong}
}
$$
$$
\xymatrix@C=35pt{
1_{H'HFi}\ar@{=>}[r]^\iota\ar@{=>}[rd]_{\iota}&H'1_{HFi}\ar@{}[ld]|(.3){ \cong}\ar@{=>}[rd]^{H'\iota}\ar@{=>}[d]_{1}\ar@{=>}[rr]^{H'(HF_{i}\circ \iota)}&&H'HF1_i\ar@{=>}[d]^{1} \\
 &H'1_{HFi}\ar@{=>}[r]_{H'\iota}\ar@{}@<4pt>[ru]|(.3){ \cong}&H'H1_{Fi}\ar@{=>}[r]_-{H'HF_{i}}\ar@{=>}[ru]^{H'HF_{i}}\ar@{}[u]|{ \cong}&H'HF1_i.\ar@{}@<-4pt>[lu]|(.3){ \cong}
}
$$

 If $\Phi:F\Rightarrow G$ any 1-cell in $\brep(I,\T)$, then the invertible naturality 2-cell $$m_\Phi: m_G\circ (H'_*(H_*\Phi))\Rrightarrow (H'H)_*\Phi\circ m_F,$$  at any arrow $a:j\to i$ of $I$, is provided by the canonical isomorphism $1\circ H'H\Phi a\cong H'H\Phi a\circ 1$ in the bicategory $\T''(H'HFj,H'HFi)$.

$(iii)$  For any representation $F:I\to \T$, the 2-cell attached by
$$m=m_F:(1_\T)_*F\Rightarrow F$$ at any arrow $a:j\to i$ of
$I$ is the identity, that is, $ma =1_{Fa}$. For any pair of
composable arrows $k\overset{ b}\to j \overset{a}\rightarrow i$ and
any object $i$ of $I$, the corresponding invertible structure
3-cells (\ref{3phi}), $$\begin{array}{l}m_{a,b}:F_{a,b}\circ
(ma\otimes mb)\Rrightarrow m(ab)\circ ((1_\T)_*F)_{a,b},\\
m_i:m1_i\circ ((1_\T)_*F)_i\Rrightarrow F_i,\end{array}$$ are,
respectively, the canonical isomorphisms in the diagrams below.
$$
\xymatrix@R=20pt@C=20pt{ Fa\otimes Fb\ar@{=>}[r]^{1}\ar@{=>}[d]_{1\otimes 1}\ar@{}[rrd]|{ \cong}&Fa\otimes Fb\ar@{=>}[r]^{F_{a,b}}&F(ab)\ar@{=>}[d]^{1}\\
Fa\otimes Fb\ar@{=>}[rr]_{F_{a,b}}&&F(ab),}\hspace{0.4cm} \xymatrix{&1_{Fi}\ar@{}[d]|{ \cong}\ar@{=>}[ld]_{1}\ar@{=>}[rd]^{F_i}&\\1_{Fi}\ar@{=>}[r]_{F_i}&F1_i\ar@{=>}[r]_{1}&F1_i.}
$$

If $\Phi:F\Rightarrow G$ any 1-cell in $\brep(I,\T)$, then the
invertible naturality 2-cell $$m_\Phi: m_G\circ
((1_\T)_*\Phi))\Rrightarrow \Phi\circ m_F,$$ at any arrow $a:j\to i$
of $I$, is provided by the canonical isomorphism $1\circ \Phi a\cong
\Phi a\circ 1$ in the bicategory $\T(Fj,Fi)$.
\end{proof}

\subsection{Representations of free categories}

Let us now replace category $I$ above by a (directed)
graph $\mathcal{G}$. For any tricategory $\T$, there is a bicategory of {\em representations of $\mathcal{G}$ in $\T$}, denoted by
\begin{equation}\label{gph}\brep(\mathcal{G},\T),
\end{equation}
where a 0-cell  $f:\mathcal{G}\to \T$, consists of a pair of maps
that assign an object $fi$ to each vertex $i\in \mathcal{G}$ and a
1-cell $fa:fj\to fi$ to each edge $a:j\to i$ in $\mathcal{G}$,
respectively. A 1-cell $\phi:f\Rightarrow g$ may exist only if $f$
and $g$ agree on vertices, that is, $fi=gi$ for all $i\in
\mathcal{G}$; and then it consists of a map that assigns to each
edge  $a:j\to i$ in the graph a 2-cell $\phi a:fa\Rightarrow ga$ of
$\T$. And a 2-cell $m:\phi\Rrightarrow\psi$, for
$\phi,\psi:f\Rightarrow g$ two 1-cells as above, consists of a
family of 3-cells in $\T$, $ ma:\phi a\Rrightarrow \psi a$, one for
each arrow $a:j\to i$ in $I$. Compositions in
$\brep(\mathcal{G},\T)$ are defined in the natural way by the same
rules as those stated above for the bicategory of representations
of a category in a tricategory.

Suppose now that $I(\mathcal{G})$ is the free category generated by
the graph $\mathcal{G}$ \cite{mac}. Then, restriction to the basic
graph gives a 2-functor
\begin{equation}\label{res} R: \brep(I(\mathcal{G}),\T)\to \brep(\mathcal{G},\T),
\end{equation}
and we shall prove the following auxiliary statement to be used
later:

\begin{lemma}\label{rufc} Let $I=I(\mathcal{G})$ be the free category generated by a graph $\mathcal{G}$.
Then, for any tricategory $\T$, there are a homomorphism
\begin{equation}\label{ext}L:\brep(\mathcal{G},\T)\to   \brep(I,\T), \end{equation}
and a lax transformation
\begin{equation}\label{v} \v:LR\Rightarrow 1_{\brep(I,\T)},\end{equation}
such that the following facts hold:
\begin{enumerate}
\item[(a)] $RL=1_{\brep(\mathcal{G},\T)}$,\  $\v L=1_{L}$,\
$R\v=1_{R}$.
\item[(b)] The image of $L$ is contained into the sub-bicategory $\buhrep(I,\T)\subseteq \brep(I,\T)$.
\item[(c)]
The restricted homomorphisms of $L$ and $R$ establish biadjoint
biequivalences
\begin{eqnarray}\label{lrg} \xymatrix@C=20pt{\brep(\mathcal{G},\T)\ar@<3pt>[r]^{L}\ar@{}[r]|{\sim}&
 \ar@<3pt>[l]^{R}\bhrep(I,\T)},\\
\label{lrg2}
\xymatrix@C=20pt{\brep(\mathcal{G},\T)\ar@<3pt>[r]^{L}\ar@{}[r]|(.48){\sim}&
\ar@<3pt>[l]^{R}\buhrep(I,\T)},\end{eqnarray} whose respective
unit is the identity $1:1\Rightarrow RL$, the counit is given by
the corresponding restriction of $\v:LR\Rightarrow 1$, and whose
triangulators are the canonical modifications $1\cong 1\circ 1
=\v L\circ L1$ and $R\v\circ 1 R=1\circ 1\cong 1$, respectively.
\end{enumerate}
\end{lemma}
\begin{proof} To describe the homomorphism $L$, we shall use the following useful
construction: For any list $(t_0,\dots,t_p)$ of objects in the
tricategory $\T$, let
\begin{equation}\label{multen}
\xymatrix{\overset{_\mathrm{or}}\otimes:\mathcal{T}(t_1,t_0)\times\mathcal{T}(t_2,t_1)\times\cdots\times\mathcal{T}(t_p,t_{p-1})
\longrightarrow \T(t_p,t_0)}
\end{equation}
denote the homomorphism recursively defined as the composite
$$
\xymatrix{
\prod\limits_{i=1}^{p}\T(t_i,t_{i-1})=\T(t_1,t_0)\times \prod\limits_{i=2}^{p}\T(t_i,t_{i-1})
\ar[r]^-{1\times \overset{_\mathrm{or}}\otimes}&\T(t_1,t_0)\times
\T(t_p,t_2)\ar[r]^-{\otimes}&\T(t_p,t_0).}
$$
That is, $\overset{_\mathrm{or}}\otimes$ is the homomorphism
obtained by iterating composition in the tricategory, which acts on
$0$-cells, $1$-cells and $2$-cells of the product bicategory
$\prod_{i=1}^{p}\T(t_{i},t_{i-1})$ by the recursive formula
$$
\overset{_\mathrm{or}}\otimes(x_1,\dots,x_p)=\left\{\begin{array}{lll}x_1&\text{if}& p=1,\\[5pt]
x_1\otimes\big(\overset{_\mathrm{or}}\otimes(x_2,\dots,x_p)\big)&
\text{if}&p\geq 2.\end{array}\right. $$

$\bullet$ {\em{The definition of $L$ on $0$-cells} }.  The
homomorphism $L$ takes a  representation of the graph in the
tricategory, say $f:\mathcal{G}\to \T$, to the unitary homomorphic
representation of the free category
$$L(f)=F:I\to \T,$$ such that
\begin{equation}\label{lf0} Fi=fi, \hspace{0.3cm}\text{for any
vertex $i$ of $\mathcal{G}$ (= objects of
$I$)},\hspace{2cm}\end{equation} and associates to  strings $
a:\,a(p)\overset{ a_{p}}\longrightarrow \cdots
\overset{a_{2}}\longrightarrow a(1)\overset{ a_{1}}\longrightarrow
a(0) $ of adjacent edges in $\mathcal{G}$  the 1-cells of $\T$
\begin{equation}\label{lf1}\xymatrix{Fa=
\overset{_\mathrm{or}}\otimes(fa_1,\dots,fa_p):fa(p)\to
fa(0).}\end{equation} The structure 2-cells  $F_{a,b}:Fa\otimes Fb
\Rightarrow F(ab)$, for any pair of strings in the graph,
$a=a_1\cdots a_p$ as above and $ b=b_1\cdots b_q$ with $b(0)=a(p)$,
are canonically obtained from the associativity constraint in the
tricategory: first by taking $F_{a_1,b}=1_{F(a_1b)}$ and then,
recursively for $p>1$, defining $F_{a,b}$ as the composite
$$\xymatrix@C=35pt{F_{a,b}:\ Fa\otimes Fb
\overset{ \boldsymbol{a}}\Longrightarrow Fa_1\otimes(
Fa'\otimes Fb)\ar@{=>}[r]^-{1\otimes F_{a',b} } &F(ab),}
$$
where $a'=a_2\cdots a_p$ (whence $Fa=Fa_1\otimes Fa'$). And the structure
3-cells  $F_{a,b,c}$, for any three strings in the graph
$a$, $b$ and $c$ as above with $a(p)=b(0)$ and $b(q)=c(0)$, are the unique
isomorphisms constructed from the tricategory coherence 3-cells
$\boldsymbol{\pi}$. For a particular construction of these isomorphisms, we can
first take each $F_{a_1,b,c}$ to be the canonical isomorphism
$$
\xymatrix@R=20pt@C=10pt{(Fa_1\otimes Fb)\otimes
Fc\ar@{}@<-63pt>[d]|{
F_{a_1,b,c}:}\ar@{}[drrrr]|{ \cong}\ar@{=>}[d]_{1\otimes
1}\ar@{=>}[rrrr]^{\boldsymbol{a}}&&&&Fa_1\otimes
(Fb\otimes Fc)\ar@{=>}[d]^{1\otimes F_{b,c}}\\F(a_1b)\otimes
Fc\ar@{=>}[r]^-{ \boldsymbol{a}}& Fa_1\otimes (Fb\otimes
Fc)\ar@{=>}[rr]^(.6){1\otimes F_{b,c}}&
&F(a_1bc)&\ar@{=>}[l]_{1}Fa_1\otimes F(bc), }
$$
and then, recursively for $p>1$, take $F_{a,b,c}$ to be the 3-cell
canonically obtained from $F_{a_2\cdots a_p,b,c}$ by pasting the
diagram bellow, where, as above, we write $a'$ for $a_2\cdots a_p$.
$$
\xymatrix@C=-28pt@R=4pt{&(Fa\!\otimes\! Fb)\!\otimes\!
Fc\ar@{=>}[rrdd]^{ \boldsymbol{a}}
\ar@{=>}[ldd]_{ \boldsymbol{a}\otimes 1}&&\\
&&&\\
(Fa_1\!\otimes\! (Fa'\!\otimes\! Fb))\!\otimes\!
Fc\ar@{}[rrr]|{ { \boldsymbol{\pi}}
\Rrightarrow}\ar@{=>}[rd]^(.55){\boldsymbol{a}}\ar@{=>}[ddd]|{(1\otimes
F_{a',b})\otimes 1} \ar@{}@<-50pt>[ddd]_{\textstyle
F_{a,b,c}:}&&&Fa\!\otimes\! (Fb\!\otimes\!
Fc)\ar@{=>}[ldd]_{\boldsymbol{a}}\ar@{=>}[rdd]^(.55){1\otimes F_{b,c}}\\
& Fa_1\!\otimes\! ((Fa'\!\otimes\! Fb)\!\otimes\! Fc)
\ar@{=>}[rd]^(.6){1\otimes\boldsymbol{a}}\ar@{=>}[ddd]|{1\otimes(F_{a',b}\otimes 1)}&&\\
&&Fa_1\!\otimes\!( Fa'\!\otimes\! (Fb\!\otimes\! Fc))
\ar@{=>}[rdd]^{1\otimes(1\otimes F_{b,c})}\ar@{}[rr]|(.6){
\cong}
\ar@{}[ll]|(.75){ \cong}&&\hspace{0.5cm}Fa\!\otimes\! F(bc)\ar@{=>}[ldd]^{ \boldsymbol{a}}\\
F(ab)\!\otimes\! Fc\ar@{=>}[rd]_(.4){
\boldsymbol{a}}\ar@{}@<5pt>[rrr]_(.6)
{ { 1\otimes F_{a',b,c}}\Rrightarrow}&&&\\
&Fa_1\!\otimes\!(F(a'b)\!\otimes\! Fc)\ar@{=>}[rd]_{1\otimes
F_{a'b,c}}&&Fa_1\!\otimes\!(Fa'\!\otimes\! F(bc))
\ar@{=>}[ld]^{\ \ 1\otimes F_{a',bc}}\\
&& F(abc)& }
$$

Note that, since  all structure 2-cells $F_{a,b}$  are equivalences
in the corresponding hom-bicategories of $\T$ in which they lie, as
well as all the structure 3-cells $F_{a,b,c}$ are invertible, the so
defined unitary representation $F:I\to \T$ is actually a homomorphic
one; that is, $L(f)=F\in \buhrep(I,\T)\subseteq \brep(I,\T)$.

\vspace{0.2cm} $\bullet$ {\em{The definition of $L$ on $1$-cells} }.
Any 1-cell $\phi:f \Rightarrow g$ of $\brep(\mathcal{G},\T)$, is
taken by $L$
 to the 1-cell in  $\buhrep(I,\T)$
$$L(\phi)=\Phi:F\Rightarrow G,
$$
consisting of the 2-cells in $\T$
\begin{equation}\label{lph1}\xymatrix{\Phi a=
\overset{_\mathrm{or}}\otimes(\phi{a_1},\dots,\phi{a_p}):Fa
\Rightarrow Ga},\end{equation} attached to the strings  $ a=a_1\cdots a_p$ of
adjacent edges in the graph. The structure (actually invertible)
3-cells $\Phi_{a,b}$, for any pair of strings in the graph, $a$ and
$b$ with $b(0)=a(p)$ as above, are defined by induction on the
length of $a$ as follows: each $\Phi_{a_1,b}$ is the canonical
isomorphism
$$
\xymatrix@R=20pt{Fa_1\otimes Fb\ar@{=>}[r]^{1}\ar@{}[rd]|{
\cong}\ar@{}@<-60pt>[d]|{\textstyle \Phi_{a_1,b}:}\ar@{=>}[d]_{\Phi
a_1\otimes \Phi b}& F(a_1b)\ar@{=>}[d]^{\Phi(a_1b)=\Phi{a_1}\otimes
\Phi{b}}\\Ga_1\otimes Gb\ar@{=>}[r]^{1}&G(a_1b),}
$$
and then, for $p>1$, each $\Phi_{a,b}$ is recursively obtained from
$\Phi_{a',b}$, where $a'=a_2\cdots a_p$, by pasting
$$
\xymatrix@R=25pt{Fa\otimes Fb \ar@{}@<35pt>[d]|{ \cong}
\ar@{}@<-60pt>[d]|{\textstyle \Phi_{a,b}:} \ar@{=>}[d]_{\Phi
a\otimes\Phi b} \ar@{=>}[r]^-{ \boldsymbol{a}}&
Fa_1\otimes (Fa'\otimes Fb) \ar@{}@<55pt>[d]|{
{ 1\!\otimes\! \Phi_{a'\!,b}}\Rrightarrow}
\ar@{=>}[d]|{\Phi{a_1} \otimes (\Phi{a'}\otimes \Phi{b})}
\ar@{=>}[r]^-{1\otimes F_{a'\!,b}}&F(ab)\ar@{=>}[d]^{\Phi(ab)=\Phi{a_1}\otimes (\Phi{a'}\otimes\Phi{b})}\\
Ga\otimes Gb\ar@{=>}[r]_-{ \boldsymbol{a}}& Ga_1\otimes
(Ga'\otimes Gb)\ar@{=>}[r]_-{1\otimes G_{a'\!,b}}&G(ab).}
$$

Note that, since  all structure 3-cells $\Phi_{a,b}$  are
invertible, the thus defined unitary 1-cell $\Phi$ of $\burep(I,\T)$
is actually a 1-cell of $\buhrep(I,\T)$, that is, $L(\phi)\in
\buhrep(I,\T)$.

\vspace{0.2cm} $\bullet$ {\em{The definition of $L$ on $2$-cells} }.
For $\phi,\, \psi :f \Rightarrow g$, any two 1-cells in
$\brep(\mathcal{G},\T)$, the homomorphism $L$ on a 2-cell
$m:\phi\Rrightarrow \psi$ gives the 2-cell of $\buhrep(I,\T)$
$$L(m)=M:\Phi\Rrightarrow\Psi,
$$
consisting of the 3-cells in $\T$
\begin{equation}\label{lm1}\xymatrix{Ma=
\overset{_\mathrm{or}}\otimes(m{a_1},\dots,m{a_p}):\Phi a
\Rightarrow \Psi a},\end{equation} for the  strings  $ a=a_1\cdots a_p$  of
adjacent edges in the graph $\mathcal{G}$.

\vspace{0.2cm}$\bullet$ {\em{The structure constraints of $L$}}.
If $\phi:f\Rightarrow g$ and $\psi: g\Rightarrow h$  are 1-cells in
$\burep(\mathcal{G},\T)$,  then the structure isomorphism in $\burep(I,T)$
$$L_{\psi,\phi}
: L(\psi)\circ L(\phi) \cong L(\psi\circ \phi),$$ at each string $ a=a_1\cdots a_p$ as above, is recursively
defined as the identity 3-cell
on $\psi{a_1}\circ \phi{a_1}$ if $p=1$, while, for $p>1$,
$L_{\psi,\phi}\,{a}:L(\psi){a}\circ L(\phi){a}\Rrightarrow
L(\psi\circ\phi){a}$ is obtained from $L_{\psi,\phi}\,{a'}$, where
$a'=a_2\cdots a_p$, as the composite
\begin{multline}\nonumber L(\psi)a\circ L(\phi)a=(\psi{a_1}\otimes L(\psi){a'})
\circ (\phi{a_1}\otimes L(\phi){a'})\cong
(\psi{a_1}\circ\phi{a_1})\otimes (L(\psi){a'}\circ L(\phi){a'})\\
\overset{1\otimes L_{\psi,\phi}{a'}}\Rrightarrow
(\psi{a_1}\circ\phi{a_1})\otimes
L(\psi\circ\phi){a'}=L(\psi\circ\phi)a.\end{multline}

And, similarly, the structure isomorphism
$$L_f:1_{L(f)}\cong L(1_f) $$
consists of the 3-cells $L_f\,{a}:1_{L(f)a}\Rrightarrow L(1_f){a}$,
where $L_f{a_1}=1:1_{fa_1}\Rrightarrow 1_{fa_1}$ and, for $p>1$,
$L_f{a}$ is recursively obtained from $L_f{a'}$, $a'=a_2\cdots a_p$,
as the composite
$$ 1_{L(f)a}=1_{fa_1\otimes L(f)a'}\cong 1_{fa_1}\otimes
1_{L(f)a'}\overset{1\otimes L_f{a'}}\Rrightarrow 1_{fa_1}\otimes
L(1_f){a'}=L(1_f)a.$$
This completes the description of the homomorphism $L$.

\vspace{0.2cm} $\bullet$ {\em{The definition of lax transformation
$\v$}}. The component of this lax transformation at a representation
$F:I\to \T$, $\v=\v(F):LR(F)\Rightarrow F$, is defined on identities
by
$$
\v1_i=F_i:1_{Fi}\Rightarrow F1_i,
$$
for any vertex $i$ of $\mathcal{G}$, and it associates to each
string of adjacent edges in the graph $a=a_1\cdots a_p$ the 2-cell
\begin{equation}\label{va}\xymatrix{\v a:\overset{_\mathrm{or}}\otimes(Fa_1,\dots,Fa_p)\Rightarrow
Fa,}\end{equation} which is given by taking $\v{ a_1}=1_{Fa_1}$ if
$p=1$, and then, recursively for $p>1$, by taking
 $\v{a}$ as the composite
$$\xymatrix@C=35pt{\v a:\ \overset{_\mathrm{or}}\otimes(Fa_1,\dots,Fa_p)
\ar@{=>}[r]^-{1\otimes \v{a'} } &Fa_1\otimes
Fa'\ar@{=>}[r]^-{F_{a_1,a'}}&Fa,}
$$
where $a'=a_2\cdots a_p$.

The structure 3-cell
\begin{equation}\label{vab}\v_{a,b}: F_{a,b}\circ (\v a\otimes \v b)\Rrightarrow \v(ab)\circ
LR(F)_{a,b},\end{equation} for any pair of composable morphisms in
$I$, is defined as follows: when $a=1_i$ or $b=1_j$ are identities,
then $\v_{1_i,b}$ and $\v_{a,1_j}$ are respectively given  by pasting
the diagrams
$$
\xymatrix@C=0pt@R=15pt{ &1_{Fi}\otimes LR(F)b\ar@{=>}[rr]^{\boldsymbol{l}}
\ar@{=>}[rd]^{1\otimes \v b}\ar@{=>}[dd]_{F_i\otimes\v b}&&LR(F)b\ar@{=>}[dd]^{\v b}\\
 \v_{1_i,b}:&\ar@{}[r]|(.4){\cong}&1_{Fi}\otimes Fb\ar@{=>}[rd]^{\boldsymbol{l}}
 \ar@{=>}[ld]_{F_i\otimes 1}\ar@{}[ru]|{\cong}\ar@{}[d]|{{\widehat{F}_b}\Rrightarrow}& \\
&F1_i\otimes Fb\ar@{=>}[rr]_{F_{1_i,b}}&&Fb, } \hspace{0.3cm}
\xymatrix@C=0pt@R=15pt{ &LR(F)a\otimes 1_{Fj}\ar@{=>}[rr]^{\boldsymbol{r}^\bullet}
\ar@{=>}[rd]^{\v a\otimes 1}\ar@{=>}[dd]_{\v a\otimes F_j}&&LR(F)a\ar@{=>}[dd]^{\v a}\\
\v_{a,1_j}:&\ar@{}[r]|(.4){\cong}&Fa\times 1_{Fj}\ar@{=>}[rd]^{\boldsymbol{r}^\bullet}
\ar@{=>}[ld]_{1\otimes F_j}\ar@{}[ru]|{\cong}\ar@{}[d]|{{\widetilde{F}_a}\Rrightarrow}& \\
&Fa\otimes F1_j\ar@{=>}[rr]_{F_{a,1_i}}&&Fa, }
$$
and, for strings  $a$ and $b$ in the graph with $b(q)=a(0)$, $\v_{a,b}$
is defined by induction on the length of $a$ by taking
$\v_{a_1\!,b}$ to be the canonical isomorphism
$$
\xymatrix@C=15pt{Fa_1\otimes LR(F)b\ar@{=>}[d]_-{1\otimes \v b}\ar@{}[rd]|(.5){
\cong}\ar@{}@<-2cm>[d]|{ \textstyle\v_{a_1,b}:}\ar@{=>}[r]^-{1}&
LR(F)(a_1b)\ar@{=>}[d]^{\v(a_1b)=F_{a_1\!,b}\circ (1\otimes \v b)}\\
Fa_1\otimes Fb\ar@{=>}[r]_-{F_{a_1,b}}&F(a_1b),}
$$
and then, for $p>1$, $\v_{a,b}$ is recursively obtained from
$\v_{a'\!,b}$, where $a'=a_2\cdots a_p$, by pasting
$$
\xymatrix@C=40pt{LR(F)a\otimes LR(F)b \ar@{=>}[r]^-{
\boldsymbol{a}}\ar@{=>}[d]_{(1\otimes \v a')\otimes\v b}
\ar@{}@<-2.5cm>[dd]|-{\textstyle \v_{a,b}:}&Fa_1\!\otimes\!(LR(F)a'\!\otimes\! LR(F)b)
\ar@{}[rd]|{ { 1\otimes \v_{a'\!,b}}\Rrightarrow}\ar@{=>}[d]|{1\otimes (\v a'\otimes\v b)}
\ar@{=>}[r]^-{1\otimes LR(F)_{a'\!,b}}\ar@{}[ld]|{ \cong}&LR(F)(ab)\ar@{=>}[d]^{1\otimes \v (a'b)}\\
(Fa_1\!\otimes\! Fa')\!\otimes\! Fb\ar@{}[rrd]|{ { F_{a_1\!,a'\!,b}}\Rrightarrow}
\ar@{=>}[r]^-{\boldsymbol{a}}\ar@{=>}[d]_{F_{a_1\!,a'}\otimes 1}&Fa_1\!\otimes\!(Fa'\!\otimes \!Fb)
\ar@{=>}[r]_{1\otimes F_{a'\!,b}}&Fa_1\otimes F(a'b)\ar@{=>}[d]^{F_{a_1\!,a'b}}\\
Fa\otimes Fb\ar@{=>}[rr]^(.4){F_{a,b}}&& F(ab). }
$$
And the structure 3-cell
\begin{equation}\label{vi} \v_i:\v 1_i\circ LR(F)_i\Rrightarrow F_i,\end{equation}
for any vertex $i$ of the graph, is the canonical isomorphism
$F_i\circ 1\cong F_i$.

The naturality component of $\v$ at a 1-cell $\Phi:F\Rightarrow G$
in $\burep(I,\T)$,
\begin{equation}\label{vphi}\v_\Phi:\v(G)\circ LR(\Phi)\Rrightarrow \Phi\circ \v(F),\end{equation}
is given on identities by
$$\xymatrix@C=30pt@R=25pt{ \ar@{}@<-25pt>[d]|{\v_\Phi
1_i:}1_{Fi}\ar@{=>}[r]^{F_i}\ar@{=>}[d]_{1}\ar@{=>}[rd]|{G_i}&F1_i\ar@{=>}[d]^{\Phi 1_i}\\
1_{Fi}\ar@{=>}[r]_{G_i}\ar@{}[ru]|(.3){\cong}|(.7){\overset{ \Phi_i}\cong}&G1_i,}
$$
and it is recursively defined at each string of adjacent edges in
the graph $a=a_1\cdots a_p$, by the 3-cells $\v_{\Phi}a$ where, if
$p=1$, then
$$
\xymatrix@C=20pt@R=20pt{F{a_1}\ar@{=>}[r]^{1}\ar@{}[rd]|{
\cong}\ar@{}@<-45pt>[d]|{\textstyle
\v_{\Phi}a_1:}\ar@{=>}[d]_{\Phi{a_1}}&
Fa_1\ar@{=>}[d]^{\Phi a_1}\\
Ga_1\ar@{=>}[r]^{1}&Ga_1,}
$$
is the canonical isomorphism, and then, when $p>1$, $\v_{\Phi}a$ is
 obtained from $\v_{\Phi}a'$, where $a'=a_2\cdots a_p$,
by pasting
$$
\xymatrix@R=25pt@C=35pt{Fa_1\otimes LR(F)a'
\ar@{}@<45pt>[d]|(.4){ { 1\otimes
\v_{\Phi}a'}\Rrightarrow} \ar@{}@<-80pt>[d]|{ \textstyle\v_{\Phi}a:}
\ar@{=>}[d]_{\Phi a_1\otimes LR(\Phi)a'} \ar@{=>}[r]^-{1\otimes
\v(F)a'}& Fa_1\otimes Fa' \ar@{}@<40pt>[d]|(.4){
{ \Phi_{a_1\!,a'}}\Rrightarrow} \ar@{=>}[d]|{\Phi
a_1\otimes \Phi a'}
\ar@{=>}[r]^-{F_{a_1\!,a'}}&Fa\ar@{=>}[d]^{\Phi a}\\
Ga_1\otimes LR(G)a'\ar@{=>}[r]^-{1\otimes \v(G)a'}& Ga_1\otimes
Ga'\ar@{=>}[r]^-{G_{a_1\!,a­}}&Ga.}
$$

\vspace{0.2cm} We are now ready to complete the proof of the lemma.
That the equalities $RL=1$, $\v L=1$, and $R\v=1$ hold only requires
a straightforward verification, and then part (a) follows. Moreover,
(b) has already been shown by construction of the homomorphism $L$.

$\bullet$ {\em The proof of}  (c).  Suppose that $F:I\to \T$ is any
homomorphic representation. This means that all structure 2-cells
$F_{a,b}$ and $F_i$ are equivalences, and 3-cells $F_{a,b,c}$,
$\widehat{F}_a$, and $\widetilde{F}_a$ are isomorphisms in the
hom-bicategories of $\T$ in which they lie. Then, directly from the
construction given, it easily follows that all the 2-cells $\v(F)a$
in (\ref{va}) are equivalences in the corresponding
hom-bicategories, and that all the 3-cells $\v(F)_{a,b}$ in
$(\ref{vab})$, and $\v_i$ in $(\ref{vi})$ are invertible. Hence,
each $\v(F):LR(F)\Rightarrow F$, for $F:I\to \T$ any homomorphic
representation, is an equivalence in the bicategory $\bhrep(I,\T)$.
Moreover, if $\Phi:F\Rightarrow G$ is any 1-cell  in $\bhrep(I,\T)$,
so that every 3-cell $\Phi_{a,b}$ is an isomorphism, then  we see
that the component (\ref{vphi}) of $\v$ at $\Phi$ consists only of
invertible 3-cells $\v_\Phi a$, whence $\v_\Phi$ is invertible
itself. Therefore, when  $\v$ is restricted to $\bhrep(I,\T)$, it
actually gives a pseudo-equivalence between $LR$ and $1$, the
identity homomorphism on the bicategory $\bhrep(I,\T)$. The claimed
biadjoint biequivalence (\ref{lrg}) is now an easy consequence of
all the already parts proved. Finally, it is clear that the
biadjoint biequivalence (\ref{lrg}) gives by restriction the
biadjoint biequivalence $(\ref{lrg2})$.
\end{proof}

\section{The Grothendieck nerve of a tricategory}

Let us briefly recall that it was Grothendieck
\cite{groth} who first associated a simplicial set
\begin{equation}\label{ncat}\ner C:\Delta^{\mathrm{op}}\to \set\end{equation} to a small
category $C$, calling it its \emph{nerve}. The set of $p$-simplices
$$
\ner_{p}C = \bigsqcup_{(c_0,\ldots,c_p)}\hspace{-0.15cm}
C(c_1,c_0)\times C(c_2,c_1)\times\cdots\times C(c_p,c_{p-1})
$$ consists of length $p$ sequences of composable morphisms in $C$.  Geometric realization
of its nerve is the {\em classifying space} of the category, $\class
C$.  A main result here shows how the Grothendieck nerve construction
for categories rises to tricategories.

 \subsection{The pseudo-simplicial bicategory nerve of a tricategory} When a tricategory $\T$ is strict, that is, a 3-category,  then the nerve construction (\ref{ncat})
  actually works by giving a simplicial 2-category (see Example \ref{ex3cat}).
  However, for an arbitrary tricategory, the device is more complicated since the compositions of cells in a tricategory is in general
  not associative and not unitary (which is crucial for the simplicial structure in the construction of
  $\ner\T$ as above), but it is only so up to coherent isomorphisms. This  `defect'
  has the effect of forcing one to deal with the classifying space of a nerve of $\T$ that is not
   simplicial but only up to coherent isomorphisms, that is, a {\em pseudo-simplicial bicategory} as  stated in the theorem below.
   Pseudo-simplicial bicategories, and the tricategory they form (whose 1-cells are pseudo-simplicial homomorphisms, 2-cells pseudo-simplicial transformations, and 3-cells pseudo-simplicial modifications) are treated in \cite{c-c-g},  to which we refer the reader.

\begin{theorem}\label{pstc} Any tricategory  $\T$
defines a normal pseudo-simplicial bicategory, called the  {\em
nerve of the tricategory},
\begin{equation}\label{bsnt}\ner\T=(\ner\T,\chi,\omega):\Delta^{\!{\mathrm{op}}}\
\to \bicat,
\end{equation}
whose bicategory of $p$-simplices, for $p\geq 1$, is {\em
\begin{equation}\label{bpsnt} \ner_{p}\mathcal{T} = \bigsqcup_{(t_0,\ldots,t_p)\in
\mbox{\scriptsize Ob}\mathcal{T}^{p+1}}\hspace{-0.3cm}
\mathcal{T}(t_1,t_0)\times\mathcal{T}(t_2,t_1)\times\cdots\times\mathcal{T}(t_p,t_{p-1}),
\end{equation}}

\noindent{and}  {\em $\ner_0\T=\mbox{0b}\T$}, as a discrete
bicategory. The face and degeneracy homomorphisms are defined on
$0$-cells, $1$-cells and $2$-cells of $\ner_{p}\mathcal{T}$ by the
ordinary formulas
\begin{equation}\label{fade}\begin{array}{l} d_i(x_1,\dots,x_p)=\left\{
\begin{array}{lcl}
\hspace{-0.2cm}  (x_2,\dots,x_p) &\text{ if }& i=0 ,  \\[4pt]
\hspace{-0.2cm}  (x_1,\dots,x_i\otimes x_{i+1},\dots,x_p) & \text{ if } & 0<i<p , \\[4pt]
\hspace{-0.2cm}  (x_1,\dots,x_{p-1}) & \text{ if }& i=p,
\end{array}\right.\\[20pt]
s_i(x_1,\dots,x_p)=(x_1,\dots,x_i,1,x_{i+1},\dots, x_p).
\end{array}
\end{equation}
Indeed, if $a:[q]\to [p]$ is any map in the simplicial category
$\Delta$, then the associated homomorphism $\ner_a\T:\ner_{p}\T
\to\ner_{q}\T$ is induced by  the composition
${\mathcal{T}(t',t)\times
\mathcal{T}(t'',t')\overset{\hspace{-4pt}\otimes}\to
\mathcal{T}(t'',t)}$ and unit $1_t:1\to \mathcal{T}(t,t)$
homomorphisms. The structure pseudo-equivalences
\begin{equation}\label{chis}\begin{array}{l}\xymatrix @C=10pt{\ner_p\T
\ar@/^0.9pc/[rr]^{ \ner_{\!b}\!\T\ \ner_{\!a}\!\T} \ar@/_0.9pc/[rr]_{ \ner_{\!ab}\!\T} \ar@{}[rr]|{
\Downarrow^{\chi_{a,b}}}& &\ner_n\T,}\end{array}\end{equation} for each pair of
composable maps  $[n]\overset{b}\to [q]\overset{a}\to [p]$ in
$\Delta$, and the invertible modifications
\begin{equation}\label{ws}\begin{array}{l}\xymatrix@C=40pt{\ner_{\!c}\hspace{-1pt}\T\ \ner_{\!b}\hspace{-1pt}\T\
 \ner_{\!a}\hspace{-1pt}\T\ar@{=>}[d]_{ \chi_{b,c}
\ner_{\!a}\hspace{-1pt}\T}\ar@{}[rd]|{ {
\omega_{a,b,c}}\,\Rrightarrow}\ar@{=>}[r]^{
\ner_{\!c}\hspace{-1pt}\T\,\chi_{a,b}}& \ner_c\hspace{-1pt}\T \ \ner_{\!ab}\hspace{-1pt}\T
\ar@{=>}[d]^{ \chi_{ab,c}}\\  \ner_{\!bc}\hspace{-1pt}\T
\  \ner_{\!a}\hspace{-1pt}\T\ar@{=>}[r]_{
\chi_{a,bc}}& \ner_{\!abc}\hspace{-1pt}\T,}\end{array}\end{equation} respectively associated to triplets of composable
arrows $[m]\overset{c}\to [n]\overset{b}\to [q] \overset{a}\to [p]$,
canonically arise all from the structure pseudo equivalences and
modifications data of the tricategory.
\end{theorem}
We shall prove Theorem \ref{pstc} simultaneously with Theorem
\ref{psth} below, which states the basic properties concerning 
the behavior of the Grothendieck nerve construction, $\T\mapsto
\ner\T$, with respect to trihomomorphisms of tricategories.

\begin{theorem}\label{psth} (i) Any trihomomorphism between tricategories $H:\T\to\T'$ induces
a normal pseudo-simplicial homomorphism \begin{equation}\ner H=(\ner
H,\theta,\Pi):\ner\T\to\ner \T',\end{equation} which, at any integer
$p\geq 0$, is the evident homomorphism
$\ner_pH:\ner_p\T\to\ner_p\T'$ defined on any cell $(x_1,\dots,x_p)$
of $\ner_p\T$ by \begin{equation}\ner_pH(x_1,\dots,x_p)=(Hx_1,\dots,
Hx_p).\end{equation}

The structure pseudo-equivalence
\begin{equation}\label{thet}\begin{array}{l}\xymatrix@R=15pt{\ner_p\T
\ar@{}[rd]|{\theta_a\,\Rightarrow}\ar[r]^{\ner_{\!p}H}\ar[d]_{\ner_{\!a}\!\T}&\ner_p\T'
\ar[d]^{\ner_{\!a}\!\T'}\\\ner_q\T\ar[r]_{\ner_{\!q}H}&\ner_q\T',}\end{array}\end{equation} for each  map
$a:[q]\to [p]$ in $\Delta$, and the invertible modifications
\begin{equation}\label{piH} \begin{array}{l}\xymatrix@R=8pt@C=30pt{\ner_{\!n}H\
\ner_{\!b}\T\ \ner_{\!a}\T
\ar@2{->}[dd]_{ \theta \ner_{\!a}\!\T} \ar@{=>}[rr]^{ \ner_{\!n}\!H\chi} &
&\ner_{\!n}H\ \ner_{\!ab}\T\ar@{=>}[dd]^{\theta}
\\ \ar@{}[rr]|(0.5){{\Pi_{a,b}}\,\Rrightarrow} && \\
\ner_{\!b}\T'\ \ner_{\!q}H\ \ner_{\!a}\T\ar@{=>}[r]_{ \ner_{\!b}\!\T' \theta}&
\ner_{\!b}\T'\ \ner_{\!a}\T'\ \ner_{\!p}H\ar@{=>}[r]_{\chi \ner_{\!p}H}&
\ner_{\!ab}\T'\ \ner_{\!p}H,}\end{array}
\end{equation}
 respectively associated to pairs of composable
arrows $[n]\overset{b}\to [q] \overset{a}\to [p]$, canonically arise
all from the structure pseudo equivalences and modifications data of
the trihomomorphism $H$ and the involved tricategories $\T$ and
$\T'$.

 (ii) For any
pair of composable trihomomorphisms $H:\T\to \T'$ and $H':\T'\to
\T''$, there is a pseudo-simplicial pseudo-equivalence \begin{equation}\label{trucom}\ner H'\
\ner H\Rightarrow \ner (H'H).\end{equation}

(iii) For any tricategory $\T$, there is a pseudo-simplicial
pseudo-equivalence \begin{equation}\label{truni} \ner 1_\T\Rightarrow 1_{\ner\!\T}.\end{equation}

\end{theorem}

\begin{proof}[Proof of Theorems $\ref{pstc}$ and $\ref{psth}$] Let us note
that, for any integer $p\geq 0$, the category $[p]$ is free on the
graph \begin{equation}\label{graphp}\mathcal{G}_p=(p\to  \cdots\to
1\to 0),\end{equation} and that
$\ner_{p}\T=\brep(\mathcal{G}_p,\T)$. The existence of a biadjoint biequivalence
\begin{equation}\label{plrg}L_p\dashv R_p: \ner_p\T \rightleftarrows \bhrep([p],\T)\end{equation}
 follows from Lemma
\ref{rufc},
where $R_p$ is the 2-functor defined by
restricting to the basic graph $\mathcal{G}_p$ of the category
$[p]$, such that $R_pL_p=1$, whose unity is the identity, and whose
counit $\v_p:L_pR_p\Rightarrow 1$ is a pseudo-equivalence satisfying
the equalities $\v_p L_p=1$ and $R_p\v_p=1$.

Then, if $a:[q]\to [p]$ is any map in the simplicial category, the
associated homomorphism $\ner_{\!a}\T:\ner_p\T\to\ner_q\T$, is
defined to be the composite
\begin{equation}\label{astar} \begin{array}{l}\xymatrix@R=20pt{ \ner_p\T\ar@{.>}[r]^{\ner_{\!a}\T}\ar[d]_{L_p}
&\ner_q\T\\
\bhrep([p],\T)\ar[r]^{a^*}&\bhrep([q],\T). \ar[u]_{R_q}}
\end{array}\end{equation}

 Observe that, thus
defined, the homomorphism $\ner_{\!a}\T$ maps the component
bicategory of $\ner_{p}\T$ at $(t_0,\dots,t_p)$  into the component
at $(t_{a(0)},\dots,t_{a(q)})$ of $\ner_{q}\T$, and it acts on
$0$-cells, $1$-cells, and $2$-cells of $\ner_{p}\mathcal{T}$ by the
formula
$$\ner_{\!a}\!\T(x_1,\dots,x_p)=(y_1,\dots,y_q)$$ where, for $0\leq k< q$, (see $(\ref{multen})$ for the definition of $\overset{_\mathrm{or}}\otimes$)
\begin{equation}\label{fa}
y_{k+1}=\left\{\begin{array}{lll} \xymatrix{\hspace{-4pt}\overset{_\mathrm{or}}\otimes(x_{a(k)+1},\ldots, x_{a(k+1)})}& \text{\em if} & a(k)<a(k+1),\\[6pt]
1& \text{\em if} & a(k)=a(k+1).\end{array}\right.
\end{equation}
Whence, in particular, the formulas (\ref{fade}) for the face and
degeneracy homomorphisms.

The pseudo natural equivalence $(\ref{chis})$ is
   \begin{equation}\label{chi}\xymatrix@C=35pt{\ner_{\!b}\!\T\ \ner_{\!a}\!\T\!= R_n
   b^*L_qR_q
    a^* L_p\ar@{=>}[rr]^-{\chi_{a,b}=R_nb^*\v_q a^*L_p}
    &&R_n b^*a^*L_p= R_n (ab)^*L_p=\ner_{\!ab}\T,}\end{equation}
and the invertible modification $(\ref{ws})$ is
\begin{equation}\label{omega} \omega_{a,b,c}=R_mc^*\omega'_ba^*L_p,\end{equation}
where $\omega'_b$ is the canonical modification
\begin{equation}\label{2vec}\begin{array}{l}
\xymatrix@C=45pt@R=18pt{L_nR_n b^*L_qR_q\ar@{=>}[r]^(.55){L_nR_n b^*\v_q}\ar@{=>}[d]_{\v_n b^*L_qR_q}
\ar@{}[rd]|{ \overset{ (\ref{4})}\Rrightarrow}\ar@{}@<-2cm>[d]|-{\omega'_b:}
&L_nR_n b^*\ar@{=>}[d]^{\v_n b^*}
\\ b^*L_qR_q\ar@{=>}[r]_{b^*\v_q}&b^*.}\end{array}
\end{equation}

Thus defined, $\ner\T$ is actually a normal pseudo-simplicial
bicategory. Both coherence conditions for $\ner\T$ (i.e., conditions
$(\mathbf{CC1})$ and $(\mathbf{CC2})$ in \cite{c-c-g}, with the
modifications $\gamma$ and $\delta$ the unique unity coherence
isomorphisms $1\circ 1\cong 1$) follow from the equalities $R_pL_p=1$, $\v_p L_p=1$, and
$R_p\v_p=1$, by employing the useful Fact
\ref{fpre1} below. This proves Theorem \ref{pstc}.

\begin{fact}\label{fpre1} Let $\alpha:F\Rightarrow F':\b\to \c$ be a lax transformation between
homomorphisms of bicategories. Then, for any $2$-cell in $\b$
$$\begin{array}{l}\label{e1.1}\xymatrix@C=10pt@R=0pt{
&x_1\ar[r]&\cdots\ar[r]&x_n\ar[dr]^{a_n}&\\ x\ar[ur]^{ a_0}\ar[rd]_{b_0}\ar@{}[rrrr]|{\Downarrow u}&&&&x',\\
&x'_1\ar[r]&\cdots\ar[r]&x'_m\ar[ur]_{
b_m}&}\end{array}$$ the following equality holds:
$$
\xymatrix@C=8pt@R=0pt{&&&&&\ar@{}@<2pt>[ddddd]^(.7){ =}~\hspace{0.3cm}&&&&&&&&&\\
&Fx_1\ar[r]&\cdots\ar[r]&Fx_n\ar[rd]^{Fa_n}& &&& &Fx_1\ar[r]\ar[ddd]&\cdots\ar[r]&Fx_n\ar[rd]^{Fa_n}
\ar[ddd]
& \\
Fx\ar[ru]^{Fa_0}\ar[rd]^{Fb_0}\ar[ddd]_{\alpha x}\ar@{}@<6pt>[dddd]^{ \Downarrow\!\alpha_{b_0}}
\ar@{}[rrrr]|{\Downarrow Fu}&
&
&\ar@{}@<6pt>[dddd]^{\Downarrow\!\alpha_{b_m}}&
Fx'\ar[ddd]\ar@{}@<-2pt>[ddd]^(.5){\alpha x'}&&&
Fx\ar[ru]^{F\!a_0}\ar[ddd]\ar@{}@<0.5pt>[ddd]_(.5){\alpha x}
\ar@{}@<6pt>[dddd]^(.25){ \Downarrow\!\alpha_{a_0}}&
&&&Fx'\ar[ddd]^{\alpha x'}
\ar@{}@<-30pt>[dddd]^(.25){\Downarrow\!\alpha_{a_n}}\\
&Fx'_1\ar[r]\ar[ddd]&\cdots\ar[r]&Fx'_m\ar[ru]^{Fb_m}\ar[ddd]& &&& &&\cdots&&\\
&&\cdots&&&&&&F'x_1\ar[r]&\cdots\ar[r]&F'x_n\ar[rd]_{F'a_n}& \\
F'x\ar[rd]_{F'b_0}&&&&F'x'&&&F'x\ar[ru]_{F'a_0}\ar[rd]_{F'b_0}\ar@{}[rrrr]|{\Downarrow F'u}&&&&F'\!x'.\\
&F'x'_1\ar[r]&\cdots\ar[r]&F'x'_m\ar[ru]_{F'b_m}&&&&&F'x'_1\ar[r]&\cdots\ar[r]&F'x'_m\ar[ru]_{F'b_m}&
}
$$
\end{fact}
And when it comes to Theorem \ref{psth}, first, let us  note
that the homomorphisms $\ner_pH:\ner_p\T\to\ner_p\T'$, $p\geq 0$,
make commutative the diagrams
$$
\xymatrix@R=20pt{\ner_p\T\ar[r]^{\ner_pH}\ar[d]_{L=L_p^\T}&\ner_p\T'\\ \bhrep([p],\T)\ar[r]^{H_*}&\bhrep([p],\T')
\ar[u]_{R=R_p^{\T'}},}
$$
where  $H_*$ is the induced homomorphism by the trihomomorphism
$H:\T\to\T'$ (see Lemma \ref{transtri} $(i)$). Then, the
pseudo-equivalence $(\ref{thet})$, $\theta_a$, is provided by the
pseudo-equivalences $\v:LR\Rightarrow 1$ and their adjoint
quasi-inverses $\v^\bullet:1\Rightarrow LR$ (which we can choose
such that $R\v^\bullet=1$ and $\v^\bullet L=1$); that is, $ \theta_a
= Ra^*\v^\bullet H_*L\circ RH_*LRa^*L$,
$$
\xymatrix@C=-30pt@R=20pt{&
RH_*a^*L=Ra^*H_*L\ar@{=>}[rd]\ar@{}@<10pt>[rd]|(.6){\ Ra^*\v^\bullet H_*L}&\\
\ner_qH\ \ner_{\!a}\T=RH_*LRa^*L\ar@{=>}[ru]\ar@{}@<12pt>[ru]|(.4){RH_*\v a^*L}\ar@{.>}[rr]^{\theta_a} &
&Ra^*LRH_*L= \ner_{\!a}\T'\ \ner_pH.}
$$

And, for $[n]\overset{b}\to [q]\overset{a}\to [p]$, any two
composable arrows  of $\Delta$, the  structure invertible
modification $(\ref{piH})$, $\Pi_{a,b}$, is the modification
obtained by pasting the diagram
$$
\xymatrix@C=32pt@R=30pt{RH_*LRb^*LRa^*L \ar@{}@<-3pt>[rrrrd]^{\overset{(\ref{4})}\Rrightarrow}
\ar@{=>}[d]_{RH_*\v b^*LRa^*L}\ar@{=>}[rrrr]^-{RH_*LRb^*\v a^*L}
&&&&RH_*RLb^*a^*L\ar@{=>}[d]^{RH_*\v b^*a^*L}\\
RH_*b^*LRa^*L\ar@{}@<3pt>[rrd]_{\overset{(\ref{4})}\Rrightarrow}
\ar@{=>}[rrrr]^{Rb^*H_*\v a^*L}\ar@{=>}[d]_{Rb^*\v^\bullet H_*LR a^*L}
&&&&RH_*b^*a^*L\ar@{=>}[ld]^1\ar@{=>}@/_1.5pc/[lllldd]_{Rb^*\!\v^\bullet\! H_*\!a^*\!L\hspace{10pt} }
\ar@{=>}[d]^{Rb^*\!a^*\!\v^\bullet\! H_*\!L}
\\Rb^*LRH_*LRa^*L\ar@{=>}[d]_{Rb^*LRH_*\v a^*L}&&&Rb^*a^*H_*L
\ar@{=>}[r]_{Rb^*\!a^*\!\v^\bullet\! H_*\!L}&Rb^*a^*LRH_*L\ar@{}[lu]_(.3){\cong} \\
Rb^*LRH_*a^*L\ar@{=>}[rrrr]^(.6){Rb^*LRa^*\v^\bullet H_*L}\ar@{=>}[urrr]_(.5){\ \  Rb^*\v a^*H_*L}
\ar@{}@<-24pt>[rrruu]^(.6){\overset{(A)}\cong}&&&&
Rb^*LRa^*LRH_*L\ar@{=>}[u]_{Rb^*\!\v a^*\!LRH_*\!L}\ar@{}[lu]|{ \overset{(\ref{4})}\Rrightarrow}
}
$$
where the isomorphism labelled $(A)$ is given by the adjunction
invertible modification  $\v\circ \v^\bullet\cong 1$.

The coherence conditions for $\ner H:\ner\T \to\ner\T'$  (i.e.,
conditions $(\mathbf{CC3})$ and $(\mathbf{CC4})$ in \cite{c-c-g},
with the modifications $\Gamma$ the  coherence isomorphisms $1\circ
1\cong 1$), are easily verified by using Fact \ref{fpre1}.

Suppose now $\T\overset{H}\to \T'\overset{H'}\to \T''$ two
composable trihomomorphisms. Then, the pseudo-simplicial
pseudo-equivalence (\ref{trucom}),  $\alpha: \ner H'\ \ner
H\Rightarrow \ner (H'H)$, is, at any integer $p\geq 0$, given by
$\alpha_p=RmL\circ RH'_*\v H_*L$,
$$
\xymatrix@C=-10pt{&
RH'_*H_*L\ar@{=>}[rd]\ar@{}@<14pt>[rd]|(.6){\ R mL}&\\
\ner_pH'\ \ner_pH=RH'_*LRH_*L\ar@{=>}[ru]\ar@{}@<18pt>[ru]|(.3){RH'_*\v H_*L}
\ar@{.>}[rr]^{\alpha_p} &
&R(H'H)_*L= \ner_p(H'H),}
$$
where the pseudo-equivalence $m:H'_*H_*\Rightarrow (H'H)_*:\bhrep([p],\T)\to\bhrep([p],\T'')$
is that given in Lemma \ref{transtri} $(ii)$. The naturality
component of $\alpha$ at any map $a:[q]\to [p]$,
$$
\xymatrix@C=40pt{\ner_qH'\ \ner_qH\ \ner_{\!a}\T\ar@{}[rrd]|{\Rrightarrow}
\ar@{=>}[rr]^{\alpha_q\ner_{\!a}\T}\ar@{=>}[d]_{\ner_q\!H'\,\theta_a}&&\ner_q(H'H)\
\ner_{\!a}\!\T\ar@{=>}[d]^{\theta_a}\\
\ner_qH'\ \ner_{\!a}\T'\ \ner_pH\ar@{=>}[r]_{\theta_a\ner_p\!H}&\ner_{\!a}\T''\
\ner_pH'\ \ner_pH\ar@{=>}[r]_{\ner_{\!a}\!\T''\alpha_p}& \ner_{\!a}\T''\ \ner_p(H'H),}
$$
is provided by the invertible modification obtained by pasting in
$$
\xymatrix@C=34pt{H'_*LRH_*LRa^* \ar@{}@<-4pt>[rrd]^{  \overset{(\ref{4})}\Rrightarrow}
\ar@{=>}[d]_{H'_*LRH_*\v a^*} \ar@{=>}[rr]^-{H'_*\v LRa^*}&&
H'_*H_*LRa^*\ar@{=>}[r]^{ m LRa^* }\ar@{=>}[d]_(.5){H'_*H_*\v a^*}
\ar@{}@<-4pt>[rd]^{  \overset{(\ref{4})}\Rrightarrow}&(H'H)_*LRa^*
 \ar@{=>}[d]^{(H'H)_*\v a^*}\\
H'_*LRH_*a^* \ar@{}@<14pt>[d]^(.55){  \overset{(A)}\Rrightarrow}
\ar@{=>}[rr]^-{H'_*\v H_*a^*}\ar@{=>}[d]_{H'_*LRa^*\v^\bullet H_*}\ar@{=>}[rd]^1 &\ar@{}[d]|(.45){  \cong}&
H'_*H_*a^*\ar@{=>}[r]^{ma^*=a^*m}\ar@{=>}[ddd]_{a^*\v^\bullet H'_*H_*}& (H'H)_*a^* \ar@{}@<-4pt>[ldd]^{  \overset{(\ref{4})}\Rrightarrow}
\ar@{=>}[ddd]^{a^*\v^\bullet(H'H)_*}\\
H'_*LRa^*LRH_*\ar@{}@<-4pt>[rdd]^{  \overset{(\ref{4})}\Rrightarrow}
\ar@{=>}[dd]_{H'_*\v a^*LRH_*}\ar@{=>}[r]_{H'_*\hspace{-1pt}LRa^*\hspace{-1pt}\v\hspace{-1pt} H_*}&H'_*LRa^*H_* \ar@{}@<-4pt>[rdd]^(.6){  \overset{(\ref{4})}\Rrightarrow}
\ar@{=>}[ru]^{H'_*\v H_*a^*}&&\\
&&& \\
H'_*a^*LRH_* \ar@{=>}@/_2pc/[rruuu]^{H'_*a^*\v H_*}\ar@{=>}[r]_{a^*\v^\bullet H'_*LRH_*}&
a^*LRH'_*LRH_*\ar@{=>}[r]_{\ a^*LRH'_*\v H_*}&a^*LRH'_*H_*\ar@{=>}[r]_{a^*LRm}&a^*LR(H'H)_*,
}
$$
where the isomorphism labelled $(A)$ is given by the adjunction
invertible modification  $\v\circ \v^\bullet\cong 1$.

Finally, the pseudo-simplicial pseudo-equivalence $(\ref{truni})$,
$\beta:\ner 1_\T\Rightarrow 1_{\ner\!\T}$, is defined by the family
of pseudo-equivalences
\begin{equation} \xymatrix@C=45pt{\ner_p1_\T=R(1_\T)_*L\ar@{=>}[r]^-{\beta_p=RmL}&RL=1_{\ner_p\!\T},} \end{equation}
 where $m:(1_\T)_*\Rightarrow 1:\bhrep([p],\T)\to \bhrep([p],\T)$
is the pseudo-equivalence in Lemma \ref{transtri} $(iii)$. The
naturality invertible modification attached at any map
$a:[q]\to[p]$,
$$
\xymatrix@C=30pt@R=20pt{\ner_p1_\T\ \ner_{\!a}\!\T \ar@{}[rd]|{ \Rrightarrow}\ar@{=>}[d]_{\theta_a}\ar@{=>}[r]^-{\beta_p\ner_{\!a}\!\T}&\ner_{\!a}\!\T\ar@{=>}[d]^{1}\\
\ner_{\!a}\!\T\ \ner_q1_\T \ar@{=>}[r]_{\ner_{\!a}\!\T\beta_q}&\ner_{\!a}\!\T,}
$$
is that obtained by pasting the diagram
$$
\xymatrix{R(1_\T)_*LRa^*L  \ar@{}@<-6pt>[rd]^{\overset{(\ref{4})}\Rrightarrow}
\ar@{=>}[rr]^{RmLRa^*L}\ar@{=>}[d]_{R(1_\T)_*\v a^*\!L}&&RLRa^*L=Ra^*L\ar@{=>}[dd]^{1}\ar@{=>}[ld]\ar@{}@<-3pt>[ld]_-{1=R\v a^*\!L}\\
R(1_\T)_*a^*L\ar@{}@<-7pt>[rd]^{\overset{(\ref{4})}\Rrightarrow}
\ar@{=>}[r]^-{Rma^*\!L}\ar@{=>}[d]_{Ra^*\v^\bullet (1_\T)_*L}&Ra^*L\ar@{}[r]|-{\cong}\ar@{=>}[rd]\ar@{}@<-3pt>[rd]_(.4){1=Ra^*\v^\bullet\! L}& \\
Ra^*LR(1_\T)_*L\ar@{=>}[rr]_{Ra^*\!LRmL}&&Ra^*LRL=Ra^*L.}
$$
The coherence conditions $(\mathbf{CC5})$ and $(\mathbf{CC6})$ in
\cite{c-c-g}, for both $\alpha$ and $\beta$,  are plainly verified.
\end{proof}

\section{The classifying space of a tricategory}
\subsection{Preliminaries on classifying spaces of bicategories}When a bicategory $\b$ is regarded as a tricategory all
of whose 3-cells are identities, the nerve construction (\ref{bsnt}) on
it actually produces a normal pseudo-simplicial category
$$\ner\b=(\ner\b,\chi):\Delta^{\!^{\mathrm{op}}}\to \cat,$$ which is
called in \cite[\S 3]{ccg} the {\em pseudo-simplicial nerve of the
bicategory}. The {\em classifying space of the bicategory}, denoted
by $\class\b$, is then defined to be the ordinary classifying space
of the category obtained by the Grothendieck construction
\cite{groth71} on the nerve of the bicategory, that is,
$\class\b=\class\!\int_{\!\Delta}\!\ner\b$,  \cite[Definition
3.1]{ccg}. The following facts, concerning classifying spaces
of bicategories, are proved in \cite[(30) and Theorem 7.1]{ccg}:

\begin{fact}\label{i} Each homomorphism between bicategories $F:\b\to \c$ induces
a continuous cellular map $\class F:\class\b\to\class\c$. Thus, the
classifying space construction, $\b\mapsto \class\b$, defines a
functor from the category {\bf Hom} of bicategories  to
CW-complexes.
\end{fact}
\begin{fact}\label{ii} If $F,F':\b\to\c$ are two homomorphisms between bicategories, then any lax (or oplax)
transformation, $F\Rightarrow F'$, canonically defines a homotopy
between the induced maps on classifying spaces, $ \class F\simeq
\class F':\class\b\to\class\c$.
\end{fact}
\begin{fact}\label{iii} If a homomorphism of bicategories has a left or right biadjoint, the map induced on classifying spaces is a homotopy
equivalence. In particular, any biequivalence of bicategories
induces a homotopy equivalence on classifying spaces.
\end{fact}

Furthermore, we should recall that the classifying space of any pseudo-simplicial bicategory
$\f:\Delta^{\!^{\mathrm{op}}}\to \bicat$ is defined \cite[Definition
5.4]{c-c-g} to be the classifying space of its {\em bicategory of
simplices} $\int_{\!\Delta}\!\f$, also called {\em the Grothendieck
construction on $\f$} \cite[\S 3.1]{c-c-g}. That is, the bicategory
whose objects are the pairs $(x,p)$, where $p\geq 0$ is an integer
and $x$ is an object of the bicategory $\f_p$, and whose
hom-categories  are
$$\xymatrix{\int_{\!\Delta}\!\f\big((y,q),(x,p)\big)=\bigsqcup\limits_{[q]\overset{a}\to [p]}\f_q(y,a^*x),}$$
where the disjoint union is over all arrows $a:[q]\to [p]$ in the
simplicial category $\Delta$; compositions, identities,  and
structure constraints are defined in the natural way. We refer the
reader to \cite[\S 3]{c-c-g} for details about the bicategorical
Grothendieck construction trihomomorphism
$$\xymatrix{\int_{\!\Delta}{\!-} :\bicat^{
\Delta^{\mathrm{op}}}\rightarrow \bicat,}$$ from the tricategory of pseudo-simplicial
bicategories to the tricategory of bicategories.
The following facts, concerning  classifying spaces of
pseudo-simplicial bicategories, are proved in \cite[(42), (43),
Proposition 5.5, and Theorem 5.7]{c-c-g} :

\begin{fact}\label{nv} (i) If $\f,\g:\Delta^\mathrm{op}\to \bicat$ are
pseudo-simplicial bicategories, then each pseudo-simplicial
homomorphism $F:\f\to \g$ induces a continuous map $\class\!
\int_{\!\Delta}\!F:\class\! \int_{\!\Delta}\!\f\to \class\!
\int_{\!\Delta}\!\g$.

(ii) For any pseudo-simplicial bicategory
$\f:\Delta^{\!^{\mathrm{op}}}\to \bicat$, there is a homotopy
$$\xymatrix{\class \!\int_{\!\Delta}\!1_\f\simeq
1_{\class\!\int_{\!\Delta}\!\f}:\class\! \int_{\!\Delta}\!\f\to\class\! \int_{\!\Delta}\!\f.}$$

(iii)  For  any pair of composable pseudo-simplicial homomorphism
 $F:\f\to \g$, $G:\g\to \h$, there is a homotopy
$$\xymatrix{\class \!\int_{\!\Delta}\!\! G\ \, \class
\!\int_{\!\Delta}\!\!F\simeq \class \!\int_{\!\Delta}\!
(GF):\class\! \int_{\!\Delta}\!\f\to\class\! \int_{\!\Delta}\!\h.}$$
\end{fact}
\begin{fact}\label{vii}
Any pseudo-simplicial transformation $F\Rightarrow G:\f\to\g$
induces a homotopy $$\xymatrix{\class\! \int_{\!\Delta}\!F\simeq
\class\! \int_{\!\Delta}\!G:\class\! \int_{\!\Delta}\!\f\to \class\!
\int_{\!\Delta}\!\g.}$$
\end{fact}

\begin{fact}\label{fvi}If $F:\f\to\g$ is pseudo simplicial
homomorphism, between pseudo simplicial bicategories
$\f,\g:\Delta^{\!^{\mathrm{op}}}\to \bicat$, such that the induced
map $\class F_p:\class \f_p\to\class\g_p$ is a homotopy
equivalence, for all integers $p\geq 0$, then the induced map
$\class\!\int_{\!\Delta}\!F:\class\!\int_{\!\Delta}\!\f\to
\class\!\int_{\!\Delta}\!\g$ is a homotopy equivalence.
\end{fact}

\begin{fact}\label{iv} If  $\f:\Delta^{\!{^\mathrm{op}}}\to
\mathbf{Hom}\subset \bicat$ is a simplicial bicategory, then
 there is a natural  homotopy equivalence
$$\xymatrix{ \class\! \int_{\!\Delta}\!\f\simeq |\class \f|,}$$
where the latter is the geometric realization of the simplicial
space $\class\f:\Delta^{\!^{\mathrm{op}}}\to \Top$, obtained by
composing $\f$ with the classifying space functor $\class:\Hom\to
\Top$.
\end{fact}

\subsection{The classifying space construction for tricategories}
 We are now ready
to set the following definition, which recovers the more traditional
way through which a classifying space is assigned in the literature
to certain specific kinds of tricategories, such as 3-categories,
bicategories, monoidal categories, or braided monoidal categories
(see Examples \ref{ex3cat}, \ref{emb}, and \ref{ebmc} below, also
\cite{b-f-s-v,c-c-g} or \cite{jardine} and references therein).

\begin{definition}\label{cst}
The classifying space $\class\T$, of a tricategory $\T$, is the
classifying space of its bicategorical pseudo-simplicial
Grothendieck nerve {\em (\ref{bsnt})},
$\ner\T:\Delta^{\!{\mathrm{op}}}\to \bicat$, that is,
\begin{equation}\label{ecst} \xymatrix{\class\T:=\class \!\int_{\!\Delta}\!\!\ner\T}.
\end{equation}
\end{definition}

Let us remark that the classifying space of a tricategory $\T$ is
then realized as the classifying space of a category canonically
associated to it, namely, as
$$\xymatrix{\class\T=
\class\!\int_{\!\Delta}\!\!\ner\big(\!\int_{\!\Delta}\!\!\ner\T\big)=
|\ner\big(\int_{\!\Delta}\!\!\ner\big(\!\int_{\!\Delta}\!\!\ner\T\big)\big)|.}$$

The next two theorems deal with the basic properties concerning with the
homotopy behavior of the classifying space construction, $\T\mapsto
\class\T$, with respect to trihomomorphisms of tricategories.

\begin{theorem}\label{clastri} (i) Any trihomomorphism between tricategories $H:\T\to\T'$ induces a
continuous map $\class H:\class\T\to\class\T'$.

(ii) For any pair of composable trihomomorphisms $H:\T\to \T'$ and
$H':\T'\to \T''$, there is a homotopy $$\class H'\ \class H\simeq
\class (H'H):\class\T\to\class\T''.$$

(iii) For any tricategory $\T$, there is a homotopy $\class
1_\T\simeq 1_{\class \T}:\class\T\to\class\T.$
\end{theorem}
\begin{proof} $(i)$ By Theorem \ref{psth} $(i)$, any trihomomorphism
$H:\T\to \T'$ gives rise to a pseudo-simplicial homomorphism on the
corresponding Grothendieck nerves $\ner H:\ner\T\to\ner \T'$, which,
by Fact \ref{nv} $(i)$, produces the claimed continuous map $\class
H=\class\!\int_{\!\Delta}\!\ner H:\class\T\to\class \T'$.

\vspace{0.2cm} $(ii)$ Suppose $\T\overset{H}\to \T'\overset{H'}\to
\T''$ are trihomomorphisms. By Theorem \ref{psth} $(ii)$, there is a
pseudo-simplicial pseudo-equivalence $\ner H'\ \ner H\Rightarrow
\ner (H'H)$, which, by Fact \ref{vii}, induces a homotopy
$$\xymatrix{ \class \!\int_{\!\Delta}\!(\ner H'\,\ner H)\simeq
\class\! \int_{\!\Delta}\!\ner( H'H)=\class(H'H).}$$ Then, the
result follows since, by Fact \ref{nv} $(iii)$, there is a homotopy
$$\xymatrix{\class \!\int_{\!\Delta}\!(\ner H'\,\ner H)\simeq
\class \!\int_{\!\Delta}\!\ner H'\,  \class \!\int_{\!\Delta}\!\ner
H =\class H'\,\class H.}$$

\vspace{0.2cm} $(iii)$  By Theorem \ref{psth} $(iii)$,  there is a
pseudo-simplicial pseudo-equivalence $\ner 1_\T\Rightarrow
1_{\ner\!\T}$, which, by Fact \ref{vii}, induces a homotopy
$\xymatrix{\class 1_\T= \class\!\int_{\!\Delta}\!\ner 1_\T\simeq
\class \!\int_{\!\Delta}\! 1_{\ner\T}}$. Since,  by Fact \ref{nv}
$(ii)$, there is a homotopy $\xymatrix{\class \!\int_{\!\Delta}\!
1_{\ner\T}\simeq \class
1_{\!\int_{\!\Delta}\!\ner\T}=1_{\class\T}}$, the result follows.
\end{proof}

\begin{theorem}\label{trahom} If $F,G:\T\to\T'$ are two trihomomorphisms between
tricategories, then any tritransformation, $F\Rightarrow G$,
canonically defines a homotopy between the induced maps on
classifying spaces, $ \class F\simeq \class G:\class\T\to\class\T'$.
\end{theorem}
\begin{proof} Suppose $\theta=(\theta,\Pi,\mathrm{M}):F\Rightarrow
G:\T\to\T'$ is a tritransformation. There is a trihomomorphism
$H:\T\times [1]\to \T'$ making  the diagram commutative
\begin{equation}\label{homo}\begin{array}{c}
\xymatrix@C=60pt@R=20pt{\T\!\times\![0]\cong \T \ar[d]_{  1\times\delta_0}\ar[dr]^{  F}&
\\ \T\times[1]\ar[r]^{  H}&\T',\\
\T\times [0]\cong \T\ar[ru]_{  G}\ar[u]^{  1\times \delta_1}&}
\end{array}
\end{equation}
such that, for any objects $p,q$ of $\T$, the homomorphism
$$H=H_{(p,1)(q,0)}:\T\!\times\![1]((p,1),(q,0))\longrightarrow
\T'(Fp,Gq)$$ is the composite of
$$\xymatrix@C=40pt{\T(p,q)\!\times\! \{(1,0)\}\cong
\T(p,q)\overset{G}\longrightarrow
\T'(Gp,Gq)\ar[r]^(.7){\T'(\theta_p,1)}&\T'(Fp,Gq).}$$ For objects
$p,q,r$ of $\T$, the pseudo-equivalence
$$
\xymatrix@C=45pt{(\T\!\times\![1])((q,0),(r,0))\times (\T\!\times\![1])((p,1),(q,0))\ar[r]^(.6){H\times H}
\ar[d]_{\otimes}
\ar@{}[rd]|{\Downarrow\chi^H}&\T'(Gq,Gr)\times
\T'(Fp,Gq)\ar[d]^{\otimes}\\ \T\!\times\![1]((p,1),(r,0))\ar[r]^(.6){H}&\T'(Fp,Gr)}
$$
is obtained by pasting the diagram
$$
\xymatrix@C=40pt{\T(q,r)\!\times\!\T(p,q)\ar@{}[rd]|{\Downarrow \chi^G}\ar[d]_{\otimes}\ar[r]^-{G\times G}&\T'(Gq,Gr)\!\times\!\T'(Gp,Gq)
\ar[d]^{\otimes}\ar[r]^{1\times\T'(\theta_p,1)}\ar@{}[rd]|{\cong}&\T'(Gq,Gr)\!\times\!\T'(Fp,Gq)\ar[d]^{\otimes} \\
\T(p,r)\ar[r]^(.4){G}&\T'(Gp,Gr)\ar[r]^{\T'(\theta_p,1)}&\T'(Fp,Gr),}
$$
and the pseudo-equivalence
$$
\xymatrix@C=45pt{(\T\!\times\![1])((q,1),(r,0))\times (\T\!\times\![1])((p,1),(q,1))\ar[r]^(.6){H\times H}
\ar[d]_{\otimes}
\ar@{}[rd]|{\Downarrow\chi^H}&\T'(Fq,Gr)\times
\T'(Fp,Fq)\ar[d]^{\otimes}\\ \T\!\times\![1]((p,1),(r,0))\ar[r]^(.6){H}&\T'(Fp,Gr)}
$$
by pasting in
$$
\xymatrix@C=40pt@R=15pt{\T(q,r)\!\times\!\T(p,q)\ar[r]^-{G\times F}\ar[rd]_{F\times F}\ar[ddd]_{\otimes}&
\T'(Gq,Gr)\!\times\!\T'(Fp,Fq)\ar@{}[lddd]|{\Downarrow \chi^F}\ar@{}[rddd]|{\cong}\ar@{}[d]|{\Downarrow \theta^\bullet\times 1}\ar[r]^{\T'(\theta_q,1)\times 1}&
\T'(Fq,Gr)\!\times\!\T'(Fp,Fq)\ar[ddd]^{\otimes}\\
&\T'(Fq,Fr)\!\times\!\T'(Fp,Fq)\ar[ru]_{\T'(1,\theta_r)\times 1}\ar[d]^{\otimes}
&\\
&\T'(Fp,Fr)\ar@{}[d]|{\Downarrow\theta}\ar[rd]^{\T'(1,\theta_r)}&\\
\T(p,r)\ar[ru]^{F}\ar[r]^{G}&\T'(Gp,Gr)\ar[r]^{\T'(\theta_p,1)}&\T'(Fp,Gr).}
$$

For $p,q,r,s$ any objects of $\T$, the component of the invertible
modification $\omega^H$ at the triples of composable 1-cells of
$\T\!\times\![1]$
$$\begin{array}{l}\xymatrix{(p,1)\ar[r]^{(z,(1,0))}&(q,0)\ar[r]^{(y,1_0)}&(r,0)\ar[r]^{(x,1_0)}&(s,0),}\\
\xymatrix{(p,1)\ar[r]^{(z,1_1))}&(q,1)\ar[r]^{(y,(1,0))}&(r,0)\ar[r]^{(x,1_0)}&(s,0),}\\
\xymatrix{(p,1)\ar[r]^{(z,1_1))}&(q,1)\ar[r]^{(y,1_1))}&(r,1)\ar[r]^{(x,(1,0))}&(s,0),}
\end{array}$$
are canonically provided by the 3-cells (\ref{1omega}),
$(\ref{2omega})$ and $(\ref{3omega})$ below.
\begin{equation}\label{1omega}\begin{array}{c}
\xymatrix@C=40pt@R=15pt{((Gx\!\otimes\!Gy)\!\otimes\!Gz)\!\otimes\!\theta_p\ar@{=>}
[dd]_{\boldsymbol{a}\otimes 1}
\ar@{=>}[rr]^{(\chi^G\otimes 1)\otimes 1}
\ar@{}[rrdd]|{\Lleftarrow\,\omega^G\otimes 1}&
&
(G(x\!\otimes\!y)\!\otimes\!Gz)\!\otimes\!\theta_p
\ar@{=>}[d]\ar@{}@<4pt>[d]^{\chi^G\otimes  1}\\
&&
G((x\!\otimes\!y)\!\otimes\!z)\!\otimes\!\theta_p\ar@{=>}[d]^{G\boldsymbol{a}\otimes 1}\\
(Gx\!\otimes\!(Gy\!\otimes\!Gz))\!\otimes\!\theta_p\ar@{=>}[r]^{(1\otimes\chi^G)\otimes 1}&
(Gx\!\otimes\!(G(y\!\otimes\!z))\!\otimes\!\theta_p\ar@{=>}[r]^{\chi^G\otimes 1}
&
G(x\!\otimes\!(y\!\otimes\!z))\!\otimes\!\theta_p.
}\end{array}
\end{equation}
\begin{equation}\label{2omega}\begin{array}{c}
\xymatrix{(Gx\!\otimes\!Gy)\!\otimes\!(\theta_q\!\otimes\!Fz)\ar@{=>}[r]^{\chi^G\otimes 1}\ar@{=>}[d]_{1\otimes \theta}
\ar@{}[rd]|{\cong}&G(x\!\otimes\!y)\!\otimes\!(\theta_q\!\otimes\!Fz)
\ar@{=>}[d]^{1\otimes \theta}\\ (Gx\!\otimes\!Gy)\!\otimes\!(Gz\!\otimes\!\theta_p)
\ar@{=>}[r]^{\chi^G\otimes 1}&G(x\!\otimes\!y)\!\otimes\!(Gz\!\otimes\!\theta_p)
}
\end{array}\end{equation}
\begin{equation}\label{3omega}\begin{array}{c}
\xymatrix@C=35pt@R=15pt{Gx\!\otimes\!((\theta_r\!\otimes\!Fy)\!\otimes\!Fz)\ar@{=>}[r]^{1\otimes(\theta\otimes 1)}
\ar@{=>}[d]_{1\otimes\boldsymbol{a}}\ar@{}[rrdd]|{\Lleftarrow\, 1\otimes \Pi}&
Gx\!\otimes\!((Gy\!\otimes\!\theta_q)\!\otimes\!Fz)\ar@{=>}[r]^{1\otimes\boldsymbol{a}}&
Gx\!\otimes\!(Gy\!\otimes\!(\theta_q\!\otimes\!Fz))\ar@{=>}[d]^{1\otimes(1\otimes\theta)}\\
Gx\!\otimes\!(\theta_r\!\otimes\!(Fy\!\otimes\!Fz))\ar@{=>}[d]^{1\otimes(1\otimes\chi^F)}
&&Gx\!\otimes\!(Gy\!\otimes\!(Gz\!\otimes\!\theta_p))\ar@{=>}[d]^{1\otimes\boldsymbol{a}^\bullet}\\
Gx\!\otimes\!(\theta_r\!\otimes\!F(y\!\otimes\!z))\ar@{=>}[r]^{1\otimes \theta}&
Gx\!\otimes\!(G(y\!\otimes\!z)\!\otimes\!\theta_p)
&Gx\!\otimes\!((Gy\!\otimes\!Gz)\!\otimes\!\theta_p)\ar@{=>}[l]_{1\otimes (\chi^G\otimes 1)}
}\end{array}
\end{equation}

To finish the description of the homomorphism $H$, say that the
component of the invertible modification $\delta^H$ at any morphism
$(x,(1,0)):(p,1)\to (q,0)$ is canonically obtained from the 3-cells
$1\otimes \mathrm{M}$ and $\delta^G\otimes 1$ below, while the
component of $\gamma^H$ is provided by 3-cell $\gamma^G\otimes 1$.
$$\xymatrix@C=40pt{Gx\!\otimes\!
\theta_p\ar@{=>}[r]^{1\otimes\boldsymbol{r}}
\ar@{}[rrd]|{\Lleftarrow\, 1\otimes \mathrm{M} }
\ar@{=>}[d]_{1\otimes\,\boldsymbol{l}^\bullet}
&Gx\!\otimes\!(\theta_p\!\otimes\!1) \ar@{=>}[r]^{1\otimes(1\otimes
\iota^F)}&Gx\!\otimes\!(\theta_p\!\otimes\!F1)
\ar@{=>}[d]^{1\otimes \theta}\\
Gx\!\otimes\!(1\otimes\theta_p)\ar@{=>}[rr]^(.35){1\otimes(\iota^G\otimes 1)}&&Gx\!\otimes\!
(G1\!\otimes\!\theta_q),
}
$$
$$
\xymatrix@C=40pt{Gx\!\otimes\!\theta_p\ar@{=>}[r]^{\boldsymbol{r}\otimes 1}
\ar@{=>}[d]_{G\boldsymbol{r}\otimes 1} \ar@{}[rd]|{\delta^G\otimes 1\,\Rrightarrow}
&(Gx\!\otimes\!1)\!\otimes\!\theta_p
\ar@{=>}[d]^{(1\otimes\iota^G)\otimes 1}
\\
G(x\!\otimes\!1)\!\otimes\!\theta_p&(Gx\!\otimes\!G1)\!\otimes\!\theta_p,
\ar@{=>}[l]_{\chi^G\otimes 1}}
\xymatrix@C=40pt{(1\!\otimes\!Gx)\!\otimes\!\theta_p\ar@{=>}[r]^{(\iota^G\otimes 1)\otimes 1}
\ar@{=>}[d]_{\boldsymbol{l}\otimes 1}  \ar@{}[rd]|{\Lleftarrow\,\gamma^G\otimes 1}
&(G1\!\otimes\!Gx)\!\otimes\!\theta_p\ar@{=>}[d]^{\chi^G\otimes 1}\\
Gx\!\otimes\!\theta_p&G(1\!\otimes\!x)\!\otimes\!\theta_p.\ar@{=>}[l]_{G\boldsymbol{l}\otimes 1}}
$$

We are now ready to complete the proof of the theorem: Applying the
classifying space construction to diagram (\ref{homo}), we obtain a
diagram of maps
$$
\xymatrix@C=60pt{\class\T\times \class[0]\cong \class\T \ar[d]_{  1\times \delta_0}
\ar[dr]^{  \class F}& \\
\class\T\times\class[1]\ar[r]^{   \class H}&\class\T',\\
\class\T\times \class[0]\cong \class\T\ar[ru]_{  \class G}\ar[u]^{  1\times\delta_1}&}
$$
where, by Theorem \ref{clastri} $(ii)$, both triangles are homotopy
commutative. Since $\class[1]=[0,1]$, the unity interval, the result
follows.
\end{proof}

As a relevant consequence for triequivalences between tricategories
\cite[Definition 3.5]{g-p-s}, we have the following:
\begin{theorem}\label{theoclasstri} $(i)$ If $F:\T\to \T'$ is any trihomomorphism such
that there are a trihomomorphism $G:\T'\to\T$ and tritransformations
$FG\Rightarrow 1_{T'}$ and $1\Rightarrow GF$, then the induced map
$\class F:\class\T\to\class \T'$ is a homotopy equivalence.

$(ii)$ Any triequivalence of tricategories induces a homotopy
equivalence on classifying spaces.
\end{theorem}
\begin{proof}
$(i)$ Given any trihomomorphism $F:\T\to\T'$ in the hypothesis, by
Theorem \ref{clastri} $(ii)$, there is a homotopy $\class F\,\class
G\simeq \class(FG)$. By Theorem \ref{trahom},  the
existence of a homotopy $\class (FG)\simeq \class 1_{\T'}$ follows. Since,
by Theorem \ref{clastri} $(iii)$, there is a homotopy $\class
1_{\T'}\simeq 1_{\class\T'}$, we conclude the existence of a
homotopy $\class F\,\class G\simeq 1_{\class\T'}$.  Analogously, we
can prove that $1_{\class \T}\simeq \class G\,\class F$, which
completes the proof.

Part $(ii)$ clearly follows from part $(i)$.
\end{proof}

\subsection{Example: Classifying spaces of $3$-categories}\label{ex3cat} In \cite{segal68}, Segal observed that, if $\mathbb{C}$ is a
topological category, then its Grothendieck nerve (\ref{ncat}) is, in
a natural way, a simplicial space, that is,
$\ner\mathbb{C}:\Delta^{\!^{\mathrm{op}}}\to \Top$. Then, he defines
the \emph{classifying space} of a topological category $\mathbb{C}$
to be $|\ner\mathbb{C}|$,
 the realization of this simplicial space.

This notion given by Segal provides, for instance, the usual
definition of classifying spaces of strict tricategories, or
3-categories (hence of categories and 2-categories, which can be
regarded as special 3-categories). A 3-category $\T$ is just a
category enriched in the category of 2-categories and 2-functors,
that is, a category $\T$ endowed with 2-categorical hom-sets
$\T(t',t)$, in such a way that the compositions $\T(t',t)\times
\T(t'',t')\rightarrow \T(t'',t)$ are 2-functors. By replacing the
hom 2-categories $\T(t',t)$ by their classifying spaces, we obtain a
topological category, say $\mathbb{C}_\T$, with discrete space of
objects and whose hom-spaces are
$\mathbb{C}_\T(t',t)=\class\T(t',t)$. The classifying space of this
topological category $|\ner\mathbb{C}_\T|$ is, by definition, the
classifying space  of the 3-category $\T$.

\begin{theorem}\label{the3cat}
For any $3$-category $\T$ there are homotopy equivalences
$$|\ner\mathbb{C}_\T|\simeq \class\T\simeq |\diag\ner\ner\ner\T|,$$
where
$\ner\ner\ner\T:\Delta^{\!{\mathrm{op}}}\!\times\!\Delta^{\!{\mathrm{op}}}\!\times\!
\Delta^{\!{\mathrm{op}}}\to \set$, $([p],[q],[r])\mapsto
\ner_{\!r}\ner_{q}\ner_{p}\T$, is the trisimplicial set {\em triple
nerve} of the 3-category, and
$\diag\ner\ner\ner\T:\Delta^{\mathrm{op}}\to\set$, $[p]\mapsto
\ner_{p}\ner_{p}\ner_{p}\T$, its diagonal simplicial set.
\end{theorem}
\begin{proof}
Note that, for any 3-category  $\T$, the equality of simplicial
spaces  $\ner\mathbb{C}_\T=\class\ner\T$ holds, where the latter is
the simplicial space obtained by composing its (actually simplicial
2-category) nerve (\ref{bsnt}), $\ner\T:\Delta^{\!{\mathrm{op}}}\to
\text{2-}\cat$, with the classifying functor
$\class:\text{2-}\cat\to \Top$. Then, by Fact \ref{iv}, there is a
natural homotopy equivalence $\class\T\simeq |\ner\mathbb{C}_\T|$.

Furthermore, an iterated use of the natural homotopy equivalences
$\class\T\simeq |\class\ner\T|$ (which, of course, also work
both for 2-categories and categories) give the following chain of
homotopy equivalences, for any 3-category $\T$,
\begin{equation}\label{neqt}
\begin{array}{lll}
           \class \T&\simeq&|[p]\mapsto \class\ner_{p}\T|
           \ \simeq \ |[p]\mapsto |[q]\mapsto \class\ner_{q}\ner_{p}\T|| \\ [4pt]
           &\simeq&  |[p]\mapsto |[q]\mapsto |[r]\mapsto \ner_{\!r}\ner_{q}\ner_{p}\T|||
          \  \cong\ |\diag\ner\ner\ner\T|,
\end{array}
\end{equation}
where the last homeomorphism above is a consequence of Quillen's
Lemma \cite[page 86]{quillen}.\end{proof}

\subsection{Example: Classifying spaces of  bicategories}
When a bicategory $\b$ is viewed as a tricategory whose 3-cells are
all identities, then its classifying space as a tricategory,
according to Definition \ref{cst}, coincides with the classifying
space of the bicategory, $\class\b$, as  defined in
\cite[Definition 3.1]{ccg}.

\subsection{The Segal nerve of a tricategory}\label{subsegalner} Several theoretical interests suggest dealing with the
so-called Segal nerve construction for tricategories. This
associates to any tricategory $\T$ a simplicial bicategory, denoted
by $\bsner\T$, which can be thought of as a `rectification' of the
pseudo-simplicial Grothendieck nerve of the tricategory $\ner\T$
$(\ref{bsnt})$, since both are biequivalent in the tricategory of
pseudo-simplicial bicategories and therefore model the same homotopy type.  Furthermore, $\bsner\T$ is a weak 3-category under the
point of view of Tamsamani \cite{tam} and Simpson \cite{sim}, in the
sense that it is a {\em special simplicial bicategory}, that is, a
simplicial bicategory $S:\Delta^{\!^{\mathrm{op}}} \to \Hom $
satisfying the following two conditions:

(i)  $\bsner_0$ is discrete (i.e., all its 1- and 2-cells are
identities).

(ii) for $n\geq 2$, the Segal projection homomorphisms (see
\cite[Definition 1.2]{segal74})
\begin{equation}\label{1.2.29}
p_n=\prod_{k=1}^n d_0\cdots d_{k-2} d_{k+1}\cdots d_n \!\!:\bsner_n\longrightarrow
\bsner_{1}{_{d_0}\times_{d_1}}\bsner_{1}{_{d_0}\times_{d_1}}\overset{(n}\cdots {_{d_0}\times_{d_1}}
\bsner_{1}
\end{equation}
are biequivalences of bicategories.

For a reduced special simplicial bicategories $S$ as above, that is,
with $S_0=1$, the one-object discrete bicategory, the simplicial
space $\class S:\Delta^{\!^{\mathrm{op}}} \to \Top$, obtained by
replacing each bicategory by its classifying space,  satisfies 
hypothesis $(i)$ and $(ii)$ of Segal's Proposition 1.5 in
\cite{segal74} (see also \cite{m-s}): $\class S_0$ is contractible,
and the maps $\class p_n:\class S_n\to \class(S_1^n)=(\class S_1)^n$
are homotopy equivalences. Then, $\class S_1$ becomes an $H$-space
with multiplication induced by the composite homomorphism $S_1\times
S_1\overset{p_2^\bullet}\simeq S_2\overset{d_1}\to S_1$, and we have
the following useful result:

\begin{lemma}\label{lms}
If $S$ is any reduced special simplicial bicategory, then  the loop
space $\Omega|\class S|$ is a group completion of $\class S_1$.
Then,  if $\pi_0S_1$ is a group, there is a homotopy equivalence
$\class S_1\simeq \Omega|\class S|$. \end{lemma}

Let us recall that, for a given tricategory $\T$, the construction
given in (\ref{buhrep}) of the bicategory of unitary homomorphic
representations of any small category $I$ in the tricategory $\T$,
$I\mapsto \buhrep(I,\T)$, is clearly functorial on the small category
$I$, and it leads to the definition below.

\begin{definition} \label{dbsner} The  {\em Segal nerve} of a tricategory
$\T$ is  the simplicial bicategory
\begin{equation}\label{bsn}\bsner\T:\Delta^{\!^{\mathrm{op}}}
\to  \Hom\subset \bicat,\hspace{0.4cm}[p]\mapsto \bsner_{p}\T=\buhrep([p],\T).
\end{equation}
\end{definition}

\vspace{0.2cm} We should remark that, when $\T=\b$ is a bicategory,
that is, when its 3-cells are all identities, then the Segal nerve
$\bsner\b$ is introduced in \cite[Definition 5.2]{ccg}, although  it
was first studied by Lack and Paoli in \cite{lack-paoli} under the
name  of `2-nerve of $\b$'. However,  this  may be a confusing
terminology  for $\bsner\b$ since, for example, the so-called
`geometric nerve' of a 2-category $\b$ \cite{b-c,ccg} is also called
the `2-nerve of $\b$' in
 \cite{tonks}.

\begin{theorem}\label{bsners} Let $\T$ be a tricategory. Then, the
following statements hold:

$(i)$ There is a normal pseudo-simplicial homomorphism
\begin{equation}\label{cm2}L:\ner\T\to\bsner\T,\end{equation}
such that, for any $p\geq 0$, the homomorphism $L_p:\ner_p\T\to
\bsner_p\T$ is a biequivalence of bicategories.

$(ii)$ The simplicial bicategory $\bsner\T$ is special.
\end{theorem}

\begin{proof} Let us recall the explicit construction of $\ner\T=(\ner\T,\chi,\omega)$ given in the
proof of Theorem \ref{pstc}, particularly the constructions of the
bicategories $\ner_p\T$ in $(\ref{bpsnt})$, of the homomorphisms
$\ner_{\!a}\T$ in $(\ref{astar})$, of the pseudo-equivalences
$\chi_{a,b}$ in $(\ref{chi})$, and of the modifications
$\omega_{a,b,c}$ in $(\ref{omega})$. Furthermore, recall from Lemma
\ref{rufc} $(c)$ that every biadjoint biequivalence $(\ref{plrg})$
restricts by giving a biadjoint biequivalence
\begin{equation}\label{plrg2}L_p\dashv R_p: \ner_p\T \rightleftarrows \bsner_p\T.\end{equation}

The normal pseudo-simplicial homomorphism $(\ref{cm2})$,
$L=(L,\theta,\Pi):\ner\T\to\bsner\T$, is then defined by the
homomorphisms $L_p: \ner_p\T\to\bsner_p\T$, $p\geq 0$.  For any map
$a:[q]\to [p]$ in the simplicial category, the structure
pseudo-equivalence
$$
\xymatrix@R=18pt{\ner_p\T\ar[r]^{L_p}\ar[d]_{\ner_a\!\T}
\ar@{}[rd]|{ \theta_a\,\Rightarrow}&\bsner_p\T\ar[d]^{a^*}\\
\ner_q\T\ar[r]_{L_q}&\bsner_q\T,}
$$
is provided by the counit pseudo-equivalence $\v_q:L_qR_q\Rightarrow
1_{\bsner_q\T}$ ; that is, \begin{equation}\xymatrix@C=40pt{L_q\
\ner_{\!a}\T=L_qR_qa^*L_p\ar@{=>}[r]^-{\theta_a=\v_qa^*L_p}&a^*\,
L_p.}\end{equation} For $[n]\overset{b}\to [q]\overset{a}\to [p]$,
any two composable arrows  of $\Delta$, the  structure invertible
modification
$$\xymatrix@R=8pt@C=20pt{L_n\,\ner_{\!b}T\,\ner_{\!a}\T\ar@2{->}[dd]_{ \theta \ner_{\!a}\!\T}
\ar@{=>}[rr]^{ L_n\chi} & &L_n\,\ner_{\!ab}\T\ar@{=>}[dd]^{\theta}
\\ \ar@{}[rr]|(0.5){ { \Pi_{a,b}}\,\Rrightarrow} && \\
b^*L_q\,\ner_{\!a}\T\ar@{=>}[r]_{ b^*\theta}&b^*a^*L_p\ar@{=}[r]_{ 1}&
(ab)^*L_p }
$$
is directly provided by the canonical modification $(\ref{2vec})$,
$\Pi_{a,b}=\omega'_ba^*L_p$.

The coherence conditions for $L$ (i.e., conditions $(\mathbf{CC3})$
and $(\mathbf{CC4})$ in \cite{c-c-g}, with the modifications
$\Gamma$ the  coherence isomorphisms $1\circ 1\cong 1$), are easily
verified by using Fact \ref{fpre1}. This complete the proof of part
$(i)$.

And when it comes to part $(ii)$,  that is, that $\bsner\T$ is a
special simplicial bicategory, we have the following (quite obvious)
identifications between bicategories: \begin{equation}\label{i0}
\bsner_0\T= \buhrep([0],\T)=\mathrm{Ob}\T=\ner_0\T,\end{equation}
\begin{equation}\label{i1}\xymatrix{\bsner_1\T=
\buhrep([1],\T)=\bigsqcup\limits_{(t_0,t_1)}\hspace{-0.15cm}
\mathcal{T}(t_1,t_0)=\ner_1\T,}\end{equation}  and, for any integer $p\geq 2$,
$$
\begin{array}{ll}\bsner_{1}\!\T{_{d_0}\times_{d_1}}
\overset{(p}\cdots {_{d_0}\times_{d_1}}
\bsner_{1}\!\T &\hspace{-0.2cm}= \bigsqcup\limits_{(t_0,t_1)}\hspace{-0.15cm}
\mathcal{T}(t_1,t_0)\times
\bigsqcup\limits_{(t_1,t_2)}\hspace{-0.15cm}
\mathcal{T}(t_2,t_1)\times\cdots\times \bigsqcup\limits_{(t_{p-1},t_p)}\hspace{-0.25cm}
\mathcal{T}(t_{p},t_{p-1})\\[14pt] &\hspace{-0.2cm}=
\bigsqcup\limits_{(t_0,\ldots,t_p)}\hspace{-0.3cm}
\mathcal{T}(t_1,t_0)\times\mathcal{T}(t_2,t_1)\times\cdots\times\mathcal{T}(t_p,t_{p-1})
=\ner_p\T.\end{array}
$$
Through these identifications we see that, for any integer $p\geq
2$, the Segal projection homomorphism (\ref{1.2.29}) is precisely
the biequivalence $R_p:\bsner_p\T\to\ner_p\T$ in $(\ref{plrg2})$
which, recall, is defined by restricting it to the basic graph
$(\ref{graphp})$ of the category $[p]$. Whence the simplicial
bicategory $\bsner\T$ is special.
\end{proof}

The following theorem states that the classifying space of a
tricategory $\T$ can be realized, up to homotopy equivalence, by its
Segal nerve $\bsner\T$. This fact will be relevant for our later
discussions on loop spaces. Let
$$\class\bsner\T:\Delta^{\!{\mathrm{op}}}\to\Top$$
be the simplicial space obtained by composing
$\bsner\T:\Delta^{\!^{\mathrm{op}}}\to \Hom\subset\bicat$ with the
classifying functor $\class:\Hom\to\Top$ (recall Fact \ref{i}).

\begin{theorem}\label{segr} For any tricategory $\T$, there is a
 homotopy equivalence $\class\T\simeq |\class\bsner\T|$.
\end{theorem}
\begin{proof} Let us consider the pseudo-simplicial homomorphism $(\ref{cm2})$,
$L:\ner\T\to\bsner\T$. Since, for every integer $p\geq 0$, the
homomorphism $L_p:\ner_{p}\T\to\bsner_{p}\T$ is a biequivalence, it
follows from Fact \ref{iii} that the induced cellular map $\class
L_p:\class \ner_{p}\T\to\class\bsner_{p}\T$ is a homotopy
equivalence. Then, by Fact \ref{fvi}, the induced map
$\class\int_{\!\Delta}\!L:\class\!\int_{\!\Delta}\!\ner\T\to\class\!\int_{\!\Delta}\!\bsner\T$
is a homotopy equivalence. Since, by definition,
$\class\T=\class\!\int_{\!\Delta}\!\ner\T$, whereas, by Fact
\ref{iv}, there is a homotopy equivalence
$\class\!\int_{\!\Delta}\!\bsner\T\simeq |\class\bsner\T|$, the
claimed homotopy equivalence follows.
\end{proof}

\subsection{Example: Classifying spaces of monoidal
bicategories}\label{emb} Any monoidal bicategory
$(\b,\otimes)=(\b,\otimes,\text{I},\boldsymbol{a},\boldsymbol{l},\boldsymbol{r},\pi,\mu,\lambda,\rho)$
 can be viewed as a
tricategory
\begin{equation}\label{ome}\Sigma(\b,\otimes)\end{equation}
 with only one object, say $*$,  whose hom-bicategory is the underlying bicategory. Thus,
 $\Sigma(\b,\otimes)(*,*)=\b$, and
it is the  composition  given by the tensor functor
${\otimes:\b\times\b\rightarrow\b}$ and the identity at the object
is $1_*=\text{I}$, the unit object of the monoidal bicategory. The
structure pseudo-equivalences and modifications $\boldsymbol{a}$,
$\boldsymbol{l}$, $\boldsymbol{r}$, $\pi$, $\mu$, $\lambda$, and
$\rho$ for $\Sigma(\b,\otimes)$ are just those of the monoidal
bicategory, respectively. Call this tricategory the {\em suspension},
or {\em delooping, tricategory of the bicategory $\b$ induced by the
 monoidal structure given on it}, and call its corresponding
Grothendieck nerve (\ref{bsnt}) the {\em nerve of the monoidal
bicategory}, hereafter denoted by $\ner(\b,\otimes)$. Thus,
\begin{equation}\label{rbc}\ner(\b,\otimes)\!=\!\ner\Sigma(\b,\otimes): \Delta^{\mathrm{op}} \to\bicat,
\hspace{0.4cm} [p]\mapsto \b^p,\end{equation} is a normal
pseudo-simplicial bicategory, whose bicategory of $p$-simplices is
the $p$-fold power of the underlying bicategory $\b$, with face and
degeneracy homomorphisms induced by the  tensor
 homomorphism $\otimes:\b\times\b\to \b$ and unit object I, following  the familiar formulas (\ref{fade}),
 in analogy with those of the reduced bar construction on a topological monoid, and with structure
 pseudo-equivalences and modifications canonically arising from the data of the monoidal structure on $\b$.
The general Definition \ref{cst} for classifying spaces of
tricategories leads to the following:

\begin{definition}\label{dcsmb}
The {\em classifying space} of the monoidal bicategory, denoted by
$\class(\b,\otimes)$, is defined to be the classifying space of its
delooping tricategory $\Sigma(\b,\otimes)$. Thus,
\begin{equation}\xymatrix{\class(\b,\otimes)=\class\Sigma(\b,\otimes)=\class
\!\int_{\!\Delta}\!\!\ner(\b,\otimes)}.\end{equation}
\end{definition}

 If $(C,\otimes)=(C,\otimes,\text{I},\boldsymbol{a}, \boldsymbol{l},\boldsymbol{r})$ is a monoidal category,
  and we regard $C$ as a bicategory all of whose 2-cells are identities, then the suspension  tricategory
$\Sigma(C,\otimes)$ is actually a bicategory,
  called in \cite[2.10]{KV94b} the {\em delooping bicategory of the category
  induced by its monoidal structure}.
  The  nerve of $\Sigma(C,\otimes)$ then becomes the pseudo-simplicial category
 \begin{equation}\label{psnm}\ner(C,\otimes):\Delta^{\mathrm{op}} \to\cat,
 \hspace{0.4cm} [p]\mapsto C^p,\end{equation}
used by  Jardine in \cite[\S 3]{jardine} to define the classifying
space of the monoidal category just as above:
$\class(C,\otimes)=\!\int_{\!\Delta}\!\!\ner(C,\otimes)$  (see also
\cite{b-c2}, \cite{b-f-s-v}, or \cite[Appendix]{hinich}, for more
references). Thus,
\begin{equation}\label{csmc}\class(C,\otimes)=\class\Sigma(C,\otimes),\end{equation} as above for
arbitrary monoidal bicategories.

It is a well-known fact by Stasheff \cite{sta}  that the classifying
space of a monoidal category $(C,\otimes)$ is, up to group
completion, a loop space. More precisely, it is a fact that
\begin{quote} {\em `` the loop space $\Omega\class(C,\otimes)$ is a
group completion of $\class C$" }\end{quote} (see [23, Propositions
3.5 and 3.8] or [7, Corollary 4]). Next theorem extends Stasheff's
result to bicategories, by showing that the group completion of the
classifying space of a bicategory enriched with a monoidal structure
is homotopy equivalent to a loop space.

\begin{theorem}\label{ltmb} For any monoidal bicategory $(\b,\otimes)$, the loop space of its classifying
space, $\Omega\class(\b,\otimes)$, is a group completion of the
classifiying space of the underlying bicategory, $\class\b$. In
particular, if the monoid  of connected components $\pi_0\b$ is a
group, then there is a homotopy equivalence $\class \b\simeq
 \Omega\class(\b,\otimes)$.
\end{theorem}

\begin{proof} By Theorem \ref{segr}, $\class(\b,\otimes)$ is
homotopy equivalent to $|\class\bsner\Sigma(\b,\otimes)|$, the
geometric realization of the the simplicial space obtained by taking
classifying spaces on the simplicial bicategory
$\bsner\Sigma(\b,\otimes)$, Segal nerve of the suspension tricategory
of the monoidal bicategory. By Theorem \ref{bsners},
$\bsner\Sigma(\b,\otimes)$ is a special simplicial bicategory.
Furthermore, since the tricategory $\Sigma(\b,\otimes)$ has only one
object, $\bsner\Sigma(\b,\otimes)$ is reduced (see (\ref{i0})).
Hence, the result follows from Lemma \ref{lms}, since
$\bsner_1\Sigma(\b,\otimes)=\b$, by the identification (\ref{i1}).
\end{proof}

\subsection{Example: Classifying spaces of braided monoidal
categories}\label{ebmc} Let $(C,\otimes,\boldsymbol{c})$ be a
braided monoidal category as in \cite{joyal}. Thanks to the
braidings $\boldsymbol{c}:x\otimes y\to y\otimes x$,  the given
tensor product on $C$ defines in the natural way a tensor product
homomorphism on the suspension bicategory  of the underlying
monoidal category, $\otimes:\Sigma (C,\otimes)\times\Sigma
(C,\otimes)\to\Sigma(C,\otimes)$. Thus $(\Sigma(C,\otimes),\otimes)$
is a monoidal bicategory. The corresponding suspension tricategory,
\begin{equation}\label{omec}\Sigma^2(C,\otimes,\boldsymbol{c})=\Sigma(\Sigma(C,\otimes),\otimes)\end{equation}
is called the {\em double suspension}, or {\em double delooping, of the
underlying category $C$ associated to the given braided monoidal
structure on it} (see \cite[4,2.5]{berger}, \cite[4.2]{KV94b} or
\cite[7.9]{g-p-s}). This is a tricategory
 with only one object, say $*$,
only one arrow $*=1_*:*\to *$, the objects  of $C$ are the 2-cells,
and the morphisms of  $C$ are the 3-cells. The hom-bicategory is
$\Sigma^2(C,\otimes,\boldsymbol{c})(*,*)=\Sigma(C,\otimes)$, the
suspension bicategory  of the underlying monoidal category
$(C,\otimes)$, the composition is also (as the horizontal one in
$\Sigma(C,\otimes)$) given by the tensor functor ${\otimes:C\times
C\rightarrow C}$ and the interchange 3-cell between the two
different composites of 2-cells is given by the braiding. The most
striking instance is for $(C,\otimes,\boldsymbol{c})=(A,+,0)$, the
strict braided monoidal
 category with only one object defined by an abelian group $A$, where both composition and tensor
 product are given by the addition $+$ in $A$; in this case, the
 double suspension tricategory $\Sigma^2A$ is precisely the
 3-category treated in Examples \ref{eta} and \ref{eta2}.

For any braided monoidal category $(C,\otimes,\boldsymbol{c})$, the
Grothendieck nerve (\ref{bsnt}) of the double suspension tricategory
$\Sigma^2(C,\otimes,\boldsymbol{c})$ coincides with the
pseudo-simplicial bicategory called in \cite{c-c-g} the {\em nerve
of the braided monoidal category}, and denoted by
$\ner(C,\otimes,\boldsymbol{c})$. Thus,
\begin{equation}\label{psnbmc}\ner(C,\otimes,\boldsymbol{c})\!=\!\ner\Sigma^2(C,\otimes,\boldsymbol{c}):\Delta^{\!^{\mathrm{op}}}
\to\bicat, \hspace{0.4cm} [p]\mapsto
(\Sigma(C,\otimes))^p= \Sigma (C^p,\otimes),\end{equation}
is a normal pseudo-simplicial one-object bicategory whose bicategory
of $p$-simplicies is the suspension bicategory of the monoidal
category $p$-fold power of $(C,\otimes)$. Since the {\em classifying
space of the braided monoidal category} \cite[Definition
6.1]{c-c-g}, $\class(C,\otimes,\boldsymbol{c})$, is just given by
\begin{equation}\xymatrix{\class(C,\otimes,\boldsymbol{c})=\class
\!\int_{\!\Delta}\!\!\ner(C,\otimes,\boldsymbol{c}),}\end{equation}
we have the following:
\begin{theorem}\label{csmb}
The  classifying space of a braided monoidal category coincides with
the classifying space of its double suspension tricategory,
$\xymatrix{\class(C,\otimes,\boldsymbol{c})=\class
\Sigma^2(C,\otimes,\boldsymbol{c})}$.
\end{theorem}

Braided monoidal categories have been playing a key role in recent
developments in  quantum theory and its related topics,
  mainly thanks to the following result:

\begin{quote} {\em ``The group completion of the classifying space of a braided monoidal category is a
double loop space"}
\end{quote}

\noindent as  was noticed by J. D. Stasheff in \cite{sta}, but
originally proven by Z. Fiedorowicz in \cite{fie} (other proofs can
be found in \cite{b-f-s-v,berger}). Below we show a new proof of
Stasheff-Fiedorowicz's result, in the more precise form stated in
\cite[Theorem 6.10]{c-c-g}:

\begin{theorem}\label{stbmc}
For any braided monoidal category $(C,\otimes,\boldsymbol{c})$ there
is a homotopy equivalence
$\class(C,\otimes)\simeq\Omega\class(C,\otimes,\boldsymbol{c}).$
\end{theorem}
\begin{proof} By Theorem \ref{csmb}, the classifying space of any braided monoidal category
$(C,\otimes,\boldsymbol{c})$ is the same as the classifying space of
the monoidal bicategory $\Sigma(\Sigma( C,\otimes),\otimes)$.
Therefore,
$\Omega\class(C,\otimes,\boldsymbol{c})=\Omega\class\Sigma(\Sigma(
C,\otimes),\otimes)$. Since $\Sigma(C,\otimes)$ has only one object,
it is obvious that its monoid of connected component $\pi_0\Sigma(
C,\otimes)=1$, the trivial group. Then, by Theorem \ref{ltmb}, there
is a homotopy equivalence $\class \Sigma (C,\otimes)\simeq
\Omega\class\Sigma(\Sigma( C,\otimes),\otimes)$. Since, by
(\ref{csmc}), $\class(C,\otimes)=\class \Sigma(C,\otimes)$, the
result follows.\end{proof}

\begin{corollary}\label{corbmc} For any braided monoidal category
$(C,\otimes,\boldsymbol{c})$, the double loop space of its
classifying space, $\Omega^2\class(C,\otimes,\boldsymbol{c})$, is a
group completion of the classifying space of the underlying
category, $\class C$.
\end{corollary}
\begin{proof}
By Theorem \ref{ltmb}, $\Omega\class(C,\otimes)$ is a group
completion of $\class C$. By Theorem \ref{stbmc} above, there is a
homotopy equivalence $\Omega\class(C,\otimes)\simeq
\Omega^2\class(C,\otimes,\boldsymbol{c})$, whence the result.
\end{proof}
\begin{remark}{\em When each suspension bicategory
$\Sigma(C^p,\otimes)$ is replaced in (\ref{psnbmc}) by its nerve,
that is, by $\ner(C^p,\otimes):\Delta^{\mathrm{op}}\to \cat$, then
one has a pseudo-bisimplicial category
$$\Delta^{\mathrm{op}}\times \Delta^{\mathrm{op}}\to
\cat,\hspace{0.4cm} ([p],[q])\mapsto C^{pq},$$ which (for
$(C,\otimes)$ strict) is taken in \cite{b-f-s-v} to construct a
double delooping space for the classifying space of the underlying
category $\class C$. That double delooping space is homotopy
equivalent to the classifying space of the braided monoidal
$\class(C,\otimes,\boldsymbol{c})$}.
\end{remark}

\section{The Street nerve of a tricategory}
With the notion of the classifying space of a tricategory $\T$ given
above,  the resulting CW-complex $\class\T$ thus obtained has many
cells with little apparent intuitive connection with the cells of
the original tricategory, and they do not enjoy any proper geometric
meaning. This leads one to search for any simplicial set realizing
the space $\class\T$ and whose cells give a logical geometric
meaning to the data of the tricategory. With the definition below
(which, up to minor changes and terminology,  is essentially due to
Street \cite{street2,street3}), we give a convincing natural
response for such a simplicial set.

\subsection{The geometric nerve of a tricategory} \label{subsec41}For any given tricategory $\T$, the construction
$I\mapsto \urep(I,\T)$ given in (\ref{urep}), which carries each
category $I$ to the set of unitary representations of $I$ in $\T$,
is clearly functorial on the small category $I$, whence we have the
following simplicial set:

\begin{definition}\label{goner} The {\em geometric nerve} of a tricategory $\T$ is the simplicial set
\begin{equation}\label{gn} \Delta\T:
\Delta^{\mathrm{op}}\ \to \ \set,\hspace{0.5cm} [p]\mapsto
\urep([p],\T),\end{equation} whose $p$-simplices are unitary
representations of the category $[p]$ in $\T$.
\end{definition}

The simplicial set  $\Delta\T$ encodes the entire tricategorical
structure of $\T$ and, as we will prove below, faithfully represent
the classifying space of the tricategory $\T$, up to homotopy. We
shall stress here that the simplices of the geometric nerve
$\Delta\T$  have the following pleasing  geometric description: The
vertices of $\gner \T$ are points labelled with the objects $ F_0$ of
$\T$. The 1-simplices are paths labelled with the 1-cells
\begin{equation}\label{dnt1}F_{01}: F_1\to F_0.\end{equation}The 2-simplices are oriented triangles
\begin{equation}\label{dnt2}
\xymatrix@C=10pt@R=10pt{& F_ 2\ar[ld]_{ F_{12}}\ar[rd]^{ F_{02}}
\ar@{}[d]|(0.58){ \overset{F_{012_{}}}\Rightarrow}& \\  F_ 1\ar[rr]_{ F_{01}}&& F_ 0,}
\end{equation}
with objects $ F_ i$ placed on the vertices, 1-cells $ F_{ij}:F_j\to
F_i$ on the edges, and labelling the inner as  a 2-cell $ F_{012}:
F_{01}\otimes  F_{12}\Rightarrow  F_{02}$. For $p\geq 3$,  a
$p$-simplex of $\gner \T$ is geometrically represented by a diagram
in $\T$ with the shape of the 3-skeleton of an oriented standard
$p$-simplex whose 3-faces are oriented tetrahedrons
\begin{equation}\label{dnt3}
\xymatrix @C=6pt@R=6pt{
 &  F_ l \ar[dl] \ar[dd]\ar[dr]& \\
 F_ k \ar[dr] \ar@{.>}[rr] & &  F_ i ,\\
 &  F_ j \ar[ur]
 & &}
\end{equation}
one for each $0\leq i<j<k<l\leq p$, whose faces
\begin{equation}\label{dnt4}
\xymatrix@C=10pt@R=10pt{& F_ l\ar[ld]_{ F_{kl}}\ar[rd]^{ F_{jl}}
\ar@{}[d]|(0.55){  \overset{ F_{jkl_{}}}\Rightarrow}& \\  F_ k\ar[rr]_{ F_{jk}}&& F_ j,}\hspace{0.3cm}
\xymatrix@C=10pt@R=10pt{& F_ l\ar[ld]_{ F_{kl}}\ar[rd]^{ F_{il }}
\ar@{}[d]|(0.58){ \overset{  F_{ikl_{}}}\Rightarrow}& \\  F_ k\ar[rr]_{ F_{ik}}&& F_ i,}\hspace{0.3cm}
\xymatrix@C=10pt@R=10pt{& F_ l\ar[ld]_{ F_{jl}}\ar[rd]^{ F_{il}}
\ar@{}[d]|(0.58){  \overset{ F_{ijl_{}}}\Rightarrow}& \\  F_ j\ar[rr]_{ F_{ij}}&& F_ i,}\hspace{0.3cm}
\xymatrix@C=10pt@R=10pt{& F_ k\ar[ld]_{ F_{jk}}\ar[rd]^{ F_{ik}}
\ar@{}[d]|(0.55){ \overset{  F_{ijk_{}}}\Rightarrow}& \\  F_ j\ar[rr]_{ F_{ij}}&& F_ i,}
\end{equation}
are geometric 2-simplices as above, and
\begin{equation}\label{dnt5}
\xymatrix@C=8pt@R=6pt{& F_ l\ar[dd]\ar[rd]\ar[ld] &&&&
 F_ l\ar[rd]\ar[ld]\ar@{}[d]|(0.6){\Rightarrow}&\\  F_ k\ar[rd] &\ar@{}[l]|(.4){\Rightarrow}& F_ i
\ar@{}[l]|(.6){\Uparrow}&\overset{   F_{ijkl}}\Rrightarrow& F_ k\ar[rr]\ar[rd]&& F_ i \\
& F_ j\ar[ru] &&&& F_ j\ar[ru]\ar@{}[u]|(.6){
\Uparrow}&}
\end{equation}
is a 3-cell of the tricategory that labels the inner of the tetrahedron. For $p\geq 4$, these data are required to satisfy the coherence condition $(\bf{CR1})$, as stated in
 Section \ref{secrep}; that is,  for each $0\leq i<j<k<l<m\leq p$, the following diagram commutes:
$$
\xymatrix@C=-3pt@R=6pt{&&F_m\ar[dll]\ar[drr]\ar[ldd]\ar[rdd]
\ar@{}[dddddrrrrrrrrrrrrrrrrrr]|{\text{(The fourth Street's oriental,\cite{street2}})}
\ar@{}[dd]|(.6){\Rightarrow}&&&&&&&&&&&&&&&&&&F_m\ar[dll]\ar[drr]\ar[ldd]&&\\
F_l\ar[dr]\ar@{}[r]|(.9){\Rightarrow}&&&\ar@{}[r]|(.1){\Uparrow}&F_i&\ar@3{->}[rrrrrrrrrrrrr]^{F_{ijkm}}&&&&&&&&&&&&&F_l\ar[dr]\ar@{}[r]|(.9){\Rightarrow}&&\ar@{}@<-4pt>[ru]|(.1){\Uparrow}&&F_i\\
&F_k\ar[rr]&\ar@3{->}[dd]\ar@{}@<-5pt>[dd]_{F_{jklm}}&F_j\ar[ru]&&&&&&&&&&&&&&&&F_k\ar[rr]\ar[rrru]&
\ar@3{->}[dd]\ar@{}@<5pt>[dd]^{F_{iklm}}&F_j\ar[ru]\ar@{}@<8pt>[ru]|(.2){\Uparrow}&\\
&&&&&\vspace{2cm}&&&&&&&&&&&&&\\
&&F_m\ar[dll]\ar[drr]\ar[rdd]&&&&&&&&&F_m\ar[dll]\ar[drr]\ar@{}[d]|(.6){\Rightarrow}&&&&&&&&&F_m\ar[dll]\ar[drr]\ar@{}[d]|(.6){\Rightarrow}&&\\
F_l\ar@{}[rr]|(.8){\Rightarrow}\ar[dr]\ar[rrrd]&&&\ar@{}[r]|(.1){\Uparrow}&F_i\ar@3{->}[rrrrr]^{F_{ijlm}}&&&\hspace{1cm}&&F_l\ar[rrrr]\ar[rrrd]\ar[dr]&&&&F_i\ar@3{->}[rrrrr]^{F_{ijkl}}&&\hspace{1cm}
&&&F_l\ar[dr]\ar[rrrr]&&&&F_i\\
&F_k\ar[rr]\ar@{}@<3pt>[ru]|(.3){\Uparrow}&&F_j\ar[ru]&&&&&&&F_k\ar[rr]\ar@{}@<3pt>[ru]|(.3){\Uparrow}&&F_j\ar[ru]\ar@{}@<16pt>[ru]|(.35){\Uparrow}&&&&&&&F_k\ar[rr]
\ar[rrru]\ar@{}@<1pt>[ru]|(.6){\Uparrow}
&&F_j\ar[ru]\ar@{}@<8pt>[ru]|(.2){\Uparrow}&}
$$

The simplicial set $\gner\T$ is
coskeletal in dimensions greater than 4. More precisely, for $p\geq
 4$, a $p$-simplex $ F:\gner[p]\to \T$ of
$\gner\T$ is determined uniquely by its
boundary $\partial  F=( F d^0,\ldots, F d^p)$
 $$\xymatrix@C=25pt@R=15pt{\partial\Delta[p]\ar[r]^{\partial  F}\ar@{_{(}->}[d]&
 \Delta\T,\\
\Delta[p]\ar[ru]_ F&}$$
 and,  for $p\geq 5$, every possible boundary of a $p$-simplex,
  $\partial\Delta[p]\to \Delta\T$,
  is actually the boundary $\partial  F$ of a geometric $p$-simplex $ F$ of the tricategory $\T$.

\subsection{Example: Geometric nerves of bicategories}\label{exgeobicat}
  When a bicategory $\b$ is regarded as a tricategory,
all of whose 3-cells are identities, then a unitary representation of
any category $I$ in it is the same as a unitary lax functor $I\to
\b$. Hence, the simplicial set $\Delta\b$ is precisely the unitary
{\em geometric nerve of the bicategory}, as it is called in
\cite{ccg} (but denoted by $\Delta^{\hspace{-3pt}^\mathrm{u}}\b$)
where, in Theorem 6.1, the following fact is proved:

\begin{fact} \label{viii}  For any bicategory $\b$,  there is a  homotopy
equivalence $\class\b\simeq |\Delta\b|$.
\end{fact}
The construction of the geometric nerve for a bicategory was first
given in the late eighties by J. Duskin and R. Street (see
\cite[pag. 573]{street}). In \cite{duskin}, Duskin gave  a
characterization of the unitary geometric nerve
 of a bicategory $\b$ in
terms of its simplicial structure. The result states that a
simplicial set is  isomorphic to the  geometric nerve of a
bicategory if and only if it satisfies the coskeletal conditions
above as well as supporting appropriate sets of `abstractly
invertible' 1- and 2-simplices (see Gurski \cite{gurski2}, for an
interesting new approach to this subject).

\subsection{Geometric nerves realize classifying spaces of
tricategories} We now state and prove a main result of the paper,
namely:

\begin{theorem}\label{mainth} For any tricategory $\T$, there is a
 homotopy equivalence $\class\T\simeq  |\Delta\T|$.
\end{theorem}
\begin{proof}
 Let us consider, for any given tricategory $\T$,  the simplicial bicategory
\begin{equation}\label{eqrs}
\begin{array}{ll}\bgner\T:\Delta^{\mathrm{op}} \to \Hom\subset \bicat,&
[q]\mapsto \burep([q], \T),
\end{array}
\end{equation}
whose bicategories of $q$-simplices are the bicategories of unitary
representations (\ref{burep}) of $[q]$ in $\T$. In this simplicial
bicategory, the homomorphism induced by any map $a:[q]\to [p]$,
$a^*:\bgner_{p}\T \to \bgner_{q}\T$, is actually a 2-functor. Hence,
 the bisimplicial set
$$
\gner\bgner\T:\Delta^{\!^{\mathrm{op}}}\times
\Delta^{\!^{\mathrm{op}}} \to  \set,\hspace{0.4cm}
([p],[q])\mapsto \gner_p\bgner_q\T=\urep([p], \burep([q], \T)),
$$
is well defined, since the geometric nerve construction $\gner$ is
functorial on unitary homomorphisms between bicategories. The plan
is to prove the existence of  homotopy equivalences
\begin{equation}\label{he1} \class\T\simeq
|\diag \gner\bgner\T|,\end{equation}
\begin{equation}\label{he2} |\gner\T|\simeq
|\diag\gner\bgner\T|,\end{equation}  whence the
theorem follows.

\vspace{0.2cm} $\bullet$ {\em The homotopy equivalence
$(\ref{he1})$}: The Segal nerve of the tricategory  $(\ref{bsn})$ is
a simplicial sub-bicategory of $\bgner\T$. Let $L:\ner\T\to
\bgner\T$ be the pseudo-simplicial homomorphism obtained by
composing the pseudo simplicial homomorphism (\ref{cm2}), equally
denoted by $L:\ner\T\to\bsner\T$,  with the simplicial inclusion
$\bsner\T\subseteq \bgner\T$. Let us now observe that, at any degree
$p\geq 0$, the homomorphism $L_p:\ner_{p}\T\to \bgner_{p}\T$ is
precisely the homomorphism (\ref{ext}), $\burep(\mathcal{G}_p,\T)\to
\burep([p],\T)$, corresponding with the basic graph
$\mathcal{G}_p=(p\to \cdots\to 1\to 0)$ of the category $[p]$. Then,
by Lemma \ref{rufc}, we have a homomorphism (\ref{res}),
$R_p:\bgner_{p}\T\to \ner_{p}\T$, such that
$R_pL_p=1_{\ner_{p}\!\T}$, and a lax transformation
$\v_p:L_pR_p\Rightarrow 1_{\bgner_{p}\T}$. It follows from  Fact
\ref{ii} that every induced  map $\class L_p:\class
\ner_{p}\T\to\class\bgner_{p}\T$ is a homotopy equivalence. Then, by
Fact \ref{fvi}, the induced map
$\class\int_{\!\Delta}\!L:\class\!\int_{\!\Delta}\!\ner\T\to\class\!\int_{\!\Delta}\!\bgner\T$
is a homotopy equivalence. Let $\class\bgner\T:
\Delta^{\!^{\mathrm{op}}} \to \Top$ be the simplicial space obtained
by composing $\bgner\T$ with the classifying functor $\class:\Hom\to
\Top$ (see Fact \ref{i}). Since, by definition,
$\class\T=\class\!\int_{\!\Delta}\!\ner\T$, whereas, by Fact
\ref{iv}, there is a homotopy equivalence
$\class\!\int_{\!\Delta}\!\bgner\T\simeq |\class\bgner\T|$, we have
a homotopy equivalence $\class\T\simeq |\class\bgner\T|$.
Furthermore, by Fact \ref{viii} in Example \ref{exgeobicat}, we have
a homotopy equivalence
$$|\class\bgner\T|\!=\!|[q]\mapsto\class\bgner_q\T|\,
\simeq \,|[q]\mapsto|\gner \bgner_q\T||\! =\! |\diag \Delta\bgner\T
|,$$ where, for the latest equality (actually a homeomorphism) we
refer to Quillen's  Lemma in \cite[page 86]{quillen}. Thus,
$\class\T\simeq |\diag \Delta\bgner\T |$, as claimed.

\vspace{0.2cm} $\bullet$ {\em The homotopy equivalence
$(\ref{he2})$}: Note that the geometric nerve $\Delta\T$ is the
simplicial set of objects of the simplicial bicategory $\bgner\T$,
that is, $\Delta\T= \gner_0 \bgner\T$. Therefore, if we regard
$\Delta \T$ as a simplicial discrete bicategory (i.e., all 1-cells
and 2-cells are identities), then $\gner\gner\T$ becomes a
bisimplicial set that is constant in the horizontal direction, and
there is a natural bisimplicial map $\gner\gner\T\hookrightarrow
\gner \bgner\T$, which is, at each horizontal level $p\geq 1$, the
composite simplicial map
\begin{equation}\label{1.2.28'}\Delta\T=  \Delta_0
\bgner\T\overset{s_0^{\mathrm h}}\hookrightarrow
\Delta_1 \bgner\T\hookrightarrow
\cdots\!\hookrightarrow \Delta_{p-1}
\bgner\T\overset{s_{p-1}^{\mathrm h}}\hookrightarrow
\Delta_p \bgner\T.\end{equation}

Next, we  prove that the  simplicial map $ \Delta\T\to
\diag\Delta\bgner\T$, induced on diagonals,
  is a  weak homotopy equivalence, whence the announced
 homotopy equivalence in (\ref{he2}). It suffices to  prove that every  one of
 the
simplicial maps in (\ref{1.2.28'}) is a weak homotopy equivalence
and, in fact, we will prove more: {\em Every simplicial map
$s_{p-1}^{\mathrm h}$, $p\geq 1$, embeds the simplicial set
$\gner_{p-1} \bgner\T$ into
$\gner_p \bgner\T$ as a simplicial
deformation retract}.  Since $d_p^{\mathrm h}s_{p-1}^{\mathrm h}=1$,
it is enough to exhibit a simplicial homotopy $ h: s_{p-1}^{\mathrm
h}d_p^{\mathrm h}\Rightarrow 1:\Delta_p \bgner\T\to\Delta_p
\bgner\T$.

To do so, we shall use the following notation for the bisimplices in
$\Delta\bgner\T$. Since  such a bisimplex of bidegree $(p,q)$, say
$F\in \Delta_p \bgner_q\T$, is  a unitary representation of the
category $[p]$ in the bicategory of unitary representations
$\burep([q],\T)$, it consists of unitary representations $F^u:[q]\to
\T$, 1-cells
  $F^{u,v}:F^{v}\Rightarrow F^{u}$,  and 2-cells $F^{u,v,w}\!:F^{u,v}\circ F^{v,w}
   \Rrightarrow F^{u,w}$ in the bicategory $\burep([q],\T)$, for  $0\leq u< v<w\leq p$,
   such
that  the diagrams
$$\xymatrix@R=20pt@C=5pt{(F^{u,v}\circ F^{v,w})\circ
F^{w,t}\ar@{=>}[d]_{ F^{u,v,w}\circ\,
1}\ar@{=>}[rr]^{
{\boldsymbol{a}}}&&F^{u,v}\circ (F^{v,w}\circ F^{w,t})\ar@{=>}[d]^{ 1\circ
F^{v,w,t}}\\ F^{u,w}\circ F^{w,t}\ar@{=>}[r]^(0.6){
F^{u,w,t}}&F^{u,t}& F^{u,v}\circ F^{v,t}\ar@{=>}[l]_(0.53){
F^{u,v,t}}
 }$$
commute for $u<v< w< t$. Hence, such a $(p,q)$-simplex is described
by a list of cells of the tricategory $\T$
\begin{equation}\label{pq}F=\big(F_i,
F^{u}_{i,j},   F^{u}_{i,j,k}, F^{u}_{i,j,k,l}, F^{u,v}_{i,j},
F^{u,v}_{i,j,k}, F^{u,v,w}_{i,j} \big),\end{equation} with   $0\leq i<j<
k< l\leq q$, where each $F_i$ ($=F^{^{_0}}_i$) is an object, the
$F^{u}_{i,j}:F_j\to F_i$  are 1-cells,
 the $F^{u}_{i,j,k}:F^{u}_{i,j}\otimes F^{u}_{j,k}\Rightarrow
 F^{u}_{i,k}$,
and $F^{u,v}_{i,j}:F^{v}_{i,j}\Rightarrow F^{u}_{i,j}$ are 2-cells,
and the remaining
 are 3-cells as in
$$\xymatrix@R=20pt@C=10pt{(F^{u}_{i,j}\otimes F^{u}_{j,k})\otimes F^{u}_{k,l}\ar@{=>}[d]_{ F^{u}_{i,j,k}
\otimes 1}\ar@{}[rrd]|{ {
 F^{u}_{i,j,k,l}}\,\Rrightarrow}\ar@{=>}[rr]^{
{\boldsymbol{a}}}&&F^{u}_{i,j}\!\otimes\!
(F^{u}_{j,k}\!\otimes\!F^{u}_{k,l})\ar@{=>}[d]^{ 1\otimes
F^u_{j,k,l}}\\ F^{u}_{i,k}\!\otimes\!F^{u}_{k,l}\ar@{=>}[r]^{
F^{u}_{i,k,l}}&F^{u}_{i,l}&
F^{u}_{i,j}\!\otimes\!F^{u}_{j,l} ,\ar@{=>}[l]_{\
F^{u}_{i,j,l}}
 }
$$
$$
\xymatrix@R=20pt@C=42pt{F^{v}_{i,j}\!\otimes\!F^{v}_{j,k}\ar@{=>}[d]_{F^{v}_{i,j,k}}
 \ar@{}@<-8pt>[rd]^(.5){ { F^{^{_{u,v}}}_{i,j,k}}\,\Rrightarrow}
\ar@{=>}[r]^{F^{u,v}_{i,j}\otimes
F^{u,v}_{j,k}}&F^{u}_{i,j}\!\otimes\!F^{u}_{j,k}\ar@{=>}[d]^{F^{u}_{i,j,k}}\\
F^{u}_{i,k}
 \ar@{=>}[r]_{F^{u,v}_{i,k}}&F^{u}_{i,k},
 }
 \hspace{0.4cm}
   \xymatrix@R=20pt@C=25pt{&F^{w}_{i,j}\ar@{=>}[rd]^{F^{u,w}_{i,j}}
\ar@{}[d]|(.6){ { F^{u,v,w}_{i,j}}\,\Rrightarrow}
   \ar@{=>}[ld]_{F^{v,w}_{i,j}}&\\F^{v}_{i,j}\ar@{=>}[rr]_{F^{u,v}_{i,j}}
 &&
 F^{u}_{i,j},}
 $$
satisfying  the various conditions. The horizontal faces and
degeneracies of  such a bisimplex (\ref{pq})  are given by the simple
rules $d_m^{\mathrm h}F=(F_i,F^{d^m\!u}_{i,j},\dots)$
 and $s_m^{\mathrm h}F=(F_i,
F^{s^m\!u}_{i,j},\dots)$, whereas the vertical ones are given by
$d_m^{\mathrm v}F=(F_{d^mi}, F^{u}_{d^m\!i,d^m\!j},\dots)$ and
$s_m^{\mathrm v}F=(F_{s^mi}, F^{u}_{s^m\!i,s^m\!j},\dots)$.

  Then, we have  the following
simplicial homotopy $ h: s_{p-1}^{\mathrm h}d_p^{\mathrm
h}\Rightarrow 1:\Delta_p \bgner\T\to\Delta_p \bgner\T$. For each
$0\leq m\leq q$, the map $h_m:\Delta_p \bgner_q\T\to \Delta_p
\bgner_{q+1}\T$ takes a $(p,q)$-simplex (\ref{pq}) of $\Delta
\bgner\T$ to the $(p,q+1)$-simplex $h_mF$ defined by

\vspace{0.2cm} \noindent-  $(h_mF)_i=F_{s^mi}$,\hspace{0.3cm} for
$0\leq i\leq q+1$,

\vspace{0.2cm}\noindent-  $(h_mF)^{u}_{i,j}=F^{u}_{s^m\!i,s^m\!j}$
\hspace{0.3cm} if $u<p$ or $j\leq m$,

\vspace{0.2cm}\noindent- $(h_mF)^{p}_{i,j}= F^{p-1}_{s^m\!i,j-1}$
\hspace{0.3cm} if $ m<j$,

\vspace{0.2cm} \noindent-
$(h_mF)^u_{i,j,k}=F^u_{s^m\!i,s^m\!j,s^m\!k}$ \hspace{0.3cm} if
$u<p$ or $k\leq m$,

\vspace{0.2cm} \noindent- $(h_mF)^{p}_{i,j,k}=
F^{p-1}_{s^m\!i,j-1,k-1}$ \hspace{0.3cm}if $m<j$,

\vspace{0.2cm} \noindent- $(h_mF)^{p}_{i,j,k}$, \ for $j\leq m<k$,
 is the 2-cell obtained by the composition

\hspace{2cm}$
\xymatrix@C=35pt{F^{p}_{i,j}\!\otimes\!F^{p-1}_{j,k-1}\ar@{=>}[r]^-{F^{p-1,p}_{i,j}\otimes
1}&F^{p-1}_{i,j}\!\otimes\!F^{p-1}_{j,k-1}
\ar@{=>}[r]^-{F^{p-1}_{i,j,k-1}}&F^{p-1}_{i,k-1},} $

\vspace{0.2cm}\noindent-
$(h_mF)^u_{i,j,k,l}=F^u_{s^m\!i,s^m\!j,s^m\!k,s^m\!l}$
\hspace{0.3cm} if $u<p$ or $l\leq m$,

\vspace{0.2cm}\noindent-
$(h_mF)^p_{i,j,k,l}=F^{p-1}_{s^mi,j-1,k-1,l-1}$ \hspace{0.3cm} if
$m<j$,

\vspace{0.2cm}\noindent- $(h_mF)^p_{i,j,k,l}$,\ for $j\leq m<k$,  is
the 3-cell obtained by pasting the diagram
$$
\xymatrix@R=22pt{(F^{p}_{i,j}\!\otimes\!F^{p-1}_{j,k-1})\!\otimes\!F^{p-1}_{k-1,l-1}\ar@{}@<-4pt>[rd]|(.4){
\cong}\ar@{=>}[r]^{\boldsymbol{a}}
\ar@{=>}[d]_{(F^{p-1,p}_{i,j}\otimes 1)\otimes 1}&
F^{p}_{i,j}\!\otimes\!(F^{p-1}_{j,k-1}\!\otimes\!F^{p-1}_{k-1,l-1})\ar@{=>}[d]_{F^{p-1,p}_{i,j}\otimes
(1\otimes 1)} \ar@{=>}[rd]^{\ \ 1\otimes
F^{p-1}_{j,k-1,l-1}}&\ar@{}[ldd]^(0.45) {
\cong}\\
(F^{p-1}_{i,j}\!\otimes\!F^{p-1}_{j,k-1})\!\otimes\!F^{p-1}_{k-1,l-1}
\ar@{}@<-5pt>[rrd]|(0.35){{
 F^{p-1}_{i,j,k-1,l-1}}\,\Rrightarrow}\ar@{=>}[r]^{\boldsymbol{a}}
\ar@{=>}[d]_{F^{p-1}_{i,j,k-1}\otimes 1}&
F^{p-1}_{i,j}\!\otimes\!(F^{p-1}_{j,k-1}\!\otimes\!F^{p-1}_{k-1,l-1}
)\ar@{=>}[rd]_(.3){1\otimes
F^{p-1}_{j,k-1,l-1}}&F^{p}_{i,j}\!\otimes\!F^{p-1}_{j,l-1}
\ar@{=>}[d]^{F^{p-1,p}_{i,j}\otimes 1}\\
F^{p-1}_{i,k-1}\!\otimes\!F^{p-1}_{k-1,l-1}\ar@{=>}[r]^(.4){F^{p-1}_{i,k-1,l-1}}&F^{p-1}_{i,l-1}&
F^{p-1}_{i,j}\!\otimes\!F^{p-1}_{j,l-1} \ar@{=>}[l]_{F^{p-1}_{i,j,l-1}} }
$$

\vspace{0.2cm}\noindent- $(h_mF)^p_{i,j,k,l}$, \  for $k\leq m<l$,
is the 3-cell obtained by pasting in
$$
\xymatrix@C=10pt@R=15pt{&
(F^{p}_{i,j}\!\otimes\!F^{p}_{j,k})\!\otimes\!F^{p-1}_{k,l-1}
\ar@{}[rdd]_{ \cong}
\ar@{=>}[dd]^(0.3){(F^{p-1,p}_{i,j}\otimes F^{p-1,p}_{j,k})\otimes 1}
\ar@{=>}[r]^{\boldsymbol a}
\ar@{=>}[ld]_-{F^p_{i,j,k}\otimes 1}&F^{p}_{i,j}\!\otimes\!
(F^{p}_{j,k}\!\otimes\!F^{p-1}_{k,l-1})
\ar@{=>}[dd]|(0.7){F^{p-1,p}_{i,j}\otimes (F^{p-1,p}_{j,k}\otimes 1)}
\ar@{=>}[rd]^-{\ 1\otimes(F^{p-1,p}_{j,k}\otimes
1)}&\\
F^p_{i,k}\!\otimes\!F^{p-1}_{k,l-1}
\ar@{}@<-5pt>[rd]^{ { F^{p-1,p}_{i,j,k}\otimes 1}\,\Rrightarrow}
\ar@{=>}[dd]_{F^{p-1,p}_{i,k}\otimes 1}
 &&&F^{p}_{i,j}\!\otimes\!
(F^{p-1}_{j,k}\!\otimes\!F^{p-1}_{k,l-1})
\ar@{=>}[dd]^{1\otimes F^{p-1}_{j,k,l-1}}\\
&(F^{p-1}_{i,j}\!\otimes\!F^{p-1}_{j,k})\!\otimes\!F^{p-1}_{k,l-1}
\ar@{=>}[ld]^{F^{p-1}_{i,j,k}\otimes  1}
\ar@{=>}[r]^{\boldsymbol a}
&F^{p-1}_{i,j}\!\otimes\!
(F^{p-1}_{j,k}\!\otimes\!F^{p-1}_{k,l-1})
\ar@{=>}[dd]^{1\otimes F^{p-1}_{j,k,l-1}}\ar@{}[r]|(0.6){ \cong}&\\
F^{p-1}_{i,k}\!\otimes\!F^{p-1}_{k,l-1}
\ar@{}[rr]|(0.6){ { F^{p-1}_{i,j,k,l-1}}\,\Rrightarrow}
\ar@{=>}[rd]_{F^{p-1}_{i,k,l-1}}
&&&F^p_{i,j}\!\otimes\!F^{p-1}_{j,l-1}
\ar@{=>}[ld]^{\ F^{p-1,p}_{i,j}\otimes 1}\\
&F^{p-1}_{i,l-1}&F^{p-1}_{i,j}\!\otimes\!F^{p-1}_{j,l-1}
\ar@{=>}[l]_{F^{p-1}_{i,j,l-1}}&
}
$$
\noindent- $(h_mF)^{u,v}_{i,j}=F^{u,v}_{s^m\!i,s^m\!j}$
\hspace{0.3cm} if $v<p$ or $j\leq m$,

\vspace{0.2cm}\noindent- $(h_mF)^{u,p}_{i,j}=F^{u,p-1}_{s^m\!i,j-1}$
\hspace{0.3cm} if $j>m$,

\vspace{0.2cm}\noindent- $(h_mF)^{u,v}_{i,j,k}=
F^{u,v}_{s^m\!i,s^m\!j,s^m\!k}$ \hspace{0.3cm} if $v<p$ or $k\leq
m$,

\vspace{0.2cm}\noindent- $(h_mF)^{u,p}_{i,j,k}=
F^{u,p-1}_{s^m\!i,j-1,k-1}$ \hspace{0.2cm} if $m<j$,

\vspace{0.2cm}\noindent- $(h_mF)^{u,p}_{i,j,k}$, for $j\leq m<k$, is
the 3-cell is the obtained by pasting in
$$
\xymatrix@R=20pt@C=30pt{&F^p_{i,j}\!\otimes\!F^{p-1}_{j,k-1}\ar@{=>}[rd]^{\ F^{u,p}_{i,j}\otimes
F^{u,p-1}_{j,k-1}}\ar@{=>}[ld]_{F^{p-1,p}_{i,j}\otimes 1}&\\F^{p-1}_{i,j}\!\otimes\!F^{p-1}_{j,k-1}
\ar@{=>}[rr]_(.35){F^{u,p-1}_{i,j}\otimes F^{u,p-1}_{j,k-1}}\ar@{=>}[d]_{F^{p-1}_{i,j,k-1}}
\ar@{}@<8pt>[rr]^{ { F^{u,p-1,p}_{i,j}}\,\Rrightarrow}
&&F^{u}_{i,j}\!\otimes\!F^{u}_{j,k-1}\ar@{=>}[d]^{F^{u}_{i,j,k-1}}\\
F^{p-1}_{i,k-1}\ar@{=>}[rr]^(.35){F^{u,p-1}_{i,k-1}}
\ar@{}@<-4pt>[rru]|(.6){ { F^{u,p-1}_{i,j,k-1}}\,\Rrightarrow}&&F^{u}_{i,k-1}
}
$$
\noindent - $(h_mF)^{u,v}_i=F^{u,v}_{s^m\!i}$ \hspace{0.3cm} if
$v<p$ or $i\leq m$,

\vspace{0.2cm}\noindent-  $(h_mF)^{u,p}_{i}=F^{u,p-1}_{i-1}$
\hspace{0.3cm} if $m<i$;

\vspace{0.2cm} \noindent-
$(h_mF)^{u,v,w}_{i,j}=F^{u,v,w}_{s^m\!i,s^m\!j}$ \hspace{0.3cm} if
$w<p$ or $j\leq m$,

\vspace{0.2cm}\noindent-
$(h_mF)^{u,v,p}_{i,j}=F^{u,v,p-1}_{s^m\!i,j-1}$\hspace{0.3cm} if
$m<j$.

So defined, a straightforward (though quite tedious) verification
shows that ${h:s_{p-1}^{\mathrm h}d_p^{\mathrm h}\Rightarrow 1}$ is
actually a simplicial homotopy, and this completes the proof.
\end{proof}

\subsection{Example: Geometric nerves of $n$-categories}  On a small category $C$, viewed as a tricategory in
which all 2-cells and 3-cells are identities,  both the geometric
and the Grothendieck nerve constructions can be identified: $\Delta
C=\ner C$,  since, for any integer $p\geq 0$, we have
$\urep([p],C)=\func([p],C)\cong
\rep(\mathcal{G}_p,C)$. For instance
$\Delta[n]$ is the
simplicial standard $n$-simplex whose $p$-simplices are the maps
$[p]\to[n]$ in $\Delta$.

The case of geometric nerves of 2-categories was dealt with by
Street in \cite{street2}.  In \cite{tonks}, but under the name of
`2-nerve of 2-categories', the geometric nerve construction
$\Delta:\mathbf{2}\text{-}\cat\to \sset $ was considered for proving
that the category $\mathbf{2}\text{-}\cat$, of small 2-categories
and 2-functors, has a Quillen model structure such that the functor
$$Ex^2\,\Delta:\mathbf{2}\text{-}\cat\to
\sset$$ is a right Quillen equivalence of model categories, where
$Ex$ is the endofunctor in $\sset$ right adjoint to the subdivision
$Sd$ (see \cite{g-j}, for example). In \cite{b-c},  it was first
proved that, for any 2-category $\b$, there is a natural homotopy
equivalence $\class\b\simeq |\Delta\b|$. It follows that the
correspondence $\b\mapsto\class\b$ induces an equivalence between
the corresponding homotopy category of 2-categories and the ordinary
homotopy category of CW-complexes. In \cite{cegarra}, generalizations
are given of both Quillen's Theorem B and Thomason's Homotopy
Colimit Theorem to 2-functors.

In \cite{street2}, Street gave a precise notion of nerve for 
$n$-categories. He extended each graph $\mathcal{G}_p=(p\to
\cdots\to 1\to 0)$ to a ``free" $\omega$-category $\mathcal{O}_p$
(called the {\em $p^{\mathrm{th}}$-oriental}) such that, for any
$n$-category $\mathcal{X}$, the $p$-simplices of its nerve, are just
$n$-functors $\mathcal{O}_p\to\mathcal{X}$, from the underlying
$n$-category of the $p^{\mathrm{th}}$-oriental to $\mathcal{X}$. In
the case when $n=3$, Street's nerve construction on a $3$-category
$\T$ just produces, up to some directional changes, its geometric
nerve $\Delta\T$, as stated in Definition \ref{goner}.  From the
homotopy equivalences in  (\ref{neqt}) and Theorem \ref{mainth}, we
get the following new result (see Example \ref{ex3cat} for a
discussion about the notion of classifying space of a 3-category):
\begin{theorem}
For any $3$-category $\T$, there are homotopy equivalences
$$|\diag\ner\ner\ner\T|\simeq \class\T\simeq |\Delta\T|.$$
\end{theorem}

\subsection{Example: Geometric nerves of braided monoidal
 categories} \label{nebmc} If $A$ is any abelian group, then the
braided monoidal category with only one object it defines,
$(A,+,0)$, has as
 double suspension the tricategory $\Sigma^2A$, treated in Examples \ref{eta} and
 \ref{eta2}. For any integer $p\geq 0$, we have
$$\urep([p],\Sigma^2A)\overset{\ref{eta}}= Z^3([p],A)=Z^3(\Delta[p],A)=K(A,3)_p,$$
whence $\Delta\Sigma^2 A=K(A,3)$,  the minimal Eilenberg-Mac Lane
complex. Hence, from Theorems \ref{csmb} and \ref{mainth}, it
follows that $\class(A,+,0)=|K(A,3)|$.

If $(C,\otimes,\boldsymbol{c})$ is any braided monoidal category,
then a unitary representation of a category $I$ in the double
suspension tricategory, $I\to \Sigma^2(C,\otimes,\boldsymbol{c})$,
is what was called in \cite[Definition 6.6]{c-c-g} and \cite[\S
4]{cegarra3} a (normal) {\em $3$-cocycle} of $I$ with coefficients
in the braided monoidal category. Therefore, the geometric nerve
$\Delta\Sigma^2(C,\otimes,\boldsymbol{c})$ coincides with the
simplicial set \cite[Definition 6.7]{c-c-g}
$$Z^3(C,\otimes,\boldsymbol{c}):\Delta^{\!^{\mathrm{op}}}\
\to \ \set,\hspace{0.4cm}[p]\mapsto
Z^3([p],(C,\otimes,\boldsymbol{c})),
$$
whose $p$-simplices are the 3-cocycles of $[p]$ in the braided
monoidal category. The geometric nerve $Z^3(C,\otimes,
\boldsymbol{c})$ is then a $4$-coskeletal 1-reduced (one vertex, one
1-simplex) simplicial set whose $2$-simplices are the objects
$F_{012}$ of $C$. The 3-simplices are morphisms of the form
$$
F_{0123}:F_{123}\otimes F_{013}\longrightarrow F_{012}\otimes F_{023},
$$  and, for $p\geq 4$, the $p$-simplices are families
of 3-simplices
$$ F_{ijkl}:F_{jkl}\otimes F_{ijl}\to F_{ijk}\otimes F_{ikl}
,$$
$0\leq i<j<k<l\leq p$, making commutative, for $0\leq i<j<k<l<m\leq
p$, the diagrams
$$\xymatrix@C=80pt@R=22pt{(F_{klm}\otimes F_{jkm})\otimes F_{ijm}\ar[d]_{ F_{jklm}\otimes 1}
\ar[r]^{ \boldsymbol{a}^{-1}(1\otimes F_{ijkm})\boldsymbol{a}}&
(F_{klm}\otimes F_{ijk})\otimes F_{ikm}\ar[d]^{\boldsymbol{c}\otimes 1}\\
(F_{jkl}\otimes F_{jlm})\otimes F_{ijm}\ar[d]_{(1\otimes F_{ijlm})\boldsymbol{a}}&
 (F_{ijk}\otimes F_{klm})\otimes F_{ikm}\ar[d]^{(1\otimes F_{iklm})\boldsymbol{a}}\\
F_{jkl}\otimes (F_{ijl}\otimes F_{ilm})\ar[r]^{ \boldsymbol{a}(F_{ijkl}\otimes 1)
\boldsymbol{a}^{-1}}& F_{ijk}\otimes (F_{ikl}\otimes F_{ilm}).}
$$

From Theorems \ref{csmb} and \ref{mainth}, we obtain the following
 known result:
\begin{theorem}[{\cite[Theorem 6.11]{c-c-g}}]
For any braided monoidal category $(C,\otimes,\boldsymbol{c})$,
there is a  homotopy equivalence
$\class(C,\otimes,\boldsymbol{c})\simeq
|Z^3(C,\otimes,\boldsymbol{c})|.$
\end{theorem}

\subsection{Example: Geometric nerves of monoidal bicategories} If $(\b,\otimes)$ is any
monoidal bicategory, then we define its {\em geometric nerve},
denoted by
$$\gner(\b,\otimes),
$$
as the geometric nerve of its suspension 3-category
$\Sigma(\b,\otimes)$ (\ref{ome}). That is,
$$\gner(\b,\otimes):\Delta^{\mathrm{op}}\to \sset, \hspace{0.3cm} [p]\mapsto
\urep([p],\Sigma(\b,\otimes)).
$$

Then, Theorem $\ref{mainth}$ particularizes to monoidal bicategories
stating the following:
\begin{theorem}\label{cbbb} For any monoidal bicategory $(\b,\otimes)$, there is a homotopy equivalence
\begin{equation}\class(\b,\otimes)\simeq |\gner(\b,\otimes)|.\end{equation}
\end{theorem}
\subsection{Example: Bicategorical groups and homotopy 3-types} A {\em
bicategorical group}, also called a  (weak) {\em $3$-group}, or {\em
Gr-bicategory},  is a monoidal bicategory $(\b,\otimes)$, in which
every 2-cell is an isomorphism;
 that is, all the hom-categories $\b(x,y)$ are groupoids, every morphism $u:x\to y$ is an equivalence,
 that is,
 there exist a morphism $u':y\to x$ and 2-cells $u\circ u'\Rightarrow 1_y$ and
 $1_x\Rightarrow u'\circ u$, and each object $x$ has a quasi-inverse with respect to the tensor product;
 that is,  there is an object $x'$ with 1-cells $1\to
x\otimes x'$ and $x'\otimes x\to 1$. In other words, a bicategorical
group is a monoidal bigroupoid $(\b,\otimes)$ in which every object
$x$ is regular, in the sense that the  homomorphisms $x\otimes
-:\b\to \b$ and $-\otimes x:\b\to \b$ are biequivalences.

If $(\b,\otimes)$ is any monoidal bicategory, then its geometric
nerve $\gner(\b,\otimes)$ is
 a 4-coskeletal reduced (one vertex) simplicial set, which satisfies the Kan extension  condition
 if and only if $(\b,\otimes)$ is a bicategorical group. In such a case,
    the homotopy groups of its  geometric realization $\pi_i\class(\b,\otimes)=\pi_i\gner(\b,\otimes)$ are plainly recognized
 to be
\begin{itemize}
\item[-] $\pi_i\class(\b,\otimes)= 0$,\ \, if $i \neq 1,2,3$.
\vspace{0.1cm}
\item[-] $\pi_1\class(\b,\otimes)=\mathrm{Ob}\b/\!\thicksim$, \ the group of equivalence classes of objects in $\b$ where multiplication is induced by the tensor product.
    \vspace{0.1cm}
\item[-] $\pi_2\class(\b,\otimes)=\mathrm{Aut}_\b(1)/\!\cong$, \ the group of isomorphism classes of autoequivalences of the unit object where the operation is induced by the horizontal composition in $\b$.
    \vspace{0.1cm}
\item[-]  $\pi_3\class(\b,\otimes)=\mathrm{Aut}_\b(1_1)$,\  the group of automorphisms of the identity 1-cell of the unit object where the operation is vertical composition in $\b$.
\end{itemize}

Thus, bicategorical groups arise as algebraic path-connected
homotopy 3-types, a fact that supports the {\em Homotopy Hypothesis}
of Baez \cite{baezh}. Indeed, every  path-connected homotopy 3-type
can be realized in this way from a bicategorical group. That is, for
any given path-connected CW-complex $X$ for which $\pi_iX=0$ for
$i\geq 4$, there is a bicategorical group $(\B,\otimes)$ such that
$\class(\b,\otimes)$ is homotopy equivalent to $X$, as we show below
(cf. Berger \cite{berger}, Joyal-Tierney \cite{j-t}, Lack
\cite{lack2}, or Leroy \cite{leroy}, for alternative approaches to
this issue).

Given $X$ as above, let $M\subseteq S(X)$ be a minimal subcomplex
that is a deformation retract of the total singular complex of $X$,
so that $|M|\simeq X$. Taking into account the Postnikov
$k$-invariants, this minimal complex $M$ can be described (see
\cite[VI. Corollary 5.13]{g-j}), up to isomorphism, 
\begin{equation}\label{fff}M=K(B,3)\!\times_t\!(K(A,2)\!\times_h\!K(G,1)),\end{equation}
by means of the group $G=\pi_1X$, the $G$-modules $A=\pi_2X$ and
$B=\pi_3X$, and two maps, $$h:G^3\to A, \hspace{0.3cm}t:A^6\times
G^4\to B,$$ defining normalized cocycles $h\in Z^3(G,A)$  and  $t\in
Z^4(K(A,2)\!\times_h\!K(G,1),B)$. That is, $M$ is the $4^{\mathrm
{th}}$ coskeleton of the truncated simplicial set
$$\mathrm{tr}_4M= \xymatrix{B^4\!\times\! A^6\!\times\! G^4\ar@<4pt>[r]\ar@{}@<2pt>[r]^{d_0}\ar@{}@<-2pt>[r]_{d_4}\ar@{}[r]|{\cdots}\ar@<-4pt>[r]&B\!\times\!A^3\!\times\!G^3
\ar@/_1.2pc/[l] \ar@{}@<-20pt>[l]^(.3){s_0}\ar@/_2.2pc/[l] \ar@{}@<-27pt>[l]^(.2){s_3} \ar@{}@<-24pt>[l]^(.5){\cdots}
\ar@<4pt>[r]\ar@{}@<2pt>[r]^{d_0}\ar@{}@<-1pt>[r]_{d_3}\ar@{}[r]|{\cdots}\ar@<-3pt>[r]& \ar@/_1.2pc/[l] \ar@{}@<-20pt>[l]^(.3){s_0}\ar@/_2.2pc/[l] \ar@{}@<-27pt>[l]^(.2){s_2} A\!\times\! G^2 \ar@/_1.8pc/[l] \ar@<3pt>[r]\ar@{}@<1pt>[r]^{d_0}\ar@{}@<-1pt>[r]_{d_2}\ar[r]\ar@<-3pt>[r]&G
\ar@/_1.2pc/[l] \ar@{}@<-20.5pt>[l]^(.3){s_0}\ar@/_2pc/[l] \ar@{}@<-30pt>[l]^(.3){s_1}\ar@<2pt>[r]\ar@{}@<-0.5pt>[r]^{d_0}
\ar@<-2pt>[r]\ar@{}@<0.5pt>[r]_{d_1}& 1, \ar@/_1.2pc/[l] \ar@{}@<-20.5pt>[l]^(.3){s_0}
}
$$
whose face and degeneracy operators are given by ($\sigma_i\in G$, $
x_j\in A$, $u_k\in B$)

$
d_i(x_1,\sigma_1,\sigma_2)=\left\{\begin{array}{ll}\sigma_2&i=0,\\\sigma_1\sigma_2&i=1,\\\sigma_1&i=2.\end{array}\right.
$

$
d_i(u_1,x_1,x_2,x_3,\sigma_1,\sigma_2,\sigma_3)=\left\{\begin{array}{ll}({}^{\sigma_1^{-1}\!}\!x_3,\sigma_2,\sigma_3)&i=0,\\
(x_2+x_3, \sigma_1\sigma_2,\sigma_3)&i=1,\\
(x_1+x_2, \sigma_1,\sigma_2\sigma_3)&i=2,\\
(x_1-h(\sigma_1,\sigma_2,\sigma_3), \sigma_1,\sigma_2)&i=3.
\end{array}\right.
$

$ d_i(u_1,u_2,u_3,u_4,
x_1,x_2,x_3,x_4,x_5,x_6,\sigma_1,\sigma_2,\sigma_3,\sigma_4)=$
$$
\left\{\begin{array}{ll}({}^{\sigma_1^{-1}}\!u_4, {}^{ \sigma_1^{-1}\!}\!x_4,{}^{\sigma_1^{-1}\!}\!x_5, {}^{\sigma_1^{-1}\!}\!x_6,\sigma_2,\sigma_3,\sigma_4)&i=0,\\
(u_3+u_4,x_2+x_4,x_3+x_5,x_6, \sigma_1\sigma_2,\sigma_3,\sigma_4)&i=1,\\
(u_2+u_3,x_1+x_2,x_3,x_5+x_6, \sigma_1,\sigma_2\sigma_3,\sigma_4)&i=2,\\
(u_1+u_2,x_1,x_2+x_3,x_4+x_5, \sigma_1,\sigma_2,\sigma_3\sigma_4)&i=3\\
(\bar{u}_1, \bar{x}_1,\bar{x}_2,\bar{x}_3,\sigma_1,\sigma_2,\sigma_3)&i=4,
\end{array}\right.
$$
where
$\bar{u}_1=u_1-t(x_1,x_2,x_3,x_4,x_5,x_6,\sigma_1,\sigma_2,\sigma_3,\sigma_4)$,
$\bar{x}_1=x_1-h(\sigma_1,\sigma_2,\sigma_3\sigma_4)+
h(\sigma_1,\sigma_2,\sigma_3)$,
$\bar{x}_2=x_2-h(\sigma_1\sigma_2,\sigma_3,\sigma_4)+{}^{\sigma_1^{-1}}\!h(\sigma_2,\sigma_3,\sigma_4)$,
and $\bar{x}_3=
x_4-{}^{\sigma_1^{-1}}\!h(\sigma_2,\sigma_3,\sigma_4)$.

$s_i(\sigma_1)=\left\{\begin{array}{ll}(0,1,\sigma_1)&i=0,\\
(0,\sigma_1,1)&i=1.\end{array}\right. $

$ s_i(x_1,\sigma_1,\sigma_2)=\left\{\begin{array}{ll}
(0,0,0,x_1,1,\sigma_1,\sigma_2)& i=0,\\
(0,0,x_1,0,\sigma_1,1,\sigma_2)& i=1,\\
(0,x_1,0,0,\sigma_1,\sigma_2,1)& i=2.
\end{array}\right.
$

$
s_i(u_1,x_1,x_2,x_3,\sigma_1,\sigma_2,\sigma_3)=\left\{\begin{array}{ll}(0,0,0,u_1,0,0,0,x_1,x_2,x_3,1,\sigma_1,\sigma_2,\sigma_3)&i=0,\\
(0,0,u_1,0,0,x_1,x_2,0,0,x_3,\sigma_1,1,\sigma_2,\sigma_3)&i=1,\\
(0,u_1,0,0,x_1,0,x_2,0,x_3,0,\sigma_1,\sigma_2,1,\sigma_3)&i=2,\\
(u_1,0,0,0,x_1,x_2,0,x_3,0,0,\sigma_1,\sigma_2,\sigma_3,1)&i=3.\\
\end{array}\right.$

Then, a bicategorical group $(\b,\otimes)$ with a simplicial
isomorphism $\gner(\b,\otimes)\cong M$ is defined as follows: a
0-cell of $\b$ is an element $\sigma\in G$. If $\sigma\neq \tau$ are
different elements of $G$, then $\b(\sigma,\tau)=\emptyset$, that
is, there is no 1-cell  between them, whereas if $\sigma=\tau$, then
1-cell $x:\sigma\to\sigma$ is an element $x\in A$. Similarly, there
is no any 2-cell in $\b$ between two 1-cells $x,y:\sigma\to \sigma$
if  $x\neq y$, whereas, when $x=y$,  a 2-cell $u:x\Rightarrow x$ is an
element $u\in B$. The vertical composition of 2-cells is given by addition in
$B$, that is,
$$
(x\overset{u}\Longrightarrow x)\cdot (x\overset{v}\Longrightarrow x)=(x\overset{u+v}\Longrightarrow x).
$$
The horizontal composition of 1-cells and 2-cells is given by
addition in $A$ and $B$ respectively, that is,
$$\xymatrix{(\sigma\ar@/^0.7pc/[r]^{x}\ar@{}[r]|{\Downarrow u}
\ar@/_0.7pc/[r]_{x}& \sigma)}\hspace{-2pt}\circ\hspace{-2pt}
\xymatrix{(\sigma\ar@/^0.7pc/[r]^{y}\ar@{}[r]|{\Downarrow v}
\ar@/_0.7pc/[r]_{y}& \sigma)}=\xymatrix@C=30pt{(\sigma\ar@/^0.8pc/[r]^{x+y}\ar@{}[r]|{\Downarrow u+v}
\ar@/_0.8pc/[r]_{x+y}& \sigma).}
$$
The associativity isomorphism is
$$
\xymatrix@C=35pt{\sigma\ar@/^0.8pc/[r]^{(x+y)+z}\ar@{}[r]|{\Downarrow \boldsymbol{a}}
\ar@/_0.8pc/[r]_{x+(y+z)}& \sigma,&\hspace{-1cm}\boldsymbol{a}=t(x,y,z,0,0,0,\sigma,1,1,1),}
$$
and the 0 of A gives the (strict) unit on each $\sigma$, that is,
$1_\sigma=0:\sigma\to\sigma$.

The (strictly unitary) tensor homomorphism $\otimes:\b\times
\b\to\b$ is given on cells of $\b$ by
$$\xymatrix{(\sigma\ar@/^0.7pc/[r]^{x}\ar@{}[r]|{\Downarrow u}
\ar@/_0.7pc/[r]_{x}& \sigma)}\hspace{-2pt}\otimes\hspace{-2pt}
\xymatrix{(\tau\ar@/^0.7pc/[r]^{y}\ar@{}[r]|{\Downarrow v}
\ar@/_0.7pc/[r]_{y}& \tau)}=\xymatrix@C=30pt{(\sigma\tau\ar@/^0.8pc/[r]^{x+{}^{\sigma}\!y}
\ar@{}[r]|{\Downarrow u+{}^{\sigma}\!v}
\ar@/_0.8pc/[r]_{x+{}^{\sigma}\!y}& \sigma\tau),}
$$
and the structure interchange isomorphism, for any 1-cells
$\sigma\overset{x'}\to\sigma\overset{x}\to\sigma$ and
$\tau\overset{y'}\to\tau\overset{y}\to\tau$,
$$
\xymatrix@C=50pt{\sigma\tau\ar@/^0.8pc/[r]^{(x+{}^{\sigma}\!y)+(x'+{}^{\sigma}\!y')}
\ar@{}[r]|{\Downarrow}
\ar@/_0.8pc/[r]_{(x+x')+{}^{\sigma}\!(y+y')}& \sigma\tau,}
$$
is that obtained by pasting in the bigroupoid $\b$ the diagram
$$
\xymatrix{\sigma\tau\ar[r]^{^{\sigma}\!y'}\ar[dr]_{^\sigma\!(y+y')}\ar@{}@<10pt>[rd]|(.55){\Leftarrow \chi}&
\sigma\tau\ar[d]
\ar@{}@<-2pt>[d]^{^{\sigma\!}\!y} \ar@{}@<3pt>[dr]|(.4){\Downarrow\, \boldsymbol{c}}
\ar[r]^{x'}&\sigma\tau\ar[r]^{^\sigma\!y}&\sigma\tau\ar[d]^{x}
\\
& \sigma\tau\ar[rr]_{x+x'}\ar[rru]^(.45){x'}&&\sigma\tau\ar@{}[lu]^{\Downarrow\bar{\chi}}}
$$
where

$ \chi=-t(0,0,0,{}^{\sigma}\!y,{}^{\sigma}\!y',0,\sigma,\tau,1,1)$,

$\boldsymbol{c}=t(0,x,0,0,{}^{\sigma}\!y,0,\sigma,\tau,1,1)-t(0,0,x,{}^{\sigma}\!y,0,0,\sigma,\tau,1,1)-t(x,0,0,0,0,{}^{\sigma}\!y,\sigma,1,\tau,1)$,

$\bar{\chi}=-t(0,x,x',0,0,0,\sigma,\tau,1,1)+t(x,0,x',0,0,0,\sigma,1,\tau,1)-t(x,x',0,0,0,0,\sigma,1,1,\tau)$.

\vspace{0.2cm} The associativity pseudo-equivalence
$(-\otimes -)\otimes -\overset{\boldsymbol{a}}\Rightarrow -\otimes(-\otimes
-):\b^3\to\b$ is defined by the 1-cells
$$h(\sigma,\tau,\gamma):(\sigma\tau)\gamma\to\sigma(\tau\gamma).$$
The naturality component of $\boldsymbol{a}$, at any 1-cells
$\sigma\overset{x}\to\sigma$,  $\tau\overset{y}\to\tau$ and
$\gamma\overset{z}\to\gamma$,
$$
\xymatrix@C=40pt@R=20pt{(\sigma\tau)\gamma\ar@{}[rd]|{\Rightarrow}\ar[r]^{h=h(\sigma,\tau,\gamma)}\ar[d]_{(x+{}^{\sigma}\!y)+{}^{\sigma\tau}\!z}&
\sigma(\tau\gamma)\ar[d]^{x+{}^{\sigma}\!(y+{}^{\tau}\!z)}\\
(\sigma\tau)\gamma\ar[r]_{h(\sigma,\tau,\gamma)}&\sigma(\tau\gamma)
}
$$
is given by pasting in $\b$ the diagram
$$
\xymatrix{(\sigma\tau)\gamma\ar[d]_{x}\ar[r]^{h}\ar@{}[rd]|(.4){\Downarrow\Omega}&\sigma(\tau\gamma)\ar[r]^{x} \ar@{}[rd]|{\Downarrow\Psi}&\sigma(\tau\gamma)\ar[r]^{{}^{\sigma}\!y}\ar@{}[rd]|(.6){\Downarrow\Phi}&\sigma(\tau\gamma)\ar[d]^{{}^{\sigma\tau}\!z}\\
(\sigma\tau)\gamma\ar[rru]^{h}\ar[r]_{{}^{\sigma}\!y}&(\sigma\tau)\gamma\ar[rru]^{h}\ar[r]_{{}^{\sigma\tau}\!z}&(\sigma\tau)\gamma\ar[r]_{h}&\sigma(\tau\gamma)
}
$$
where

\noindent$\begin{array}{cl}\Phi=&\hspace{-8pt}t(0,h,0,0,{}^{\sigma\tau}\!z,0,\sigma,\tau\gamma,1,1)-t(h,0,0,0,0,{}^{\sigma\tau}\!z,\sigma,\tau,\gamma,1)-
t(0,0,h,{}^{\sigma\tau}\!z,0,0,\sigma,\tau\gamma,1,1),\end{array}$

\noindent$\begin{array}{cl}\Psi=&\hspace{-8pt}t(h,0,0,{}^{\sigma}\!y,0,0, \sigma,\tau,1, \gamma)-t(h,0,0,0,{}^{\sigma}\!y,0,\sigma,\tau,\gamma,1)+t(0,h,0,0,{}^{\sigma}\!y,0,\sigma,\tau\gamma,1,1)\\
&\hspace{-7pt}-t(0,0,h,{}^{\sigma}\!y,0,0,\sigma,\tau\gamma,1,1),\end{array}$

\noindent$\begin{array}{cl}\Omega=&\hspace{-8pt}-t(x,h,0,0,0,0,\sigma,1,\tau,\gamma)+t(h,x,0,0,0,0,\sigma,\tau,1,\gamma)-t(h,0,x,0,0,0,\sigma,\tau,\gamma,1)\\
&\hspace{-7pt}+t(x,0,h,0,0,0,\sigma,1,\tau\gamma,1)+t(0,h,x,0,0,0,\sigma,\tau\gamma,1,1)-t(0,x,h,0,0,0,\sigma,\tau\gamma,1,1),\end{array}$

\vspace{0.2cm} The structure modification $\pi$, at any objects
$\sigma,\tau,\gamma,\delta\in G$, is
$$
\xymatrix@C=55pt{((\sigma\tau)\gamma)\delta\ar@{}[rrd]|{ \Rightarrow}\ar@{}@<4pt>[rrd]|(.48){\pi }\ar[d]_{h_4=h(\sigma,\tau,\gamma)}\ar[rr]^{h_1=h(\sigma\tau,\gamma,\delta)}&&(\sigma\tau)(\gamma\delta)
\ar[d]^{h_3=h(\sigma,\tau,\gamma\delta)}\\ (\sigma(\tau\gamma))\delta\ar[r]^{h_2=h(\sigma,\tau\gamma,\delta)}&\sigma((\tau\gamma)\delta)
\ar[r]^{h_0={}^{\sigma}\!h(\tau,\gamma,\delta)}&\sigma(\tau(\gamma\delta)),}
$$
\noindent$\begin{array}{cl}\pi=&\hspace{-5pt}t(h_3,h_1-h_0,0,h_0,0,0,\sigma,\tau,\gamma,\delta)
-t(h_2,h_4,0,0,0,0,\sigma,\tau\gamma,1,\delta)\\&\hspace{-7pt}+
t(h_2,0,h_4,0,0,0,\sigma,\tau\gamma,\delta,1)
-t(h_3,0,h_1-h_0,0,h_0,0,\sigma,\tau,\gamma\delta,1)\\
&\hspace{-7pt}+
t(0,h_3,h_1-h_0,0,h_0,0,\sigma,\tau\gamma\delta,1,1)-
t(0,0,h_2+h_4,h_0,0,0,\sigma,\tau\gamma\delta,1,1)\\
&\hspace{-7pt}-t(0,h_2,h_4,0,0,0,\sigma,\tau\gamma\delta,1,1).\end{array}$

This completes the description of bicategorical group
$(\b,\otimes)$, whose geometric nerve is recognized to be isomorphic
to the minimal complex $M$ in (\ref{fff}) by means of the simplicial
map $\varphi:\gner(\b,\otimes)\to M$ which, in dimensions $\leq 4$,
$$ \xymatrix@C=15pt@R=20pt{\Delta_4(\b,\otimes) \ar@<4pt>[r]\ar@{}@<2pt>[r]
\ar@{}[r]|{\cdots}\ar@<-4pt>[r]\ar[d]^{\varphi}&\Delta_3(\b,\otimes)
\ar@<4pt>[r]\ar@{}[r]|{\cdots}\ar@<-3pt>[r]\ar[d]^{\varphi}&
\Delta_2(\b,\otimes)  \ar@<3pt>[r]\ar[r]
 \ar@<-3pt>[r]\ar[d]^{\varphi}&\Delta_1(\b,\otimes)
\ar@<2pt>[r]
\ar@<-2pt>[r]\ar[d]^{\varphi}& \ar[d]1,\\
B^4\!\times\!A^6\!\times G^4 \ar@<4pt>[r]\ar@{}@<2pt>[r]
\ar@{}[r]|{\cdots}\ar@<-4pt>[r]&B\!\times\!A^3\!\times G^3
\ar@<4pt>[r]\ar@{}[r]|{\cdots}\ar@<-3pt>[r]&
A\!\times G^2 \ar@<3pt>[r]\ar[r]
 \ar@<-3pt>[r]&G
\ar@<2pt>[r]
\ar@<-2pt>[r]& 1,
}
$$
 is defined as follows (keeping the notations stated in $(\ref{dnt1})-(\ref{dnt5}))$: It
 carries a (unitary) representation $F:[1]\to \Sigma(\b,\otimes)$ to $\varphi(F)=F_{01}$, a representation
  $F:[2]\to \Sigma(\b,\otimes)$
  to $\varphi(F)=(-F_{012},F_{01},F_{12})$, a representation $F:[3]\to \Sigma(\b,\otimes)$ to
$$\varphi(F)= (-\!F_{0123},-\!{}^{F_{01}}\!F_{123}\!+\!F_{023}\!-\!F_{013},{}^{F_{01}}\!F_{123}\!-\!F_{023}
  ,-{}^{F_{01}}\!F_{123},F_{01},F_{12},F_{23}),$$
and a representation $F:[4]\to \Sigma(\b,\otimes)$ to
$$\varphi(F)=(u_1,u_2,u_3,u_4,
x_1,x_2,x_3,x_4,x_5,x_6,F_{01},F_{12},F_{23},F_{34}),$$ where

$
\begin{array}{ll}
u_1={}^{F_{01}}\!F_{1234}\!-\!F_{0124}+F_{0134}\!-\!F_{0234},&x_2=
{}^{F_{01}}\!F_{124}\!-\!{}^{F_{01}}\!F_{134}\!+\!F_{034}\!-\!F_{024},\\
u_2=F_{0234}\!-\!F_{0134}\!-\!{}^{F_{01}}\!F_{1234},& x_3={}^{F_{01}}\!F_{134}\!-\!F_{034},\\
u_3={}^{F_{01}}\!F_{1234}\!-\!F_{0234},& x_4={}^{F_{01}}\!F_{134}\!-\!{}^{F_{01}}\!F_{124}\!-\!{}^{F_{02}}\!F_{234},\\
u_4=-{}^{F_{01}}\!F_{1234},& x_5={}^{F_{02}}\!F_{234}\!-\!{}^{F_{01}}\!F_{134},\\
x_1=\!F_{024}\!
-\!{}^{F_{01}}\!F_{124}\!-\!F_{014},&  x_6=-\!{}^{F_{02}}\!F_{234}.
\end{array}
$

\vspace{0.2cm} To finish, we shall remark on two  particular relevant
cases of the demostrated relationship between monoidal bicategories and
path-connected homotopy 3-types. Since {\em categorical groups}
\cite[\S 3]{joyal} are the same thing as bicategorical groups in
which all 2-cells are identities, then categorical groups
are algebraic models for path-connected homotopy 2-types. This fact
goes back to Whitehead (1949) and Mac Lane-Whitehead (1950)
since every categorical group is equivalent to a strict one, and
strict categorical groups are the same as crossed modules. On the
other hand, if $(C,\otimes,\boldsymbol{c})$ is any {\em braided
categorical group}
 \cite{joyal}, then its classifying
space $\class(C,\otimes,\boldsymbol{c})$  is  the classifying space
of its suspension bicategorical group $(\Sigma(C,\otimes),\otimes)$
(see Examples \ref{ebmc} and \ref{nebmc}), which is precisely a
one-object bicategorical group. Therefore, we conclude from the above
discussion that braided categorical groups are algebraic models for
path-connected simply connected homotopy 3-types, a fact due to
Joyal and Tierney \cite{j-t}, but also proved in \cite{c-c} and,
implicitly, in \cite{joyal}.

\end{document}